\documentclass[preprint,12pt]{elsarticle}

\usepackage{amssymb}
\usepackage{mathrsfs}
\usepackage{amsmath}
\usepackage{stmaryrd}

\biboptions{compress}
\newdefinition{theorem}{Theorem}[section]
\newdefinition{lemma}[theorem]{Lemma}
\newdefinition{corollary}[theorem]{Corollary}
\newdefinition{claim}[theorem]{Claim}

\newtheorem{definition}[theorem]{Definition}
\newdefinition{remark}[theorem]{Remark}
\newdefinition{example}[theorem]{Example}
\newdefinition{proposition}[theorem]{Proposition}
\newproof{proof}{Proof}

\journal{ }

\begin{document}

\begin{frontmatter}



\title{On random convex analysis--the analytic foundation of the module approach to conditional risk measures\tnoteref{t1}}

\tnotetext[t1]{Supported by NNSF No. 11171015}


\author[gt]{Tiexin Guo\corref{cor1}}
\ead{tiexinguo@csu.edu.cn}
\cortext[cor1]{Corresponding author}
\address[gt]{School of Mathematics and Statistics, Central South University, Changsha 410075, Hunan Province, P.R. China}

\author[gs b]{Shien Zhao}
\address[gs b]{Elementary Educational College, Capital Normal University, Beijing 100048, P.R. China}

\author[gs]{Xiaolin Zeng}
\address[gs]{College of Mathematics and Statistics, Chongqing Technology and Business University, Chongqing 400067, P.R. China}

\begin{abstract}
To provide a solid analytic foundation for the module approach to conditional risk measures, this paper establishes a complete random convex analysis over random locally convex modules by simultaneously considering the two kinds of topologies (namely the $(\varepsilon,\lambda)$--topology and the locally $L^0$-- convex topology). Then, we make use of the advantage of the $(\varepsilon,\lambda)$--topology and grasp the local property of $L^0$--convex conditional risk measures to prove that every $L^{0}$--convex $L^{p}$--conditional risk measure ($1\leq p\leq+\infty$) can be uniquely extended to an $L^{0}$--convex $L^{p}_{\mathcal{F}}(\mathcal{E})$--conditional risk measure and that the dual representation theorem of the former can also be regarded as a special case of that of the latter, which shows that the study of $L^p$--conditional risk measures can be incorporated into that of $L^{p}_{\mathcal{F}}(\mathcal{E})$--conditional risk measures. In particular, in the process we find that combining the countable concatenation hull of a set and the local property of conditional risk measures is a very useful analytic skill that may considerably simplify and improve the study of $L^{0}$--convex conditional risk measures.
\end{abstract}

\begin{keyword}
Random locally convex module;
$(\varepsilon,\lambda)$--topology; Locally $L^{0}$--convex topology;
Random conjugate spaces; Random convex analysis; $L^{0}$--convex lower semicontinuous proper function; $L^{0}$--convex conditional risk measure



\end{keyword}

\end{frontmatter}



\noindent{\bf Contents} \\
\normalfont{
\begin{trivlist}
\item[1.] Introduction
\item[2.] Preliminaries
 \begin{enumerate}
 \item[2.1.] Some basic notions
 \item[2.2.] On some precise connections between the $(\varepsilon,\lambda)$--topology and locally $L^{0}$--convex topology for a random locally convex module
 \item[2.3.] On the precise relation between random conjugate spaces under the two kinds of topologies
 \item[2.4.] Bounded sets under the locally $L^{0}$--convex topology
 \end{enumerate}
\item[3.] Random convex analysis over random locally convex modules under the

\noindent ~~~locally $L^{0}$--convex topology
 \begin{enumerate}
 \item[3.1.] Separation under the locally $L^{0}$--convex topology
 \item[3.2.] Fenchel-Moreau dual representation theorem under the locally $L^{0}$--convex topology
 \item[3.3.] Random duality under the locally $L^{0}$--convex topology with respect to a random duality pair
  \begin{enumerate}
  \item[3.3.1] Random compatible locally $L^{0}$--convex topology
  \item[3.3.2] Random admissible locally $L^{0}$--convex topology
  \item[3.3.3] A characterization for a random locally convex module to be $L^{0}$--pre-barreled
  \end{enumerate}
 \item[3.4.] Continuity and subdifferentiability theorems in $L^{0}$--pre-barreled modules
 \end{enumerate}
\item[4.] Random convex analysis over random locally convex modules under the

\noindent ~~~$(\varepsilon,\lambda)$--topology
 \begin{enumerate}
 \item[4.1.] Lower semicontinuous $L^{0}$--convex functions under the $(\varepsilon,\lambda)$--topology
 \item[4.2.] Fenchel-Moreau dual representation theorem under the $(\varepsilon,\lambda)$--topology
 \item[4.3.] Extension theorem for $L^{\infty}$--conditional risk measures
 \item[4.4.] Extension theorem for $L^{p}$--conditional risk measures
 \end{enumerate}
\item[Acknowledgments]
\item[References]
\end{trivlist}
} \baselineskip 15pt
\renewcommand{\baselinestretch}{1.02}
\parindent=16pt  \parskip=2mm
\rm\normalsize\rm
\newpage

\section{Introduction}
Random metric theory is based on the idea of randomizing the classical space theory of functional analysis. This idea may date back to K. Menger, B. Schweizer and A. Sklar's idea of the theory of probabilistic metric spaces, cf. \cite{SS}. Random normed modules (briefly, $RN$ modules), random inner product modules (briefly, $RIP$ modules) and random locally convex modules (briefly, $RLC$ modules) have been the basic framework in random metric theory, and the theory of random conjugate spaces has been a powerful tool for the deep development of the basic framework. Random metric theory may now also be aptly called random functional analysis since it is being developed as functional analysis over the basic random framework. Motivated by the original notions of random metric spaces and random normed spaces introduced in \cite{SS}, all the basic notions such as $RN$, $RIP$ and $RLC$ modules together with their random conjugate spaces were naturally presented by Guo in the course of the development of random functional analysis, cf. \cite{TXG-Master,TXG-PHD,TXG-Extension,TXG-Module,TXG-Radon,TXG-basic,TXG-Sur,TXG-JFA}. Recently, we are pleased to learn that several other authors also ever independently presented the notions of $RN$ modules and random conjugate spaces. For example, as a tool for the study of ultrapowers of Lebesgue-Bochner function spaces, R.Haydon, M.Levy and Y.Raynaud introduced real $RN$ modules (called randomly normed $L^{0}$--modules in \cite{HLR}) and their random conjugate spaces, cf. \cite{HLR}. Motivated by financial applications, D. Filipovi\'{c}, M. Kupper and N. Vogelpoth also presented real $RN$ modules (called $L^{0}$--normed modules in \cite{FKV}), and in particular the notions of locally $L^{0}$--convex modules and the locally $L^{0}$--convex topology for random locally convex modules, cf. \cite{FKV}.

Before 2009, all the theory of $RN$ modules and $RLC$ modules together with their random conjugate spaces was developed under the $(\varepsilon,\lambda)$--topology. The $(\varepsilon,\lambda)$--topology is inherited from B.Schweizer and A. Sklar's work in 1960, where they first introduced such a kind of topology for more abstract probabilistic metric spaces. The $(\varepsilon,\lambda)$--topology is an abstract generalization of the topology of convergence in measure on the linear space of random variables, and hence the essence of the $(\varepsilon,\lambda)$--topology is locally nonconvex in general. Therefore, from a viewpoint of topological modules, all the pre-2009 work on random metric theory can be regarded a locally nonconvex generalization of classical results established for normed spaces and locally convex spaces, in the process the theory of random conjugate spaces has played an essential role and the order structure and module structure peculiar to $RN$ and $RLC$ modules not only can be fully available for the theory of random conjugate spaces but also lead to the difficulties in the study of complicated stratification structure.

In 2009, in \cite{FKV} D.Filipovi$\acute{c}$, M.Kupper and N.Vogelpoth presented the notion of a locally $L^{0}$--convex module, in company with which the notion of a locally $L^{0}$--convex topology for an $RN$ module and $RLC$ module was also introduced. This means that the $L^{0}$--norm on an $RN$ module or the family of $L^{0}$--seminorms on an $RLC$ module can also induce another kind of topology, namely the locally $L^{0}$--convex topology. Subsequently, the principal connections between some basic results derived from the two kinds of topologies for an $RN$ module or $RLC$ module were given in \cite{TXG-JFA}, based on which, lots of new and basic researches recently have been done in \cite{TXG-GS,TXG-YJY,TXG-XLZ,TXG-XZ,TXG-SEZ,MZW,SEZ-TXG}. The recent study further exposes the respective advantages and disadvantages of the two kinds of topologies. Although the locally $L^{0}$--convex topology is very similar to the classical locally convex topology, the study of the locally $L^{0}$--convex topology often requires the involved $L^{0}$--modules or their subsets to have the countable concatenation property, namely the stratification structure of a locally $L^{0}$--convex module has a remarkable effect on its topological structure, which makes the theory of the locally $L^{0}$--convex topology considerably different from that of the classical locally convex topology. Therefore, we have been hoping that combining the respective advantages of the two kinds of topologies produces a perfect random convex analysis over random locally convex modules for better financial applications, since D. Filipovi\'{c}, et.al's paper \cite{FKV} only consider the corresponding problems under the locally $L^{0}$--convex topology. The central purpose of this paper is to achieve such a goal.

Classical convex analysis (e.g, see \cite{ET}) is the analytic foundation for convex risk measures, cf. \cite{ADEH,Delbaen,FS,Follmer-S,Fritt-R}. However, it is no longer universally applicable to $L^{0}$--convex (or conditional convex) conditional risk measures (in particular, those defined on the model spaces of unbounded financial positions). Just to overcome the obstacle, D. Filipovi\'{c}, et.al presented a good idea of randomizing the initial data, which leads to an attempt to establish random convex analysis over locally $L^{0}$--convex modules in order to provide a new approach to conditional risk, called the module approach, cf, \cite{FKV,FKV-appro}. Although the notions of locally $L^{0}$--convex modules and in particular the locally $L^{0}$--convex topology are very important, we recently find out that many foundational problems induced by their paper \cite{FKV} are worth further studying and perfecting, which also makes our work in this paper urgent.

First, the premise for most of the principal results in \cite{FKV} is that the locally $L^{0}$--convex topology for every locally $L^{0}$--convex module can be induced by a family of $L^{0}$--seminorms. Although they ever gave a proof of the premise (namely, Theorem 2.4 of \cite{FKV}), there was a hole in the proof of \cite{FKV}, even according to some remarks in this paper we think that this premise may be not valid in general. In this paper we choose random locally convex modules as the framework on which random convex analysis is based so that we can avoid the theoretically difficult point involved in \cite{FKV} since the definition of a random locally convex module assumes the existence of a family of $L^{0}$--seminorms in advance and $L^{0}$--seminorms are often easily constructed in both the theoretic study and financial applications.

Second, Lemma 2.28 of \cite{FKV} played a crucial role in the proofs of the main results-Proposition 3.4, Theorems 2.8 and 3.8 of \cite{FKV}. However, in this paper we provide a counterexample to both Lemma 2.28 and Theorem 2.8 of \cite{FKV}, thus the two results and the related results should be modified, which is done in Subsections 3.1 and 3.2 of this paper, and in particular the refined version of the modified results can also be obtained since the precise relation between random conjugate spaces of a random locally convex module under the two kinds of topologies has been found in this paper, in particular the refined results will also be needed in the sequel of this paper.

Third, although D. Filipovi\'{c}, M. Kupper and N. vogepoth had presented the notion of $L^{0}$--barreled modules and established the continuity and subdifferentiability theorems for $L^{0}$--convex functions defined on $L^{0}$--barreled modules, namely Proposition 3.5 and Theorem 3.7 of \cite{FKV}, the two results can not be applied to the study of conditional risk measures since it remains open up to now whether the model space $L^{p}_{\mathcal{F}}(\mathcal{E})$ playing a crucial role in the module approach to conditional risk is an $L^{0}$--barreled module. In this paper, we overcome the difficulty by presenting the notion of an $L^{0}$--pre-barreled module. The notion of an $L^{0}$--pre-barreled module is weaker than that of an $L^{0}$--barreled module and meets the needs of financial applications. To prove this, we establish random duality theory of a random duality pair under random locally convex modules endowed with the locally $L^{0}$--convex topology so that we can give a characterization for random locally convex modules to be $L^{0}$--pre-barreled modules, in particular $L^{p}_{\mathcal{F}}(\mathcal{E})$ is $L^{0}$--pre-barreled when it is endowed with the locally $L^{0}$--convex topology. Further, we also establish the new continuity and subdifferentiability theorems based on the notion of an $L^{0}$--pre-barreled module. All these results are given in Subsections 3.3 and 3.4.

Concerning random convex analysis over random locally convex modules under the $(\varepsilon,\lambda)$--topology, although it is impossible to establish the corresponding continuity and subdifferentiability theorems under the $(\varepsilon,\lambda)$--topology since the $(\varepsilon,\lambda)$--topology is locally nonconvex in general, we can give a natural Fenchel-Moreau dual representation theorem, which has the same shape as the classical Fenchel-Moreau dual representation theorem. Since the $(\varepsilon,\lambda)$--topology is in harmony with the norm topology, for example, let $(\Omega,\mathcal{E},P)$ be a probability space, $\mathcal{F}$ a $\sigma$--subalgebra of $\mathcal{E}$ and $L^{p}_{\mathcal{F}}(\mathcal{E})$ the $RN$ module generated by the Banach space $L^{p}(\mathcal{E})$ $(1\leq p\leq+\infty)$, then $L^{p}(\mathcal{E})$ is dense in $L^{p}_{\mathcal{F}}(\mathcal{E})$ with respect to the $(\varepsilon,\lambda)$--topology on $L^{p}_{\mathcal{F}}(\mathcal{E})$ (clearly, this is not true with respect to the locally $L^{0}$--convex topology!). The simple fact motivates us to futher study the precise relations among the three kinds of conditional risk measures. The first kind was introduced independently by K.Detlefsen and G.Scandolo in \cite{DS} and J.Bion-Nadal in \cite{Nadal} as a monotone and cash-invariant function from $L^{\infty}(\mathcal{E})$ to $L^{\infty}(\mathcal{F})$ (briefly, an $L^{\infty}$--conditional risk measure). The second and third kinds were introduced by D. Filipovi\'{c}, M. Kupper and N. Vogelpoth in \cite{FKV-appro} as monotone and cash-invariant functions from $L^{p}(\mathcal{E})$ to $L^{r}(\mathcal{F})$ $(1\leq r\leq p<+\infty)$ and from $L^{p}_{\mathcal{F}}(\mathcal{E})$ to $\bar{L}^{0}(\mathcal{F})$ $(1\leq p\leq+\infty)$, respectively, (briefly, $L^{p}$-- and $L^{p}_{\mathcal{F}}(\mathcal{E})$--conditional risk measures, respectively). We show that an $L^{\infty}$--conditional risk measure can be uniquely extended to an $L^{\infty}_{\mathcal{F}}(\mathcal{E})$--conditional risk measure and the conditional convex dual representation theorem established in \cite{Nadal,DS} for the former can be regarded as a special case of that established in this paper for the latter. We further show that an $L^{0}$--convex $L^{p}$--conditional risk measure can be uniquely extended to an $L^{0}$--convex $L^{p}_{\mathcal{F}}(\mathcal{E})$--conditional risk measure $(1\leq p<+\infty)$ and the conditional convex dual representation theorem established in \cite{FKV-appro} for the former can also be regarded a special case of that established in \cite{FKV-appro} for the latter. The second extension theorem is not very easy, whose proof is constructive, since an $L^{0}$--convex $L^{p}$--conditional risk measure is not necessarily uniformly continuous with respect to the relative topology when $L^{p}(\mathcal{E})$ is regarded as a subspace of $L^{p}_{\mathcal{F}}(\mathcal{E})$ which is endowed with the $(\varepsilon,\lambda)$--topology, in particular, in the process we find that combining the countable concatenation hull of a set and the local property of conditional risk measures is a very useful analytic skill that may considerably simplify and improve the study of $L^{0}$--convex conditional risk measures. This shows that the two vector space approaches to conditional risk can be incorporated into the module approach. Thus this paper has provided a complete random convex analysis, and hence also a solid analytic foundation for the module approach to conditional risk.

Most of the main ideas and results of this paper were first announced in Guo's survey paper \cite{TXG-Recent} without the detailed proofs or at most with a sketch of proofs of a few of illustrative results, this paper includes many new results as well as the full proofs of the results announced in \cite{TXG-Recent}. Besides. the great distinction between this paper and \cite{TXG-Recent} is that almost all the results in Section 3 is stated under the framework of random locally convex modules endowed with the locally $L^{0}$--convex topology rather than the framework of locally $L^{0}$--convex modules for some reasons mentioned above.

The remainder of this paper is organized as follows. Section 2 provides some necessary preliminaries for sake of the reader's convenience and in particular includes some key new results on the precise relations between the two kinds of closures of an $L^{0}$--convex set and between the two kinds of random conjugate spaces of a random locally convex module under the $(\varepsilon,\lambda)$--topology and locally $L^{0}$--convex topology; Sections 3 and 4 present and prove our main results as stated above in the Introduction of this paper.

Throughout this paper, we always use the following notation and terminology:

$K:$ the scalar field R of real numbers or C of complex numbers.

$(\Omega,\mathcal{F},P):$ a probability space.

$L^{0}(\mathcal{F},K)=$ the algebra of equivalence classes of $K$--valued $\mathcal{F}$-- measurable random variables
on $(\Omega,\mathcal{F},P)$.

$L^{0}(\mathcal{F})=L^{0}(\mathcal{F},R)$.

$\bar{L}^{0}(\mathcal{F})=$ the set of equivalence classes of extended real-valued $\mathcal{F}$-- measurable random
variables on $(\Omega,\mathcal{F},P)$.

As usual, $\bar{L}^{0}(\mathcal{F})$ is partially ordered by $\xi\leq\eta$ iff $\xi^{0}(\omega)\leq\eta^{0}(\omega)$ for $P$--almost all $\omega\in \Omega$ (briefly, a.s.), where $\xi^0$ and $\eta^0$ are arbitrarily chosen representatives of $\xi$ and $\eta$, respectively. Then $(\bar{L}^{0}(\mathcal{F}),\leq)$ is a complete lattice, $\bigvee H$ and $\bigwedge H$ denote the supremum and infimum of a subset $H$, respectively. $(L^{0}(\mathcal{F}),\leq)$ is a conditionally complete lattice. Please refer to \cite{D-Sch} or \cite[p. 3026]{TXG-JFA} for the rich properties of the supremum and infimum of a set in $\bar{L}^{0}(\mathcal{F})$.

Let $\xi$ and $\eta$ be in $\bar{L}^{0}(\mathcal{F})$. $\xi<\eta$ is understood as usual, namely $\xi\leq\eta$ and $\xi\neq\eta$. In this paper we also use $``\xi<\eta $ (or $\xi\leq\eta$) on $A"$ for $``\xi^{0}(\omega)<\eta^{0}(\omega)$ (resp., $\xi^{0}(\omega)\leq\eta^{0}(\omega)$) for $P$--almost all $\omega\in A"$, where $A\in\mathcal{F}$, $\xi^0$ and $\eta^0$ are a representative of $\xi$ and $\eta$, respectively.

$\bar{L}^{0}_{+}(\mathcal{F})=\{\xi\in \bar{L}^{0}(\mathcal{F})~|~\xi\geq0\}$

$\L^{0}_{+}(\mathcal{F})=\{\xi\in L^{0}(\mathcal{F})~|~\xi\geq0\}$

$\bar{L}^{0}_{++}(\mathcal{F})=\{\xi\in \bar{L}^{0}(\mathcal{F})~|~\xi>0$ on $\Omega\}$

$L^{0}_{++}(\mathcal{F})=\{\xi\in L^{0}(\mathcal{F})~|~\xi>0$ on $\Omega\}$

Besides, $\tilde{I}_{A}$ always denotes the equivalence class of $I_{A}$, where $A\in \mathcal{F}$ and $I_{A}$ is the characteristic function of $A$.
When $\tilde{A}$ denotes the equivalence class of $A (\in \mathcal{F})$, namely $\tilde{A}=\{B\in\mathcal{F}~|~P(A\triangle B)=0\}$ (here, $A\triangle B=(A\setminus B)\bigcup(B\setminus A)$), we also use $I_{\tilde{A}}$ for $\tilde{I}_{A}$.

Specially, $[\xi<\eta]$ denotes the equivalence class of $\{\omega\in\Omega~|~\xi^0(\omega)<\eta^0(\omega)\}$, where $\xi^0$ and $\eta^0$ are arbitrarily chosen representatives of $\xi$ and $\eta$ in $\bar{L}^0(\mathcal{F})$, respectively, some more notations such as $[\xi=\eta]$ and $[\xi\neq\eta]$ can be similarly understood.

\section{Preliminaries}

\subsection{Some basic notions}

\begin{definition}($See$ \cite{TXG-Master,TXG-PHD,TXG-basic}.) An ordered pair $(E,\|\cdot\|)$ is called a random normed space (briefly, an $RN$ space) over $K$ with base $(\Omega,\mathcal{F},P)$ if $E$ is a linear space over $K$ and $\|\cdot\|$ is a mapping from $E$ to $L^{0}_{+}(\mathcal{F})$ such that the following are satisfied:

\noindent ($RN$--1). $\|\alpha x\|=|\alpha| \|x\|$, $\forall \alpha\in K$ and $x\in E$;

\noindent ($RN$--2). $\|x\|=0$ implies $x=\theta$ (the null element of $E$);

\noindent ($RN$--3). $\|x+y\|\leq\|x\|+\|y\|$, $\forall x,y\in E$.

\noindent Here $\|\cdot\|$ is called the random norm on $E$ and $\|x\|$ the random norm of $x\in E$ (If $\|\cdot\|$ only satisfies ($RN$--1) and ($RN$--3) above, it is called a random seminorm on $E$).

\noindent Furthermore, if, in addition, $E$ is a left module over the algebra $L^{0}(\mathcal{F},K)$ (briefly, an $L^{0}(\mathcal{F},K)$--module) such that

\noindent ($RNM$--1). $\|\xi x\|=|\xi| \|x\|$, $\forall \xi\in L^{0}(\mathcal{F},K)$ and $x\in E$.

\noindent Then $(E,\|\cdot\|)$ is called a random normed module (briefly, an $RN$ module) over $K$ with base $(\Omega,\mathcal{F},P)$, the random norm $\|\cdot\|$ with the property ($RNM$--1) is also called an $L^0$--norm on $E$ (a mapping only satisfying ($RN$--3) and ($RNM$--1) above is called an $L^0$--seminorm on E).

\end{definition}

\begin{remark} According to the original notion of an $RN$ space in \cite{SS}, $\|x\|$ is a nonnegative random variable for all $x\in E$. An $RN$ space in the sense of Definition 2.1 is almost equivalent to (in fact, slightly more general than) the original one in the sense of \cite{SS}. Definition 2.1 is not only very natural from traditional functional analysis but also easily avoids any possible ambiguities between random variables and their equivalence classes, and hence also more convenient for applications to Lebesgue-Bochner function spaces since the latter exactly consists of equivalence classes. $RN$ spaces in the sense of Definition 2.1 was essentially earlier employed in \cite{TXG-Master}. The study of random conjugate spaces (see Definition 2.3 below) of $RN$ spaces and applications of $RN$ spaces to best approximations in Lebesgue-Bochner function spaces lead Guo to the notion of an $RN$ module in \cite{TXG-PHD}. Subsequently, $RN$ modules and their random conjugate spaces were deeply developed by Guo in \cite{TXG-PHD,TXG-Extension,TXG-Module,TXG-Radon,TXG-reflesive,TXG-dual} so that Guo further presented the refined notions of $RN$ and random inner product (briefly, $RIP$) modules and compared the original notions of random metric spaces (briefly, $RM$ spaces) and $RN$ spaces with the currently used notions of $RM$ and $RN$ spaces in \cite{TXG-basic}. At almost the same time as Guo did the work \cite{TXG-Master,TXG-PHD}, $RN$ spaces and $RN$ modules were independently introduced by R. Haydon, M. Levy and Y. Raynaud in \cite{HLR}, where their notion of randomly normed $L^0(\mathcal{F})$--modules is exactly that of $RN$ modules over $R$ with base $(\Omega,\mathcal{F},P)$, in particular, they deeply studied the two classes of $RN$ modules-direct integrals and random Banach lattices (namely, random normed module equivalent of Banach lattices). Motivated by financial applications, D. Filipovi\'{c}, M. Kupper and N. Vogelpoth also independently came to the idea of $RN$ modules, their notion of $L^0$--normed  modules amounts to that of $RN$ modules over $R$ with base $(\Omega,\mathcal{F},P)$. The terminologies of ``$L^0$--norms and $L^0$--seminorms" are adopted from \cite{FKV}, previously they were still called random norms and random seminorms in Guo's work and \cite{HLR}.
\end{remark}

\begin{definition}($See$ \cite{TXG-Master,TXG-PHD,TXG-Module,TXG-basic}.) Let $(E_1,\|\cdot\|)$ and $(E_2,\|\cdot\|)$ be $RN$ spaces over $K$ with base $(\Omega,\mathcal{F},P)$. A linear operator $T$ from $E_1$ to $E_2$ is said to be a.s. bounded if there is $\xi\in L^{0}_{+}(\mathcal{F})$ such that $\|Tx\|_2\leq\xi\|x\|_1, \forall x\in E_1$. Denote by $B(E_1,E_2)$ the linear space of a.s. bounded linear operators from $E_1$ to $E_2$, define $\|\cdot\|:B(E_1,E_2)\rightarrow L^{0}_{+}(\mathcal{F})$ by $\|T\|=\bigwedge\{\xi\in L^{0}_{+}(\mathcal{F})~|~\|Tx\|_2\leq\xi\|x\|_1$ for all $x\in E_1\}$ for all $T\in B(E_1,E_2)$, then it is easy to check that $(B(E_1,E_2),\|\cdot\|)$ is also an $RN$ space over $K$ with base $(\Omega,\mathcal{F},P)$, in particular $(B(E_1,E_2),\|\cdot\|)$ is an $RN$ module if so is $E_2$. Specially, the $RN$ module $(E_1^{\ast},\|\cdot\|)$ with $E_1^{\ast}=B(E_1,L^{0}(\mathcal{F},K))$ is called the random conjugate space of $E_1$.
\end{definition}

\begin{remark} In Definition 2.3, let $A$ be a subalgebra dense in $L^0(\mathcal{F})$ with respect to the topology of convergence in measure, then it is easy to prove that every a.s. bounded linear operator $T$ between two $RN$ $A$--modules is an $A$--homomorphisms, thus in this case $B(E_1,E_2)$ is the same as in \cite{HLR}. Here, the notion of an $RN$ $A$--module was introduced in \cite{HLR}, which can obtained by replacing $L^0(\mathcal{F})$ with $A$ in the definition of an $RN$ module. Although $RN$ spaces and $RN$ $A$--modules are both more general than $RN$ modules, the history of random metric theory has proved that $RN$ modules are most important.
\end{remark}
.
\begin{definition}($See$ \cite{TXG-PHD,TXG-Module,TXG-Sur}.) An ordered pair $(E,\mathcal{P})$ is called a random locally convex space (briefly, an $RLC$ space) over $K$ with base $(\Omega,\mathcal{F},P)$ if $E$ is a linear space over $K$ and $\mathcal{P}$ a family of mappings from $E$ to $L^0_{+}(\mathcal{F})$ such that the following are satisfied:

\noindent ($RLC$--1). Every $\|\cdot\|\in \mathcal{P}$ is a random seminorm on $E$;

\noindent ($RLC$--2). $\bigvee\{\|x\|:\|\cdot\|\in\mathcal{P}\}=0$ iff $x=\theta$.

\noindent Furthermore, if, in addition, $E$ is an $L^{0}(\mathcal{F},K)$--module and each $\|\cdot\|\in \mathcal{P}$ is an $L^0$--seminorm on $E$, then $(E,\mathcal{P})$ is called a random locally convex module (briefly, an $RLC$ module) over $K$ with base $(\Omega,\mathcal{F},P)$.
\end{definition}

It is not very difficult to introduce the random conjugate space for an $RN$ space, whereas it is completely another thing to do for an $RLC$ space, at the earlier time Guo ever gave two definitions.

\begin{definition}($See$ \cite{TXG-PHD,TXG-Module,TXG-Sur}.) Let $(E,\mathcal{P})$ be an $RLC$ space over $K$ with base $(\Omega,\mathcal{F},P)$. A linear operator $f$ from $E$ to $L^0(\mathcal{F},K)$ (such an operator is also called a random linear functional on $E$) is called an a.s. bounded random linear functional of type I if there are $\xi\in L^0_+(\mathcal{F})$ and some finite subset $\mathcal{Q}$ of $\mathcal{P}$ such that $|f(x)|\leq\xi \|x\|_{\mathcal{Q}}$ for all $x\in E$, where $\|\cdot\|_{\mathcal{Q}}=\bigvee\{\|\cdot\|:\|\cdot\|\in \mathcal{Q}\}$, namely, $\|x\|_{\mathcal{Q}}=\bigvee\{\|x\|:\|\cdot\|\in \mathcal{Q}\}$ for all $x\in E$. Denote by $E^{\ast}_{I}$ the $L^0(\mathcal{F},K)$--module of a.s. bounded random linear functionals on $E$ of type I, called the first kind of random conjugate space of $(E,\mathcal{P})$, cf. \cite{TXG-PHD,TXG-Module}. A random linear functional $f$ on $E$ is called an a.s. bounded random linear functional of type II if there are $\xi\in L^0_+(\mathcal{F})$ and $\|\cdot\|\in \mathcal{P}_{cc}$ such that $|f(x)|\leq\xi\|x\|$ for all $x\in E$, where $\mathcal{P}_{cc}=\{\sum_{n=1}^{\infty}\tilde{I}_{A_n}\|\cdot\|_{\mathcal{Q}_n}~|~\{A_n,n\in N\}$ is a countable partition of $\Omega$ to $\mathcal{F}$ and $\{\mathcal{Q}_n,n\in N\}$ a sequence of finite subsets of $\mathcal{P}\}$. Denote by $E^{\ast}_{II}$ the $L^0(\mathcal{F},K)$--module of a.s. bounded random linear functionals on $E$ of type II, called the second kind of random conjugate space of $(E,\mathcal{P})$, cf. \cite{TXG-Sur}.
\end{definition}

In the sequel of this paper, given a random locally convex space $(E,\mathcal{P})$, $\mathcal{P}_f$ always denotes the family of finite subsets of $\mathcal{P}$, $\|\cdot\|_{\mathcal{Q}}$ is the same as in Definition 2.6 for each $\mathcal{Q}\in \mathcal{P}_f$ and $\mathcal{P}_{cc}$ also the same as in Definition 2.6. $\mathcal{P}_{cc}$ is called the countable concatenation hull of $\mathcal{P}$.

Following are the three important examples in random metric theory.

\begin{example} Let $L^0(\mathcal{F},B)$ be the $L^0(\mathcal{F},K)$--module of equivalence classes of $\mathcal{F}$--random variables (or, strongly $\mathcal{F}$--measurable functions) from $(\Omega,\mathcal{F},P)$ to a normed space $(B,\|\cdot\|)$ over $K$. $\|\cdot\|$ induces an $L^0$--norm (still denoted by $\|\cdot\|$) on $L^0(\mathcal{F},B)$ by $\|x\|:=$ the equivalence class of $\|x^0(\cdot)\|$ for all $x\in L^0(\mathcal{F},B)$, where $x^0(\cdot)$ is a representative of $x$. Then $(L^0(\mathcal{F},B),\|\cdot\|)$ is an $RN$ module over $K$ with base $(\Omega,\mathcal{F},P)$. Specially, $L^{0}(\mathcal{F},K)$ is an $RN$ module, the $L^0$--norm $\|\cdot\|$ on $L^{0}(\mathcal{F},K)$ is still denoted by $|\cdot|$.
\end{example}

$L^0(\mathcal{F},B)$ was deeply studied by Guo in \cite{TXG-PHD,TXG-Module,TXG-Radon,TXG-SBL,TXG-JAT,ZYY-TXG}. When the norm $\|\cdot\|$ on $B$ is not fixed, let $\{\|\cdot\|_\omega,\omega\in\Omega\}$ be an $\mathcal{F}$--measurable family of norms (namely $\|b\|_\omega$ is an $\mathcal{F}$--random variable as a function of $\omega\in\Omega$ for each fixed $b\in B$) and $B_\omega=$ the completion of $(B,\|\cdot\|_\omega)$ for each $\omega\in\Omega$. R. Haydon, et.al, introduced the notion of a generalized strongly $\mathcal{F}$--measurable function in \cite{HLR}, where an element $f$ in the product space $\Pi_{\omega\in\Omega}B_{\omega}$ was called a generalized strongly $\mathcal{F}$--measurable function on $\Omega$ if there is a sequence $\{f_n,n\in N\}$ of $B$--valued $\mathcal{F}$--simple functions on $\Omega$ such that $\|f(\omega)-f_n(\omega)\|_\omega\rightarrow 0 (n\rightarrow\infty)$ for each $\omega\in\Omega$. Denote by $\int^{\oplus}_{\Omega}B_\omega P(d\omega)$ the $L^0(\mathcal{F},K)$--module of equivalence classes of generalized strongly $\mathcal{F}$--measurable functions on $\Omega$, define the $L^0$--norm $\|\cdot\|$ by $\|x\|=$ the equivalence class of the mapping sending each $\omega$ to $\|x^0(\omega)\|_\omega$, where $x^0$ is a representative of $x\in\int^{\oplus}_{\Omega}B_\omega P(d\omega)$, then $(\int^{\oplus}_{\Omega}B_\omega P(d\omega),\|\cdot\|)$ is an $RN$ module over $K$ with base $(\Omega,\mathcal{F},P)$, called the direct integral of $\{B_\omega,\omega\in\Omega\}$, which played a key role in \cite{HLR}.

D. Filipovi$\acute{c}$, M. Kupper and N. Vogelpoth constructed important $RN$ modules $L^p_{\mathcal{F}}(\mathcal{E}) (1\leq p\leq+\infty)$ in \cite{FKV}, we will prove that they play the role of universal model spaces for $L^0$--convex conditional risk measures.

\begin{example} Let $(\Omega, {\mathcal E}, P)$ be a probability space and ${\mathcal F}$
a $\sigma$--subalgebra of ${\mathcal E}$. Define $|||\cdot|||_p\colon L^0({\mathcal E})\to {\bar L}^0_+({\mathcal F})$ by
$$|||x|||_p=\left\{
               \begin{array}{ll}
                 E[|x|^p|{\mathcal F}]^{1\over p}, & \hbox{when $1\leq p<\infty$;} \\
                 \bigwedge\{\xi\in {\bar L}^0_+({\mathcal F})~|~|x|\leq\xi\}, & \hbox{when $p=+\infty$;}
               \end{array}
             \right.
$$
for all $x\in L^0({\mathcal E})$.

Denote $L^p_{\mathcal F}({\mathcal E})=\{x\in L^0({\mathcal E})~|~|||x|||_p\in L^0_+({\mathcal F})\}$, then $(L^p_{\mathcal F}({\mathcal E}), |||\cdot|||_p)$ is an $RN$ module over $R$ with base $(\Omega, {\mathcal F}, P)$ and $L^p_{\mathcal F}({\mathcal E})=L^0({\mathcal F})\cdot L^p({\mathcal E})=\{~\xi x~|~\xi\in L^0({\mathcal F})~ \hbox{and}~ x\in L^p({\mathcal E})\}$.

\end{example}

To put some important classes of stochastic processes into the framework of $RN$ modules, Guo constructed a more general $RN$ module
$L^p_{\mathcal F}(S)$ in \cite{TXG-JFA} for each $p\in [1, +\infty]$, one can imagine that $S$ is an $RN$ module generated by a class of stochastic processes, $L^p_{\mathcal F}(S)$ can be constructed as follows.

\begin{example} Let $(S, \|\cdot\|)$ be an $RN$ module over $K$ with base $(\Omega, {\mathcal E}, P)$ and ${\mathcal F}$
a $\sigma$--subalgebra. Define $|||\cdot|||_p\colon S \to {\bar L}^0_+({\mathcal F})$ by
$$|||x|||_p=\left\{
               \begin{array}{ll}
                 E[\|x\|^p|{\mathcal F}]^{1\over p}, & \hbox{when $1\leq p<\infty$;} \\
                 \bigwedge\{\xi\in {\bar L}^0_+({\mathcal F})|~\|x\|\leq\xi\}, & \hbox{when $p=+\infty$;}
               \end{array}
             \right.
$$
for all $x\in S$.

Denote $L^p_{\mathcal F}(S)=\{x\in S~|~|||x|||_p\in L^0_+({\mathcal F})\}$, then $(L^p_{\mathcal F}(S), |||\cdot|||_p)$ is an $RN$ module over $K$ with base $(\Omega, {\mathcal F}, P)$. When $S=L^0({\mathcal E})$, $L^p_{\mathcal F}(S)$ is exactly $L^p_{\mathcal F}({\mathcal E})$.
\end{example}

\begin{definition}($See$ \cite{TXG-PHD,TXG-Module,TXG-Sur,TXG-SLP}.) Let $(E, {\mathcal P})$ be an $RLC$ space over $K$ with base $(\Omega, {\mathcal F}, P)$. For any positive numbers $\varepsilon$ and $\lambda$ with $0<\lambda<1$ and $\mathcal{Q}\in {\mathcal P}_f$, let $N_{\theta}(\mathcal{Q}, \varepsilon, \lambda)=\{x\in E~|~P\{\omega\in \Omega~|~\|x\|_\mathcal{Q}(\omega)<\varepsilon\}>1-\lambda\}$, then $\{N_{\theta}(\mathcal{Q}, \varepsilon, \lambda)~|~\mathcal{Q}\in {\mathcal P}_f, \varepsilon >0, 0<\lambda<1\}$ forms a local base at $\theta$ of some Hausdorff linear topology on $E$, called the $(\varepsilon, \lambda)$--topology induced by ${\mathcal P}$.
\end{definition}

From now on, we always denote by ${\mathcal T}_{\varepsilon, \lambda}$ the $(\varepsilon, \lambda)$--topology for every $RLC$ space if there is no possible confusion. Clearly, the $(\varepsilon, \lambda)$--topology for the special class of $RN$ modules $L^0({\mathcal F}, B)$ is exactly the ordinary topology of convergence in measure, and $(L^0({\mathcal F}, K), {\mathcal T}_{\varepsilon, \lambda})$ is a topological algebra over $K$. It is also easy to check that $(E, {\mathcal T}_{\varepsilon, \lambda})$ is a topological module over $(L^0({\mathcal F}, K), {\mathcal T}_{\varepsilon, \lambda})$ when $(E, {\mathcal P})$ is an $RLC$ module over $K$ with base $(\Omega, {\mathcal F}, P)$, namely the module multiplication operation is jointly continuous.

For an $RLC$ module $(E, {\mathcal P})$ over $K$ with base $(\Omega, {\mathcal F}, P)$, we always denote by $(E, {\mathcal P})^\ast_{\varepsilon, \lambda}$ ( or, briefly, $E^\ast_{\varepsilon, \lambda}$, whenever there is no confusion ) the $L^0({\mathcal F}, K)$--module of continuous module homomorphisms from $(E, {\mathcal T}_{\varepsilon, \lambda})$ to $(L^0({\mathcal F}, K), {\mathcal T}_{\varepsilon, \lambda})$, called the random conjugate space of $(E, {\mathcal P})$ under the  $(\varepsilon, \lambda)$--topology.

\begin{proposition}($See$ \cite{TXG-PHD,TXG-Module}.) Let $(E_1, \|\cdot\|_1)$ and $(E_2, \|\cdot\|_2)$ be two $RN$ modules over $K$ with base $(\Omega, {\mathcal F}, P)$ and $T$ a linear operator from $E_1$ to $E_2$. Then $T\in B(E_1, E_2)$ iff $T$ is a continuous module homomorphism from $(E_1, {\mathcal T}_{\varepsilon, \lambda})$ to $(E_2, {\mathcal T}_{\varepsilon, \lambda})$, in which case $\|T\|=\bigvee\{\|Tx\|_2~|~x\in E_1~\hbox{and}~ \|x\|_1\leq 1\}$.
\end{proposition}

Proposition 2.11 is very useful, Guo use it to prove that $(B(E_1, E_2), \|\cdot\|)$ is always ${\mathcal T}_{\varepsilon, \lambda}$--complete for any two $RN$ spaces $E_1$ and $E_2$ such that $E_2$ is ${\mathcal T}_{\varepsilon, \lambda}$--complete, in particular $E^\ast$ is ${\mathcal T}_{\varepsilon, \lambda}$--complete for every $RN$ space $E$, cf. \cite{TXG-PHD,TXG-Module}. It is also clear from Proposition 2.11 that $E^\ast=E^\ast_{\varepsilon, \lambda}$ for every $RN$ module $E$, cf. \cite{TXG-PHD,TXG-Extension}. Proposition 2.11 can be extended to a more general case when $E_1$ and $E_2$ are $RLC$ modules, cf. \cite{TXG-Sur,TXG-LHZ}. But, this paper only needs a special case of the general result in \cite{TXG-Sur,TXG-LHZ}, namely Proposition 2.12 below.

\begin{proposition}($See$ \cite{TXG-Sur,TXG-LHZ}.) Let $(E, {\mathcal P})$ be an $RLC$ module $(E, {\mathcal P})$ over $K$ with base $(\Omega, {\mathcal F}, P)$ and $f$ a random linear functional on $E$. Then $f\in E^\ast_{II}$ iff $f\in E^\ast_{\varepsilon, \lambda}$, namely $E^\ast_{II}=E^\ast_{\varepsilon, \lambda}$.
\end{proposition}

For any $\varepsilon \in L^0_{++}({\mathcal F})$, let $U(\varepsilon)=\{\xi\in L^0({\mathcal F}, K)~|~|\xi|\leq \varepsilon\}$. A subset $G$ of $L^0({\mathcal F}, K)$ is ${\mathcal T}_c$--open if for each fixed $x\in G$ there is some $\varepsilon \in L^0_{++}({\mathcal F})$ such that $x+U(\varepsilon)\subset G$. Denote by ${\mathcal T}_c$ the family of ${\mathcal T}_c$--open subsets of $L^0({\mathcal F}, K)$, then ${\mathcal T}_c$ is a Hausdorff topology on $L^0({\mathcal F}, K)$ such that $(L^0({\mathcal F}, K), {\mathcal T}_c)$ is a topological ring, namely the addition and multiplication operations are jointly continuous. D. Filipovi\'{c}, M. Kupper and N. Vogelpoth first observed this kind of topology and further pointed out that ${\mathcal T}_c$ is not necessarily a linear topology since the mapping $\alpha\mapsto \alpha x$ ($x$ is fixed) is no longer continuous in general. These observations led them to the study of a class of topological modules over the topological ring $(L^0({\mathcal F}, K), {\mathcal T}_c)$ in \cite{FKV}, where they only considered the case when $K=R$, in fact the complex case can also similarly introduced as follows.

\begin{definition}($See$ \cite{FKV}.) An ordered pair $(E, {\mathcal T})$ is a topological $L^0({\mathcal F}, K)$--module if both $(E, {\mathcal T})$ is a topological space and $E$ is an $L^0({\mathcal F}, K)$--module such that $(E, {\mathcal T})$ is a topological module over the topological ring $(L^0({\mathcal F}, K), {\mathcal T}_c)$, namely the addition and module multiplication operations are jointly continuous.
\end{definition}

Denote by $(E, {\mathcal T})^\ast_c$ ( briefly, $E^\ast_c$ ) the $L^0({\mathcal F}, K)$--module of continuous module homomorphisms from $(E, {\mathcal T})$ to $(L^0({\mathcal F}, K), {\mathcal T}_c)$, called the random conjugate space of the topological $L^0({\mathcal F}, K)$--module $(E, {\mathcal T})$, which was first introduced in \cite{FKV}.

\begin{definition}($See$ \cite{TXG-Sur,TXG-XXC,FKV}.) Let $E$ be an $L^0({\mathcal F}, K)$--module and $A$ and $B$ two subsets of $E$. $A$ is said to be $L^0$--absorbed by $B$ if there is some $\xi \in L^0_{++}({\mathcal F})$ such that $\eta A\subset B$ for all $\eta\in L^0({\mathcal F}, K)$ with $|\eta|\leq \xi$. $B$ is $L^0$--absorbent if $B$ $L^0$--absorbs every element in $E$. $B$ is $L^0$--convex if $\xi x+(1-\xi)y\in B$ for all $x,\,y\in B$ and $\xi\in L^0_+({\mathcal F})$ with $0\leq \xi\leq 1$. $B$ is $L^0$--balanced if $\eta B\subset B$ for all $\eta\in L^0({\mathcal F}, K)$ with $|\eta|\leq 1$.
\end{definition}

\begin{remark} Clearly, when $B$ is $L^0$--balanced, $A$ is $L^0$--absorbed by $B$ iff there exists some $\xi \in L^0_{++}({\mathcal F})$ such that $A\subset \xi B$. Since $L^0({\mathcal F}, K)$ is an algebra over $K$, an $L^0({\mathcal F}, K)$--module is also a linear space over $K$, then it is clear that $B$ is balanced (resp., convex) if $B$ is $L^0$--balanced (resp., $L^0$--convex). But `` being $L^0$--absorbent '' and `` being absorbent '' may not imply each other.
\end{remark}

\begin{definition}($See$ \cite{FKV}.) A topological $L^0({\mathcal F}, K)$--module $(E, {\mathcal T})$ is called a locally $L^0$--convex $L^0({\mathcal F}, K)$--module ( briefly, a locally $L^0$--convex module when $K=R$ ), in which case ${\mathcal T}$ is called a locally $L^0$--convex topology on $E$, if ${\mathcal T}$ has a local base ${\mathcal B}$ at $\theta$ ( the null element in $E$ ) such that each member in ${\mathcal B}$ is $L^0$--balanced, $L^0$--absorbent and $L^0$--convex.
\end{definition}

\begin{definition}($See$ \cite{FKV}.) Let ${\mathcal P}$ be a family of $L^0$--seminorms on an $L^0({\mathcal F}, K)$--module $E$. For any $\varepsilon \in L^0_{++}({\mathcal F})$ and any $ Q\in {\mathcal P}_f$ (namely $Q$  is a finite subset of ${\mathcal P}$), let $N_{\theta}(Q, \varepsilon)=\{x\in E~|~\|x\|_Q\leq \varepsilon\}$, then $\{~N_{\theta}(Q, \varepsilon)~|~Q\in {\mathcal P}_f, ~\varepsilon \in L^0_{++}({\mathcal F})\}$ forms a local base at $\theta$ of some locally $L^0$--convex topology, called the locally $L^0$--convex topology induced by ${\mathcal P}$.
\end{definition}

\begin{corollary} Let $(E,{\mathcal P})$ be an $RLC$ module over $K$ with base $(\Omega, {\mathcal F}, P)$ and ${\mathcal T}_c$ the locally $L^0$--convex topology induced by ${\mathcal P}$. Then $(E, {\mathcal T}_c)$ is a Hausdorff locally $L^0$--convex $L^0({\mathcal F}, K)$--module.
\end{corollary}

From now on, we always denote by $\mathcal{T}_c$ the locally $L^0$--convex topology induced by $\mathcal{P}$ for every $RLC$ module $(E,\mathcal{P})$ if there is no risk of confusion.

Let $(E,{\mathcal P})^\ast_c=(E,{\mathcal T}_c)^\ast_c$ (briefly, $E^\ast_c$, if there is no risk of confusion), called the random conjugate space of a random locally convex module $(E,{\mathcal P})$ under the locally $L^0$--convex topology ${\mathcal T}_c$ induced by ${\mathcal P}$.

\begin{proposition}($See$ \cite{TXG-JFA}.) Let $(E,{\mathcal P})$ be a random locally convex module over $K$ with base $(\Omega, {\mathcal F}, P)$ and $f\colon E\to L^0({\mathcal F}, K)$ a random linear functional. Then $f\in E^\ast_I$ iff $f\in E^\ast_c$, namely $E^\ast_I=E^\ast_c$.
\end{proposition}

\begin{remark} In \cite{TXG-JFA}, it is proved that $ E^\ast_c\subset E^\ast_I$ ( see \cite[p.3032]{TXG-JFA} ). Conversely, if $f\in E^\ast_I$, namely $f$ is a random linear functional and there are some $\xi\in L^0_+({\mathcal F})$ and $Q\in {\mathcal P}_f$ such that $|f(x)|\leq \xi \|x\|_Q$ for all $x\in E$. Lemma 2.12 of \cite{TXG-JFA} shows that $f$ must be $L^0({\mathcal F}, K)$--linear. It is also clear that $f$ is continuous from $(E,{\mathcal T}_c)$ to $(L^0({\mathcal F}, K), {\mathcal T}_c)$, and hence $f\in E^\ast_c$. Thus Proposition 2.19 has been proved in \cite{TXG-JFA}.
\end{remark}

A family ${\mathcal P}$ of random seminorms on a linear space $E$ is said to be countably concatenated (or to have the countable concatenation property) if ${\mathcal P}_{cc}={\mathcal P}$, this definition appears stronger than that given in \cite{FKV} for a family of $L^0$--seminorms on an $L^0({\mathcal F}, K)$--module since ${\mathcal P}$ must be invariant under the operation of finitely many suprema once ${\mathcal P}_{cc}={\mathcal P}$. But ${\mathcal P}$ and $\{~\|\cdot\|_Q:Q\in {\mathcal P}_f\}$ always induces the same locally $L^0$--convex topology for any family ${\mathcal P}$ of $L^0$--seminorms on an $L^0({\mathcal F}, K)$--module $E$, thus the definition is essentially equivalent to that introduced in \cite{FKV}. It is also obvious that $E^\ast_I=E^\ast_{II}$ if ${\mathcal P}_{cc}={\mathcal P}$, and hence we have the following:

\begin{corollary} ($See$ \cite{TXG-JFA}.) Let $(E,{\mathcal P})$ be an $RLC$ module over $K$ with base $(\Omega, {\mathcal F}, P)$. Then $E^\ast_c=E^\ast_{\varepsilon, \lambda}$ if ${\mathcal P}$ has the countable concatenation property. Specially, $E^\ast=E^\ast_c=E^\ast_{\varepsilon, \lambda}$ for an $RN$ module $(E,\|\cdot\|)$.
\end{corollary}

For an $RLC$ module $(E, {\mathcal P})$, by definition $E^\ast_I\subset E^\ast_{II}$, so $E^\ast_c\subset E^\ast_{\varepsilon, \lambda}$ by Propositions 2.12 and 2.19, see Subsection 2.3 for the precise relation between $E^\ast_c$ and $E^\ast_{\varepsilon, \lambda}$.

In the final part of this subsection, let us return to the basic problem: whether can a locally $L^0$--convex topology on an $L^0({\mathcal F}, K)$--module $E$ be induced by a family of $L^0$--seminorms on $E$? If the answer is yes, then the theory of Hausdorff locally $L^0$--convex modules is equivalent to that of random locally convex modules endowed with the locally $L^0$--convex topology, which will be a perfect counterpart of the classical result that a Hausdorff locally convex topology can be equivalently expressed by a separating family of seminorms. It is well known that classical gauge functionals play a crucial role in the proof of the classical result. Let $U$ be a balanced, absorbent and convex subset of a locally convex space $(E, {\mathcal T})$ and $p_U$ the gauge functional of $U$, then the following relation is easily verified: $$U^o\subset \{x\in E~|~p_U(x)<1\}\subset U\subset \{x\in E~|~p_U(x)\leq 1\}, \eqno(2.1)$$
It is the relation (2.1) that is key in the proof of the above classical result.

Random gauge functional was first introduced in \cite{FKV}. Let $U$ be an $L^0$--balanced, $L^0$--absorbent and $L^0$--convex subset of an $L^0({\mathcal F}, K)$--module $E$, define $p_U\colon E\to L^0_+({\mathcal F})$ by $p_U(x)=\bigwedge\{\xi\in L^0_+({\mathcal F})~|~x\in \xi U\}$ for all $x\in E$, called the random gauge functional of $U$. Furthermore, it is also proved in \cite{FKV} that $p_U(x)=\bigwedge\{\xi\in L^0_{++}({\mathcal F})~|~x\in \xi U\}$ for all $x\in E$ and $p_U$ is an $L^0$--seminorm on $E$. If, in addition, let $(E, {\mathcal T})$ be a locally $L^0$--convex $L^0({\mathcal F}, K)$--module, D. Filipovi\'{c}, M. Kupper and N. Vogelpoth already proved in \cite{FKV} the following:

\begin{proposition}($See$ \cite{FKV}.) Let $(E, {\mathcal T})$ be a locally $L^0$--convex $L^0({\mathcal F}, K)$--module and $U$ an $L^0$--balanced, $L^0$--absorbent and $L^0$--convex subset of $E$. Then the following statements hold:\\
(i).~~$U\subset \{x\in E~|~p_U(x)\leq 1\}$;\\
(ii).~~$p_U(x)\geq 1 $ on $B$ if ${\tilde I}_Ax\notin {\tilde I}_AK$ for all $A\in {\mathcal F}$ with $P(A)>0$ and $A\subset B$, where $B\in {\mathcal F}$ satisfies $P(B)>0$;\\
(iii).~~$U^o\subset \{x\in E~|~p_U(x)<1~\hbox{on}~\Omega\}$.
\end{proposition}

The most interesting part in Proposition 2.22 is (ii). In fact, (i) is clear and (iii) can be proved as follows: Given an $x\in U^o$, there is an $L^0$--balanced, $L^0$--absorbent and $L^0$--convex neighborhood $V$ of $\theta$ such that $x+V\in U$. Since there is $\delta \in L^0_{++}({\mathcal F})$ such that $\delta x\in V$, $(1+\delta)x=x+\delta x\in x+V\subset U$, so $x\in {\frac{1}{1+\delta}}U$, then $p_U(x)\leq {\frac{1}{1+\delta}}<1~\hbox{on}~\Omega$.

Then, can (ii) of Proposition 2.22 imply that $\{x\in E~|~p_U(x)< 1~\hbox{on}~\Omega\}\subset U$? Or, we can ask: does it hold that $\{x\in E~|~p_U(x)< 1~\hbox{on}~\Omega\}\subset U$? If it can not be guaranteed that $\{x\in E~|~p_U(x)< 1~\hbox{on}~\Omega\}\subset U$, then it is still not clear whether Theorem 2.4 of \cite{FKV} is valid, namely whether a locally $L^0$--convex topology can be induced by a family of $L^0$--seminorms. Proposition 2.24 below shows that it may be not a simple problem whether $\{x\in E~|~p_U(x)< 1~\hbox{on}~\Omega\}$ is contained in $U$, from this we even conjuncture that a locally $L^0$--convex topology may not necessarily be induced by a family of $L^0$--seminorms.

Let us first recall the notion of countable concatenation property of a set or an $L^0({\mathcal F}, K)$--module. The introducing of the notion utterly results from the study of the locally $L^0$--convex topology, the reader will see that this notion is ubiquitous in the theory of the locally $L^0$--convex topology.

From now on, we always suppose that all the $L^0({\mathcal F}, K)$--modules $E$ involved in this paper have the property that for any $x,~y\in E$, if there is a countable partition $\{A_n,n\in N\}$ of $\Omega$ to ${\mathcal F}$ such that ${\tilde I}_{A_n}x={\tilde I}_{A_n}y$ for each $n\in N$ then $x=y$. Guo already pointed out in \cite{TXG-JFA} that all random locally convex modules possess this property, so the assumption is not too restrictive.

\begin{definition}($See$ \cite{TXG-JFA}.) Let $E$ be an $L^0({\mathcal F}, K)$--module. A sequence $\{x_n, n\in N\}$ in $E$ is countably concatenated in $E$ with respect to a countable partition $\{A_n,n\in N\}$ of $\Omega$ to ${\mathcal F}$ if there is $x\in E$ such that ${\tilde I}_{A_n}x={\tilde I}_{A_n}x_n$ for each $n\in N$, in which case we define $\sum^{\infty}_{n=1}{\tilde I}_{A_n}x_n$ as $x$. A subset $G$ of $E$ is said to have the countable concatenation property if each sequence $\{x_n, n\in N\}$ in $G$ is countably concatenated in $E$ with respect to an arbitrary countable partition $\{A_n,n\in N\}$ of $\Omega$ to $\mathcal{F}$ and $\sum^{\infty}_{n=1}{\tilde I}_{A_n}x_n\in G$.
\end{definition}

\begin{proposition} Let $(E, {\mathcal T})$ be a locally $L^0$--convex $L^0({\mathcal F}, K)$--module and $U$ an $L^0$--balanced, $L^0$--absorbent and $L^0$--convex subset with the countable concatenation property. Then $U^o\subset \{x\in E~|~p_U(x)<1 ~\hbox{on}~\Omega\}\subset U\subset \{x\in E~|~p_U(x)\leq 1\}$, where $U^o$ denotes the ${\mathcal T}$--interior of $U$.
\end{proposition}

\begin{proof} By Proposition 2.22, we only need to show that $\{x\in E~|~p_U(x)<1 ~\hbox{on}~\Omega\}\subset U$. Let $x_{0}$ be a point in $E$ such that $p_U(x_{0})<1~\hbox{on}~\Omega$. Since $\{\xi\in L^0_{++}({\mathcal F})~|~x_{0}\in \xi U\}$ is downward directed, there is a sequence $\{\xi_n, n\in N\}$ in $L^0_{++}({\mathcal F})$ such that it converges to $p_U(x_{0})$ in a nonincreasing way and $x_{0}\in \xi_n U$ for each $n\in N$. By the Egoroff theorem there are a countable partition $\{A_n,n\in N\}$ of $\Omega$ to ${\mathcal F}$ and a subsequence $\{\xi_{n_k},k\in N\}$ of $\{\xi_n, n\in N\}$ such that the subsequence converges to $p_U(x_{0})$ uniformly on each $A_n$. So, we can suppose that the subsequence is just $\{\xi_n, n\in N\}$ itself and each $\xi_n<1~\hbox{on}~A_n$ since $p_U(x_{0})<1 ~\hbox{on}~\Omega$. Clearly, ${\tilde I}_{A_n}x_{0}\in {\tilde I}_{A_n}\xi_n U$ for each $n\in N$, let $u_n\in U$ be such that ${\tilde I}_{A_n}x_{0}={\tilde I}_{A_n}\xi_n u_n$ for each $n\in N$ and $\xi=\sum^{\infty}_{n=1}{\tilde I}_{A_n}\xi_n$, then there is $u\in U$ such that $u=\sum^{\infty}_{n=1}{\tilde I}_{A_n}u_n$ since $U$ has the countable concatenation property. Since $x_{0}=\sum^{\infty}_{n=1}{\tilde I}_{A_n}x_{0}=\sum^{\infty}_{n=1}{\tilde I}_{A_n}\xi_n u_n=\xi u$, $x_{0}\in \xi U\subset U$ by noticing that $\xi<1~\hbox{on}~\Omega$ and $U$ is an $L^0$--convex set with $\theta\in U$.\hfill $\square$
\end{proof}

\begin{remark} We can also give another shorter proof of Proposition 2.24 as follows: If $x_{0}\notin U$, then there is a set $B\in {\mathcal F}$ with $P(B)>0$ such that ${\tilde I}_A x_{0}\notin {\tilde I}_A U$ for all $A\in {\mathcal F}$ such that $A\subset B$ and $P(A)>0$ by Theorem 3.13 of \cite{TXG-JFA} since $U$ has the countable concatenation property. Then $p_U(x_{0})\geq 1~\hbox {on}~B$ by (ii) of Proposition 2.22.
\end{remark}

The above two proofs both use the countable concatenation property of $U$. Thus we ask: Does Proposition 2.24 hold if $U$ lacks the countable concatenation property? Recently, in \cite{TXG-GS} we proved that $E$ must have the countable concatenation property if a random locally convex module $(E,{\mathcal P})$ has a ${\mathcal T}_c$--open neighborhood $U$ of $\theta$ such that $U$ has the countable concatenation property (it is also known in \cite{TXG-GS} that at this time the ${\mathcal T}_c$--interior $U^o$ of $U$ has the countable concatenation property if $U$ is a subset such that $U^o\neq\emptyset$ and $U$ has the countable concatenation property), then we naturally ask: If $E$ has the countable concatenation property, then may a locally $L^0$--convex topology ${\mathcal T}$ on $E$ have a local base ${\mathcal B}$ at $\theta$ such that each $U\in {\mathcal B}$ is an $L^0$--balanced, $L^0$--absorbent and $L^0$--convex subset with the countable concatenation property?

We conjuncture that the answers to the above two problems are both negative according to our our experience from our recent work. Thus we are forced to frequently work with the framework of random locally convex modules $(E, {\mathcal P})$ since only the framework can fully develop the power of both the family ${\mathcal P}$ of $L^0$--seminorms and the locally $L^0$--convex topology ${\mathcal T}_c$ induced by ${\mathcal P}$, which has been enough both for the theoretic development and for financial applications as far as the work of this paper is concerned.

\subsection{On some precise connections between the $(\varepsilon, \lambda)$--topology and the locally $L^0$--convex topology for a random locally convex module}

In \cite{TXG-JFA}, Guo already proved the following two results, which will be used in this paper.

\begin{proposition}($See$ \cite{TXG-JFA}.) Let $(E, {\mathcal P})$ be an $RLC$ module. Then $E$ is ${\mathcal T}_{\varepsilon, \lambda}$--complete iff both $E$ has the countable concatenation property and $E$ is ${\mathcal T}_c$--complete.
\end{proposition}

\begin{proposition}($See$ \cite{TXG-JFA}.) Let $(E, {\mathcal P})$ be an $RLC$ module and $G$ a subset of $E$ such that $G$ has the countable concatenation property. Then ${\bar G}_{\varepsilon, \lambda}={\bar G}_c$, where ${\bar G}_{\varepsilon, \lambda}$ and ${\bar G}_c$ denotes the ${\mathcal T}_{\varepsilon, \lambda}$--  and ${\mathcal T}_c$-- closures of $G$, respectively.
\end{proposition}

Theorem 2.28 below will play a key role in the proof of the random bipolar theorem in Subsection 3.3.1 of this paper.

\begin{theorem} Let $(E, {\mathcal P})$ be an $RLC$ module over $K$ with base $(\Omega, {\mathcal F}, P)$ such that $E$ has the countable concatenation property. Then the following are true:\\
\noindent (1). ${\bar G}_c={\bar G}_{\varepsilon, \lambda}$ has the countable concatenation property if so does $G$;\\
\noindent (2). If $G$ is $L^0$--convex, then ${\bar G}_{\varepsilon, \lambda}=[H_{cc}(G)]^-_{\varepsilon, \lambda}=[H_{cc}(G)]^-_c$ has the countable concatenation property, where $H_{cc}(G)$ denotes the countable concatenation hull of $G$, namely $H_{cc}(G)=\{\Sigma_{n=1}^{\infty}\tilde{I}_{A_n}x_n:\{x_n,n\in N\}$ is a sequence in $G$ and $\{A_n,n\in N\}$ is a countable partition of $\Omega$ to $\mathcal{F}\}$.
\end{theorem}

\begin{proof} (1). $\bar{G}_c=\bar{G}_{\varepsilon,\lambda}$ is by Proposition 2.27, and thus we only need to prove that $\bar{G}_{\varepsilon,\lambda}$ has the countable concatenation property.

Let $\{x_n,n\in N\}$ be a given sequence in $\bar{G}_{\varepsilon,\lambda}$ and $\{A_n,n\in N\}$ a countable partition of $\Omega$ to $\mathcal{F}$, then by the countable concatenation property of $E$ there is $x^{\ast}\in E$ such that $x^{\ast}=\Sigma_{n=1}^{\infty}\tilde{I}_{A_n}x_n$. We claim that $x^{\ast}\in \bar{G}_{\varepsilon,\lambda}$, namely, $(x^{\ast}+N_{\theta}(\mathcal{Q},\varepsilon^{\ast},\lambda^{\ast}))\bigcap G\neq \emptyset$ for any given $\varepsilon^{\ast}>0,\lambda^{\ast}>0$ with $0<\lambda^{\ast}<1$ and any finite subset $\mathcal{Q}$ of $\mathcal{P}$, where $N_{\theta}(\mathcal{Q},\varepsilon^{\ast},\lambda^{\ast})=\{x\in E~|~P\{\omega\in\Omega~|~\|x\|_{\mathcal{Q}}(\omega)<\varepsilon^{\ast}\}>1-\lambda^{\ast}\}$.

In fact, it is clear that there exists $\bar{x}_n$ for each $x_n\in\bar{G}_{\varepsilon,\lambda}$ such that $P\{\omega\in\Omega~|~\|x_n-\bar{x}_n\|_{\mathcal{Q}}(\omega)<\varepsilon^{\ast}\}>1-\frac{1}{2^{n+1}}\lambda^{\ast}$. By the countable concatenation property of $G$, there is $\bar{x}\in G$ such that $\bar{x}=\Sigma_{n=1}^{\infty}\tilde{I}_{A_n}\bar{x}_n$. Then $P\{\omega\in\Omega~|~\|x^{\ast}-\bar{x}\|_{\mathcal{Q}}(\omega)\geq\varepsilon^{\ast}\}=\Sigma_{n=1}^{\infty}P\{\omega\in A_n~|~\|x_n-\bar{x}_n\|_{\mathcal{Q}}(\omega)\geq\varepsilon^{\ast}\}\leq\Sigma_{n=1}^{\infty}P\{\omega\in \Omega~|~\|x_n-\bar{x}_n\|_{\mathcal{Q}}(\omega)\geq\varepsilon^{\ast}\}\leq\Sigma_{n=1}^{\infty}\frac{1}{2^{n+1}}\lambda^{\ast}=\frac{1}{2}\lambda^{\ast}$, namely $P\{\omega\in \Omega~|~\|x^{\ast}-\bar{x}\|_{\mathcal{Q}}(\omega)<\varepsilon^{\ast}\}\geq1-\frac{1}{2}\lambda^{\ast}>1-\lambda^{\ast}$, that is to say, $\bar{x}\in(x^{\ast}+N_{\theta}(\mathcal{Q},\varepsilon^{\ast},\lambda^{\ast}))\bigcap G$.

(2). By (1), $[H_{cc}(G)]^{-}_{c}=[H_{cc}(G)]^{-}_{\varepsilon,\lambda}$ has the countable concatenation property. Thus we only need to prove that $\bar{G}_{\varepsilon,\lambda}=[H_{cc}(G)]^{-}_{\varepsilon,\lambda}$. We can suppose , without loss of generality, that $\theta\in G$. Then for any $x=\Sigma_{n=1}^{\infty}\tilde{I}_{A_i}g_i\in H_{cc}(G)$, it is obvious that $\{\Sigma_{i=1}^{n}\tilde{I}_{A_i}g_i~|~n\in N\}$ is a $\mathcal{T}_{\varepsilon,\lambda}$--cauchy sequence in $G$ convergent to $x$ since $\{A_n,n\in N\}$ is a countable partition of $\Omega$ to $\mathcal{F}$, which means that $x\in \bar{G}_{\varepsilon,\lambda}$, namely $[H_{cc}(G)]^{-}_{\varepsilon,\lambda}\subset \bar{G}_{\varepsilon,\lambda}$, so $[H_{cc}(G)]^{-}_{\varepsilon,\lambda}=\bar{G}_{\varepsilon,\lambda}$. \hfill $\square$
\end{proof}

\begin{remark} We can illustrate that (1) may be not true if $E$ does not have the countable concatenation property, such an example is omitted to save space.
\end{remark}

\subsection{ On the precise relation between random conjugate spaces under the two kinds of topologies}

For any $RLC$ module $(E,\mathcal{P})$, Guo already pointed out that $E^{\ast}_{\varepsilon,\lambda}$ has the countable concatenation property, cf.\cite{TXG-JFA}, let $H_{cc}(E^{\ast}_{c})$ be the countable concatenation hull of $E^{\ast}_{c}$ in $E^{\ast}_{\varepsilon,\lambda}$. The main result of this subsection is the following:

\begin{theorem} $E^{\ast}_{\varepsilon,\lambda}:=(E,\mathcal{P})^{\ast}_{\varepsilon,\lambda}=(E,\mathcal{P}_{cc})^{\ast}_{c}=H_{cc}(E^{\ast}_{c})$, where, please recall $E^{\ast}_{c}:=(E,\mathcal{P})^{\ast}_{c}$.
\end{theorem}

To prove Theorem 2.30 and meet the needs of Subsection 3.3.1, we first recall from \cite{TXG-XXC} Lemma 2.31 below and the two corollaries to it.

\begin{lemma}($See$ \cite{TXG-XXC}.) Let $X$ be a linear space over $K$, $\{p_n:X\rightarrow L^{0}_{+}(\mathcal{F})~|~n\in N\}$ a sequence of random seminorms on $X$ and $f:X\rightarrow L^{0}(\mathcal{F},K)$ a random linear functional such that $\Sigma_{n=1}^{\infty}p_n(x)$ converges in probability for each $x\in X$ and $|f(x)|\leq\Sigma_{n=1}^{\infty}p_n(x)$ for each $x\in X$. Then there is a sequence $\{f_n~|~n\in N\}$ of random linear functionals such that

\noindent (1). $|f_n(x)|\leq p_n(x)$ for all $x\in X$  and $n\in N$;

\noindent (2). $f(x)=\Sigma_{n=1}^{\infty}f_n(x)$ for all $x\in X$.
\end{lemma}

\begin{corollary}($See$ \cite{TXG-XXC}.) Let $X$ be an $L^{0}(\mathcal{F},K)$--module, $f:X\rightarrow L^{0}(\mathcal{F},K)$ an $L^{0}(\mathcal{F},K)$--linear function, $\{p_n:X\rightarrow L^{0}_{+}(\mathcal{F})~|~n\in N\}$ a sequence of $L^0$--seminorms on $X$ and $\{A_n,n\in N\}$ a countable partition of $\Omega$ to $\mathcal{F}$ such that $|f(x)|\leq\Sigma_{n=1}^{\infty}\tilde{I}_{A_n}p_n(x)$ for all $x\in X$. Then there is a sequence $\{f_n:n\in N\}$ of $L^{0}(\mathcal{F},K)$--linear functions such that

\noindent (1). $|f_n(x)|\leq p_n(x)$ for all $x\in X$  and $n\in N$;

\noindent (2). $f(x)=\Sigma_{n=1}^{\infty}\tilde{I}_{A_n}(f_n(x))$ for all $x\in X$.
\end{corollary}

\begin{corollary} ($See$ \cite{TXG-XXC}.) Let $X$ be an $L^{0}(\mathcal{F},K)$--module, $f:X\rightarrow L^{0}(\mathcal{F},K)$ an $L^{0}(\mathcal{F},K)$--linear function and $\{p_i:X\rightarrow L^{0}_+(\mathcal{F})~|~1\leq i\leq n\}$ n $L^{0}$--seminorms such that $|f(x)|\leq\Sigma_{i=1}^{n}p_i(x)$ for all $x\in X$. Then for each $1\leq i\leq n$ there is an $L^{0}(\mathcal{F},K)$--linear function $f_i$ such that

\noindent (1). $|f_i(x)|\leq p_i(x)$ for all $x\in X$  and $1\leq i\leq n$;

\noindent (2). $f(x)=\Sigma_{i=1}^{n}f_n(x)$ for all $x\in X$.
\end{corollary}

We can now prove Theorem 2.30.

\noindent {Proof of Theorem 2.30.}  Since $\mathcal{P}$ and $\mathcal{P}_{cc}$ induces the same $(\varepsilon,\lambda)$--topology on $E$, $E^{\ast}_{\varepsilon,\lambda}:=(E,\mathcal{P})^{\ast}_{\varepsilon,\lambda}=(E,\mathcal{P}_{cc})^{\ast}_{\varepsilon,\lambda}=(E,\mathcal{P}_{cc})^{\ast}_{c}$, where the last equality comes from the countable concatenation property of $\mathcal{P}_{cc}$ by Corollary 2.21. It remains to prove that $(E,\mathcal{P}_{cc})^{\ast}_{c}=H_{cc}(E^{\ast}_c)$ and we only needs to check that $(E,\mathcal{P}_{cc})^{\ast}_{c}\subset H_{cc}(E^{\ast}_c)$.

Let $f$ be any element of $(E,\mathcal{P}_{cc})^{\ast}_{c}$. Since $\mathcal{P}_{cc}$ is invariant under the operation of finitely many suprema, then there are $\|\cdot\|\in \mathcal{P}_{cc}$ and $\xi\in L^{0}_{++}(\mathcal{F})$ such that $|f(x)|\leq\xi\|x\|$ for all $x\in E$. Let $\|\cdot\|=\Sigma_{n=1}^{\infty}\tilde{I}_{A_n}\|\cdot\|_{\mathcal{Q}_n}$, where $\{A_n,n\in N\}$ is some countable partition of $\Omega$ to $\mathcal{F}$ and each $\mathcal{Q}_n\in \mathcal{P}_f$, then by Corollary 2.32 there is a sequence $\{f_n,n\in N\}$ of $L^{0}(\mathcal{F},K)$--linear functions such that

\noindent (1). $|f_n(x)|\leq\xi\|x\|_{\mathcal{Q}_n}$ for all $x\in E$ and $n\in N$;

\noindent (2). $f(x)=\Sigma_{n=1}^{\infty}\tilde{I}_{A_n}(f_n(x))$ for all $x\in E$.

(1) shows that each $f_n\in E^{\ast}_c$ and (2) further shows that $f=\Sigma_{n=1}^{\infty}\tilde{I}_{A_n}f_n$, so $f\in H_{cc}(E^{\ast}_c)$.
\hfill $\square$

For the further study of random conjugate spaces, please refer to \cite{TXG-SEZ}.

\subsection{Bounded sets under the locally $L^0-$convex topology}

Let $(E,\mathcal{T})$ be a locally $L^{0}$--convex $L^{0}(\mathcal{F},K)$--module. $A\subset E$ is said to be $\mathcal{T}$--bounded if $A$ can be $L^0$--absorbed by every neighborhood of $\theta$. The main results in this subsection are Theorems 2.34 and 2.35 below.

\begin{theorem} $(${\bf Random resonance theorem}$)$ Let $(E_1,\|\cdot\|_1)$ and $(E_2,\|\cdot\|_2)$ be two $RN$ modules such that $E_1$ is $\mathcal{T}_c$--complete and has the countable concatenation property. For a subset $\{T_{\alpha},\alpha\in\Lambda\}$ of $B(E_1,E_2)$, then $\{T_{\alpha},\alpha\in\Lambda\}$ is $\mathcal{T}_c$--bounded in $B(E_1,E_2)$ iff $\{T_{\alpha}x,\alpha\in\Lambda\}$ is $\mathcal{T}_c$--bounded in $E_2$ for all $x\in E_1$.
\end{theorem}

Let $(E_1,\|\cdot\|)$ and $(E_2,\|\cdot\|)$ be $RN$ modules over $K$ with base $(\Omega,\mathcal{F},P)$. It is easy to prove that a linear operator $T:E_1\rightarrow E_2$ belongs to $B(E_1,E_2)$ iff $T$ is a continuous module homomorphism from $(E_1,\mathcal{T}_c)$ to $(E_2,\mathcal{T}_c)$. Hence Theorem 2.34 also gives a resonance theorem for a family of continuous module homomorphisms from $(E_1,\mathcal{T}_c)$ to $(E_2,\mathcal{T}_c)$.

\begin{theorem} Let $(E,\mathcal{P})$ be an $RLC$ module over $K$ with base $(\Omega,\mathcal{F},P)$ and $A\subset E$. Then $A$ is $\mathcal{T}_c$--bounded iff $f(A)$ is $\mathcal{T}_c$--bounded in $(L^{0}(\mathcal{F},K),\mathcal{T}_c)$ for every $f\in E^{\ast}_c$.
\end{theorem}

Theorems 2.34 and 2.35 can be implied by the work on random resonance theorem at the earlier stage of $RLC$ spaces. To see this, let us recall:

\begin{definition} (See \cite{TXG-PHD,TXG-Module}.) Let $(E,\mathcal{P})$ be an $RLC$ space over $K$ with base $(\Omega,\mathcal{F},P)$. A set $A\subset E$ is said to be a.s. bounded if $\bigvee\{\|a\|:a\in A\}\in L^{0}_{+}(\mathcal{F})$ for each $\|\cdot\|\in\mathcal{P}$.
\end{definition}

Lemma 2.37 below is clear by definition.

\begin{lemma} Let $(E,\mathcal{P})$ be an $RLC$ module. Then a set $A$ of $E$ is $\mathcal{T}_c$--bounded iff $A$ is a.s. bounded.
\end{lemma}

Thus, by Proposition 2.26 one can easily see that Theorem 2.34 is a corollary to (2) of Proposition 2.38 below. Since a $\mathcal{T}_{\varepsilon,\lambda}$--complete $RN$ module is a Frech$\acute{e}$t space (namely a complete metrizable linear topological space), (1) of Proposition 2.38 can be proved with the aid of the classical resonance theorem (see \cite{Yosida}). Hence (2) of Proposition 2.38 is more interesting.

\begin{proposition} ($See$ \cite{TXG-PHD,TXG-Module}.) Let $(E_1,\|\cdot\|_1)$ and $(E_2,\|\cdot\|_2)$ be two $RN$ modules over $K$ with base $(\Omega,\mathcal{F},P)$ such that $E_1$ is $\mathcal{T}_{\varepsilon,\lambda}$--complete. Given a subset $\{T_{\alpha},\alpha\in\Lambda\}$ in $B(E_1,E_2)$, we have the following:

\noindent (1). $\{T_{\alpha},\alpha\in\Lambda\}$ is $\mathcal{T}_{\varepsilon,\lambda}$--bounded in $B(E_1,E_2)$ iff $\{T_{\alpha}x,\alpha\in\Lambda\}$ is $\mathcal{T}_{\varepsilon,\lambda}$--bounded in $E_2$ for all
$x\in E_1$.

\noindent (2). $\{T_{\alpha},\alpha\in\Lambda\}$ is a.s. bounded in $B(E_1,E_2)$ iff  $\{T_{\alpha}x,\alpha\in\Lambda\}$ is a.s. bounded in $E_2$ for all
$x\in E_1$.
\end{proposition}

One can also see that Theorem 2.35 is only a corollary to (3) of Proposition 2.39 below. However, (1) and (2) of Proposition 2.39 are also interesting because of the use of random conjugate spaces.

\begin{proposition} ($See$ \cite{TXG-PHD,TXG-Module,TXG-Sur}.) Let $(E,\mathcal{P})$ be an $RLC$ space over $K$ with base $(\Omega,\mathcal{F},P)$ and $A$ a subset of $E$. Then we have:

\noindent (1). $A$ is $\mathcal{T}_{\varepsilon,\lambda}$--bounded iff $f(A)$ is $\mathcal{T}_{\varepsilon,\lambda}$--bounded in $(L^0(\mathcal{F},K),\mathcal{T}_{\varepsilon,\lambda})$ for each $f\in E^{\ast}_{I}$.

\noindent (2). $A$ is $\mathcal{T}_{\varepsilon,\lambda}$--bounded iff $f(A)$ is $\mathcal{T}_{\varepsilon,\lambda}$--bounded in $(L^0(\mathcal{F},K),\mathcal{T}_{\varepsilon,\lambda})$ for each $f\in E^{\ast}_{II}$.

\noindent (3). $A$ is a.s. bounded iff $f(A)$ is a.s. bounded in $(L^0(\mathcal{F},K),|\cdot|)$ for each $f\in E^{\ast}_{I}$.

\noindent (4). $A$ is a.s. bounded iff $f(A)$ is a.s. bounded in $(L^0(\mathcal{F},K),|\cdot|)$ for each $f\in E^{\ast}_{II}$.
\end{proposition}

\begin{remark} $\mathcal{T}_c$--bounded sets (or more general a.s. bounded sets) have played a crucial role in random duality theory in this paper and in \cite{TXG-XXC}, whereas $\mathcal{T}_{\varepsilon,\lambda}$--bounded sets are important in other fields, for example, they are called ``probabilistically bounded sets" in the theory of probabilistic normed spaces (see \cite{SS}) and are called ``stochastically bounded sets" for Banach space-valued random elements in probability theory in Banach spaces.
\end{remark}

\section{Random convex analysis over random locally convex modules under the locally $L^{0}$--convex topology}

As has been mentioned in Section 1 of this paper, Lemma 2.28 of \cite{FKV} is key but false, we have to first correct it and the closely related basic results for the further development of random convex analysis. The first two parts of this section not only have corrected them but also have refreshed the related basic results based on the precise relation between $E^{\ast}_{\varepsilon,\lambda}$ and $E^{\ast}_{c}$.

\subsection{Separation under the locally $L^0$--convex topology}

First, Lemma 2.28 of \cite{FKV} should be modified as Lemma 3.1 below.

\begin{lemma} Let $(E,\mathcal{P})$ be an $RLC$ module over $K$ with base $(\Omega,\mathcal{F},P)$ such that $\mathcal{P}$ has the countable concatenation property, $M$ a $\mathcal{T}_c$--closed subset with the countable concatenation property  and $x\in E$ such that $\tilde{I}_A\{x\}\bigcap\tilde{I}_A M=\emptyset$ for all $A\in \mathcal{F}$ with $P(A)>0$. Then there is an $L^0$--convex, $L^0$--absorbent and $L^0$--balanced $\mathcal{T}_c$--neighborhood $U$ of $\theta$ such that $$\tilde{I}_A(x+U)\bigcap\tilde{I}_A(M+U)=\emptyset$$ for all $A\in\mathcal{F}$ with $P(A)>0$.
\end{lemma}

Lemma 2.28 of \cite{FKV} only requires $M$ to satisfy the condition that $\tilde{I}_{A}M+\tilde{I}_{A^c}\subset M$ for all $A\in\mathcal{F}$, which is much weaker than the countable concatenation property as assumed in our Lemma 3.1. Example 3.2 below is a counterexample to Lemma 2.28 of \cite{FKV}.

Let us  first point out that Lemma 3.1 has been implied by Lemma 3.10 of \cite{TXG-JFA}, which can be explained as follows. Let us recalled from \cite{TXG-JFA}: for each $\mathcal{Q}\in\mathcal{P}_f$ and $\varepsilon\in L^{0}_{++}(\mathcal{F})$, let $U_{\mathcal{Q},\varepsilon}[x]=\{y\in E~|~\|x-y\|_{\mathcal{Q}}\leq\varepsilon\}$, $e^{\ast}_{\mathcal{Q}}(x,M)=\bigwedge\{\varepsilon\in L^0_{++}(\mathcal{F})~|~U_{\mathcal{Q},\varepsilon}[x]\bigcap M\neq\emptyset\}$ and $e^{\ast}(x,M)=\bigvee\{e^{\ast}_{\mathcal{Q}}(x,M)~|~\mathcal{Q}\in\mathcal{P}_{f}\}$. Then Lemma 3.10 of \cite{TXG-JFA} shows that $e^{\ast}(x,M)\bigwedge 1\in L^0_{++}(\mathcal{F})$, which is just what the proof of Lemma 2.28 of \cite{FKV} needs, the remainder of the proof of Lemma 3.1 is the same as the corresponding part of the proof of Lemma 2.28 of \cite{FKV}.

\begin{example} Let $(\Omega,\mathcal{F},P)$ be a nonatomic probability space (namely $\mathcal{F}$ does not include any $P$--atoms), $(E,\mathcal{P})=(L^{0}(\mathcal{F},R),|\cdot|)$ and $M=\{x\in E~|~$there exists a positive number $m_x$ such that $x>m_x$ on $\Omega\}$. Then Claim 3.3 below shows that $M$ is $L^0$--convex, $\mathcal{T}_c$--closed and $\mathcal{T}_c$--open. Further, Claim 3.4 below shows that $\tilde{I}_A\{0\}\bigcap\tilde{I}_A M=\emptyset$ for all $A\in\mathcal{F}$ with $P(A)>0$, but for each $L^0$--convex, $L^0$--absorbent and $L^0$--balanced $\mathcal{T}_c$--neighborhood $U$ of 0 there is an $A_U\in\mathcal{F}$ with $P(A_U)>0$ such that $$\tilde{I}_{A_U}U\bigcap\tilde{I}_{A_U}(M+U)\neq\emptyset.$$Thus this provides a counterexample to Lemma 2.28 of \cite{FKV}.
\end{example}

\begin{claim} $M$ in Example 3.2 is $L^0$--convex, $\mathcal{T}_c$--closed and $\mathcal{T}_c$--open.
\end{claim}

\begin{proof} First, it is obvious that $M$ is $L^0$--convex.

Second, $M$ is $\mathcal{T}_c$--open. For any $y\in M$, by definition there is some positive number $m_y$ such that $y>m_y$ on $\Omega$. Let $\varepsilon^0\equiv\frac{1}{2}m_y$ and $\varepsilon$ be the equivalence class of $\varepsilon^0$, then $\varepsilon\in L^0_{++}(\mathcal{F})$ and hence $B(\varepsilon):=\{x\in E~|~|x|\leq\varepsilon\}$ is a $\mathcal{T}_c$--neighborhood of 0, it is also easy to check that $y+B(\varepsilon)\subset M$.

Finally, $M$ is also $\mathcal{T}_c$--closed, namely $E\setminus M$ is $\mathcal{T}_c$--open, which will be proved in the following three cases.

Case (1): when $y\in E\setminus M$ and $y\not\in L^0_+(\mathcal{F})$, there is $D\in\mathcal{F}$ with $P(D)>0$ such that $y<0$ on $D$. Let $\varepsilon=\tilde{I}_{D^c}+\frac{1}{2}\tilde{I}_{D}|y|(\in L^{0}_{++}(\mathcal{F}))$ and $B(\varepsilon)=\{x\in E~|~|x|\leq\varepsilon\}$, then $y+B(\varepsilon)\subset E\setminus M$. In fact, for any $z\in y+B(\varepsilon)$, $z-y\leq \tilde{I}_{D^c}+\frac{1}{2}\tilde{I}_{D}|y|$ implies that $z\leq y+\frac{1}{2}|y|=-\frac{1}{2}|y|<0$ on $D$, namely $z\in E\setminus M$.

Case (2): when $y\in E\setminus M$, $y\in L^0_+(\mathcal{F})$ and $y\not\in L^{0}_{++}(\mathcal{F})$, then there is $D\in\mathcal{F}$ with $P(D)>0$ such that $y=0$ on $D$. Since $(\Omega,\mathcal{F},P)$ is nonatomic, there is a countable partition $\{D_n,n\in N\}$ of $D$ to $\mathcal{F}$ such that $P(D_n)=\frac{1}{2^n}P(D)$ for each $n\in N$. Let $\varepsilon=\tilde{I}_{D^c}+\Sigma_{n=1}^{\infty}\frac{1}{n}\tilde{I}_{D_n} (\in L^{0}_{++}(\mathcal{F}))$ and $B(\varepsilon)=\{x\in E~|~|x|\leq\varepsilon\}$, then $z\leq \frac{1}{n}$ on $D_n$ for any $z\in y+B(\varepsilon)$, which implies that $P\{\omega\in\Omega~|~z(\omega)\leq\frac{1}{n}\}\geq P(D_n)>0$ for all $n\in N$, namely $y+B(\varepsilon)\subset E\setminus M$.

Case (3): when $y\in E\setminus M$ and $y\in L^0_{++}(\mathcal{F})$, then $P\{\omega\in \Omega~|~y(\omega)<\frac{1}{n}\}>0$ for each $n\in N$ by the definition of $M$. Let $H_n=[y<\frac{1}{n}]$ and $D_n=[\frac{1}{n+1}\leq y<\frac{1}{n}]$ for any $n\in N$, then $D_i\bigcap D_j=\emptyset$ for $i\neq j$ and $H_n=\Sigma_{i=n}^{\infty}D_i$. Obviously, it is impossible that there is some $k\in N$ such that $P(D_n)=0$ for all $n\geq k$. So, we can suppose, without loss of generality, that $P(D_n)>0$ for each $n\in N$. Let $D=\Sigma_{n=1}^{\infty}D_n$, $\varepsilon=I_{D^{c}}+\Sigma_{n=1}^{\infty} \frac{1}{n}I_{D_{n}}(\in L^{0}_{++}(\mathcal{F}))$ and $B(\varepsilon)
=\{x\in E|~|x|\leq \varepsilon\}$, then for any $z\in y+B(\varepsilon)$, $z\leq\frac{2}{n}$ on $D_n$, which means that $P\{\omega\in \Omega|z(\omega)\leq\frac{2}{n}\}\geq P(D_n)>0$ for each $n\in N$, and hence $z\in E\setminus M$. \hfill $\square$
\end{proof}

\begin{claim} Let $(E, {\mathcal P})$ and $M$ be the same as in Example 3.2. Then ${\tilde I}_A\{0\}\cap {\tilde I}_AM=\emptyset$ for all $A\in {\mathcal F}$ with $P(A)>0$. But for any $L^0$--convex, $L^0$--absorbent and $L^0$--balanced ${\mathcal T}_c$--neighborhood $U$ of $0$ there is always $A_U\in {\mathcal F}$ with $P(A_U)>0$ such that ${\tilde I}_{A_U}U\cap {\tilde I}_{A_U}(M+U)\neq\emptyset$.
\end{claim}

\begin{proof} There is $\varepsilon \in L^0_{++}({\mathcal F})$ for $U$ stated above such that $B(\varepsilon):=\{x\in E~|~|x|\leq\varepsilon\}\subset U$. For a representative $\varepsilon^0$ of $\varepsilon$, let $A_1=\{\omega\in \Omega~|~\varepsilon^0(\omega)\geq 1\}$ and $A_n=\{\omega\in\Omega~|~\frac{1}{n}\leq \varepsilon^0(\omega)<\frac{1}{n-1}\}$ for $n\geq 2$, then it is clear that $\sum^{\infty}_{n=1}P(A_n)=1$, and hence there is some $n_0\in N$ such that $P(A_{n_0})>0$. Let $A_U=A_{n_0}$ and $y_0={\tilde I}_{A^c_U}+{\tilde I}_{A_U}\varepsilon$, then $\frac{1}{n_0}\leq y_0<\frac{1}{n_0-1}$ on $A_U$ ( note: this is also true for $n_0=1$ ) and $y_0\geq \frac{1}{n_0}$ on $\Omega$ (namely, $y_0\in M$). Since ${\tilde I}_{A_U}y_0={\tilde I}_{A_U}\varepsilon\in {\tilde I}_{A_U}B(\varepsilon)\subset {\tilde I}_{A_U} U$ and ${\tilde I}_{A_U}y_0\in {\tilde I}_{A_U}M\subset{\tilde I}_{A_U}(M+U)$, so ${\tilde I}_{A_U}U\cap {\tilde I}_{A_U}(M+U)\neq\emptyset$. \hfill $\square$
\end{proof}

In fact, Example 3.2 is also a counterexample to Theorem 2.8 of \cite{FKV}, which can be explained as follows.

Since $(E, {\mathcal P})=(L^0({\mathcal F}, R), |\cdot|)$ is an $RN$ module, $|\cdot|$ has the countable concatenation property and $E^\ast_c=E^\ast_{\varepsilon, \lambda}$. It is obvious that $0\in {\overline{M}}_{\varepsilon, \lambda}$ ( namely, the ${\mathcal T}_{\varepsilon, \lambda}$--closure of $M$ ), and hence for each $f\in E^\ast_c=E^\ast_{\varepsilon, \lambda}$ there exists a sequence $\{y_n, n\in N\}$ in $M$ such that $\{f(y_n): n\in N\}$ converges in probability P to $0$, which means that it is impossible that there exists $f\in E^\ast_c$ such that $0=f(0)>\bigvee\{f(y): y\in M\}$ on $\Omega$.

Theorem 2.8 of \cite{FKV} should be modified as follows:

\begin{theorem} Let $(E, {\mathcal P})$ be an $RLC$ module over $K$ with base $(\Omega, {\mathcal F}, P)$ such that ${\mathcal P}$ has the countable concatenation property, $x\in E$ and $M$ a nonempty ${\mathcal T}_c$--closed $L^0$--convex subset with the countable concatenation property. If ${\tilde I}_A\{x\}\cap {\tilde I}_AM=\emptyset$ for all $A\in {\mathcal F}$ with $P(A)>0$, then there is $f\in E^\ast_c$ such that $$(Re f)(x)>\bigvee\{(Ref)(y): y\in M\} ~\hbox{on}~\Omega.$$
\end{theorem}

Theorem 3.5 first occurred in \cite{TXG-JFA} in its current form, which can be derived from Theorem 3.6 below. Here, we give a more simpler form of Theorem 3.6 for convenience in use, this form has been implied in the process of the proof of Theorem 3.7 of \cite{TXG-JFA}.

\begin{theorem} ($See$ \cite{TXG-JFA}.) Let $(E, {\mathcal P})$ be an random locally convex module over $K$ with base $(\Omega, {\mathcal F}, P)$, $x\in E$ and $M$ a ${\mathcal T}_{\varepsilon, \lambda}$--closed $L^0$--convex nonempty subset of $E$. If $x\notin M$, then there is $f\in E^\ast_{\varepsilon, \lambda}$ such that:\\
(1). $(Ref)(x)>\bigvee\{(Ref)(y): y\in M\}~\hbox{on}~A$;\\
(2). $(Ref)(x)=\bigvee\{(Ref)(y): y\in M\}~\hbox{on}~A^c$.\\
Here, $A$ is a representative of $[d^\ast(x, M)>0]$ and $d^\ast(x, M)=\bigvee\{d^\ast_{\mathcal{Q}}(x, M): \mathcal{Q}\in {\mathcal P}_f\}$, where $d^\ast_{\mathcal{Q}}(x, M)=\bigwedge\{\|x-y\|_{\mathcal{Q}}: y\in M\}$ for $\mathcal{Q}\in {\mathcal P}_f$. Since $x\not\in M$ iff $d^\ast(x, M)>0$, $P(A)>0$.
\end{theorem}

\begin{remark} When $(E, {\mathcal P})$ is an $RN$ module, $d^\ast(x, M)$ is just the random distance from $x$ to $M$. Thus Theorem 3.6 is best possible from the degree that $f$ separates $x$ from $M$.
\end{remark}

The condition that ${\tilde I}_A\{x\}\cap {\tilde I}_AM=\emptyset$ for all $A\in {\mathcal F}$ with $P(A)>0$ is to guarantee the separation of $x$ from $M$ by $f$ with probability $1$, but the condition is too strong to be easily satisfied in applications. In fact, Theorem 3.6 also yields a kind of generalization for Theorem 3.5, namely Corollary 3.8 below.

\begin{corollary} Let $(E, {\mathcal P})$ be an $RLC$ module over $K$ with base $(\Omega, {\mathcal F}, P)$ such that ${\mathcal P}$ has the countable concatenation property, $x\in E$ and $M\subset E$ a nonempty ${\mathcal T}_c$--closed $L^0$--convex set with the countable concatenation property. If $x\notin M$, then there is $f\in E^\ast_c$ such that:\\
(1). $(Ref)(x)>\bigvee\{(Ref)(y): y\in M\}~\hbox{on}~A$;\\
(2). $(Ref)(x)=\bigvee\{(Ref)(y): y\in M\}~\hbox{on}~A^c$.\\
Here $A$ and $d^\ast(x, M)$ are the same as in Theorem 3.6.
\end{corollary}

Example 3.2 shows that it is essential that $M$ has the countable concatenation property in Corollary 3.8, but random duality theory in Subsection 3.3 needs another generalization of Corollary 3.8, namely Corollary 3.9 below, in which the condition that ${\mathcal P}$ has the countable concatenation property is removed but (1) of Corollary 3.8 only holds on a subset $B$ of $A$ with $P(B)>0$. In fact, Corollary 3.9 is more convenient for use and its proof is based on Theorem 2.30.

\begin{corollary} Let $(E, {\mathcal P})$ be an $RLC$ module over $K$ with base $(\Omega, {\mathcal F}, P)$, $x\in E$ and $M\subset E$ a nonempty ${\mathcal T}_c$--closed $L^0$--convex set with the countable concatenation property. If $x\notin M$, then there exist $f\in E^\ast_c$ and $B\in {\mathcal F}$ with $P(B)>0$ such that:\\
(1). $(Ref)(x)>\bigvee\{(Ref)(y): y\in M\}~\hbox{on}~B$;\\
(2). $(Ref)(x)=\bigvee\{(Ref)(y): y\in M\}~\hbox{on}~B^c$.
\end{corollary}

\begin{proof} We consider the separation problem in $(E, {\mathcal P}_{cc})$. Since ${\mathcal P}_{cc}$ has the countable concatenation property and the locally $L^0$--convex topology induced by ${\mathcal P}_{cc}$ is stronger than that induced by ${\mathcal P}$. We can apply Corollary 3.8 to $(E, {\mathcal P}_{cc})$, $x$ and $M$, then there is $g\in (E, {\mathcal P}_{cc})^\ast_c$ such that:\\
(3). $(Reg)(x)>\bigvee\{(Reg)(y): y\in M\}~\hbox{on}~A$;\\
(4). $(Reg)(x)=\bigvee\{(Reg)(y): y\in M\}~\hbox{on}~A^c$.\\
Here, please note that ${\mathcal P}$ and ${\mathcal P}_{cc}$ induce the same $d^\ast(x, M)$, so $A$ is still a representative of $[d^\ast(x, M)>0]$.

Since $(E, {\mathcal P}_{cc})^\ast_c=H_{cc}(E^\ast_c)$ by Theorem 2.30, $g=\sum^\infty_{n=1}{\tilde I}_{A_n}g_n$ for some countable partition $\{A_n, n\in N\}$ of $\Omega$ to ${\mathcal F}$ and some sequence $\{g_n, n\in N\}$ in $E^\ast_c$. Let $n_0\in N$ be such that $P(A\cap A_{n_0})>0$ and further let $B=A\cap A_{n_0}$ and $f={\tilde I}_{A\cap A_{n_0}}g_{n_0}$, then $f$ and $B$ meet the needs of (1) and (2). \hfill \hfill $\square$
\end{proof}

In Theorem 3.6, $f$ belongs to $E^\ast_{\varepsilon, \lambda}$, but the study of random admissible topology in Subsection 3.3.2 requires an $f\in E^\ast_c$ to separate a point from a ${\mathcal T}_{\varepsilon, \lambda}$--closed $L^0$--convex subset. By noting that $E^\ast_{\varepsilon, \lambda}=H_{cc}(E^\ast_c)$, we can use the same reasoning as in the proof of Corollary 3.9 to obtain the following generalization of Theorem 3.6:

\begin{corollary} Let $(E, {\mathcal P})$ be an $RLC$ module over $K$ with base $(\Omega, {\mathcal F}, P)$, $x\in E$ and $M\subset E$ a nonempty ${\mathcal T}_{\varepsilon, \lambda}$--closed $L^0$--convex subset. If $x\notin M$, then there are $f\in E^\ast_c$ and some $B\in {\mathcal F}$ with $P(B)>0$ such that:\\
(1). $(Ref)(x)>\bigvee\{(Ref)(y): y\in M\}~\hbox{on}~B$;\\
(2). $(Ref)(x)=\bigvee\{(Ref)(y): y\in M\}~\hbox{on}~B^c$.
\end{corollary}

\begin{remark} Let $\xi$ be any element in $L^0({\mathcal F}, K)$ and $\xi_0$ a representative of $\xi$. Define $\xi_0^{-1}\colon\Omega\to K$ by $\xi_0^{-1}(\omega)=(\xi_0(\omega))^{-1}$ if $\xi_0(\omega)\neq 0$ and by $0$ if $\xi_0(\omega)=0$, then $\xi^{-1}:=$ the equivalence class of $\xi_0^{-1}$ is called the generalized inverse of $\xi$. $|\xi|^{-1}\xi$ is called the sign of $\xi$, denoted by $sgn(\xi)$, then ${\overline {sgn(\xi)}}\xi=|\xi|$, where ${\overline {sgn(\xi)}}$ stands for the complex conjugate of $sgn(\xi)$. Further, we also have that $\xi\cdot\xi^{-1}=\xi^{-1}\cdot\xi=I_{[\xi\neq 0]}$. If $M$ in Corollary 3.9 or Corollary 3.10 is also $L^0$--balanced, then one can make use of the notion of the sign for element in $L^0({\mathcal F}, K)$ to see that (1) and (2) of the two corollaries can be rewritten as ( cf. \cite{TXG-strict} ):\\
(1). $|f(x)|>\bigvee\{|f(y)|: y\in M\}~\hbox{on}~B$;\\
(2). $|f(x)|=\bigvee\{|f(y)|: y\in M\}~\hbox{on}~B^c$.
\end{remark}

Let $\xi=|f(x)|$ and $\eta=\bigvee\{|f(y)|: y\in M\}$, then multiplying the above two sides by $(\frac{\xi+\eta}{2})^{-1}$ and replacing $f$ with $(\frac{\xi+\eta}{2})^{-1}f$ ( still denoted by $f$ ) will obtain the following two relations:\\
(3). $|f(x)|>\bigvee\{|f(y)|: y\in M\}$;\\
(4). $|f(x)|\nleqslant 1$ and $\bigvee\{|f(y)|: y\in M\}\leq 1$.\\
(3) and (4) will be used in the proof of random bipolar theorem in Subsection 3.3.1.

All the above results from Theorem 3.5 to Corollary 3.10 are concerned with the separation between a point and a closed subset. Theorem 3.12 below, due to \cite{FKV}, is concerned with the separation between two $L^0$--convex sets with one of them open, which is peculiar to the locally $L^0$--convex topology since it is impossible to establish such a theorem under the $(\varepsilon, \lambda)$--topology. Since the Proof of Theorem 3.12 does not involve the problem of whether a locally $L^0$--convex topology can be induced by a family of $L^0$--seminorms, we still state it as in \cite{FKV}.

\begin{theorem}($See$ \cite{FKV}.) Let $(E, {\mathcal T})$ be a Hausdorff locally $L^0$--convex $L^0({\mathcal F}, K)$--module,
$M$ and $G$ two nonempty $L^0$--convex sets of $E$ with $G$ open. If ${\tilde I}_AM\cap {\tilde I}_AG=\emptyset$ for all $A\in {\mathcal F}$ with $P(A)>0$, then there is $f\in E^\ast_c$ such that:$$(Ref)(y)<(Ref)(z)~\hbox{on $\Omega$ for all $y\in G$ and $z\in M$}.$$
\end{theorem}

Theorem 3.12 will be used in Section 3.4, see \cite{TXG-JFA,TXG-GS} for its slight generalizations.

\subsection{The Fenchel-Moreau dual representation theorem under the locally $L^0$--convex topology}

Let $E$ be an $L^0({\mathcal F})$--module and $f$ a function from $E$ to ${\bar L}^0({\mathcal F})$. The effective domain of $f$ is denoted by $dom(f):=\{x\in E~|~|f(x)|<+\infty~\hbox{on}~\Omega\}$ and the epigraph of $f$ by $eip(f):=\{(x,r)\in E\times L^0({\mathcal F})~|~f(x)\leq r\}$. $f$ is proper if $dom(f)\neq\emptyset$ and $f(x)>-\infty~\hbox{on}~\Omega$. $f$ is $L^0({\mathcal F})$--convex if $f(\xi x+(1-\xi)y)\leq \xi f(x)+(1-\xi)f(y)$ for all $x,~y\in E$ and $\xi\in L^0_+({\mathcal F})$ with $0\leq \xi\leq 1$, where the following convention is adopted: $0\cdot(\pm\infty)=0$ and $+\infty\pm(\pm\infty)=+\infty$. If $(E, {\mathcal T})$ is a topological $L^0({\mathcal F})$--module, a proper function $f$ is lower semicontinuous (or ${\mathcal T}$--lower semicontinuous if there is a possible confusion) if $\{x\in E~|~f(x)\leq r\}$ is closed for all $r\in L^0({\mathcal F})$.

We can now state the main result of this subsection as Theorem 3.13 below, which is a modification and improvement of Theorem 3.8 of \cite{FKV}.

\begin{theorem} Let $(E, {\mathcal P})$ be an $RLC$ module over $R$ with base $(\Omega, {\mathcal F}, P)$ such that $E$ has the countable concatenation property. If $f$ is a proper, ${\mathcal T}_c$--lower semicontinuous $L^0({\mathcal F})$--convex function from $E$ to ${\bar L}^0({\mathcal F})$, then $f^{\ast\ast}_{c}=f$. Here $f^\ast_c\colon E^\ast_c\to {\bar L}^0({\mathcal F})$ is defined by $f^\ast_c(g)=\bigvee\{g(x)-f(x)~|~x\in E\}$ for all $g\in E^\ast_c$, called the ${\mathcal T}_c$--conjugate ( or penalty ) function of $f$, and $f^{\ast\ast}_{c}\colon E \to {\bar L}^0({\mathcal F})$ by $f^{\ast\ast}_{c}(x)=\bigvee\{g(x)-f^{\ast}_{c}(g)~|~g\in E^\ast_c\}$ for all $x\in E$, called the ${\mathcal T}_c$--biconjugate function of $f$.
\end{theorem}

As compared with Theorem 3.8 of \cite{FKV}, Theorem 3.13 requires the additional condition that $E$ has the countable concatenation property and remove the condition that ${\mathcal P}$ has the countable concatenation property. From the sequel of this subsection, one can immediately see that the additional condition is essential, while Theorem 2.30 can be used to remove the condition on ${\mathcal P}$.

To prove Theorem 3.13, let us first study the properties of an $L^0$--convex function.

Let $E$ be an $L^0({\mathcal F})$--module. $f\colon E\to {\bar L}^0({\mathcal F})$ is said to be local ( or, to have the local property ) if ${\tilde I}_Af(x)={\tilde I}_Af({\tilde I}_Ax)$ for all $x\in E$ and $A\in {\mathcal F}$. It is clear that ${\tilde I}_Af(x)=f({\tilde I}_Ax)$ for all $x\in E$ and $A\in {\mathcal F}$ iff $f$ is local, if $f(0)=0$.

\begin{lemma}($See$ \cite{FKV,FKV-appro}.) Let $E$ be an $L^0({\mathcal F})$--module. Then a function $f\colon E\to {\bar L}^0({\mathcal F})$ is $L^0$--convex iff $f$ is local and $eip(f)$ is $L^0$--convex.
\end{lemma}

According to Lemma 3.1, Proposition 3.4 and Lemma 3.10 of \cite{FKV} together can be modified to Proposition 3.15 below.

\begin{proposition} Let $(E, {\mathcal P})$ be an $RLC$ module over $R$ with base $(\Omega, {\mathcal F}, P)$ such that both $E$ and ${\mathcal P}$ have the countable concatenation property. If $f\colon E\to {\bar L}^0({\mathcal F})$ is a proper and local function, then the following are equivalent:\\
(1). $f$ is ${\mathcal T}_c$--lower semicontinuous.\\
(2). $eip(f)$ is closed in $(E, {\mathcal T}_c)\times (L^0({\mathcal F}), {\mathcal T}_c)$.\\
(3). $\varliminf_\alpha f(x_\alpha)\geq f(x)$ for all the nets $\{x_\alpha, \alpha\in \Gamma\}\subset E$ with $x_\alpha\to x$ with respect to ${\mathcal T}_c$ for some $x\in E$, where $\varliminf_\alpha f(x_\alpha)=\bigvee_{\beta\in\Gamma}(\bigwedge_{\alpha\geq\beta} f(x_\alpha))$.
\end{proposition}

\begin{proof} It is clear that (3)$\Rightarrow$(2)$\Rightarrow$(1). As to the converse implications, one can only needs to notice that both the set $\{x\in E|f(x)\leq r\}$ for any given $r\in L^0({\mathcal F})$ and $eip(f)$ have the countable concatenation property if $E$ has the property. The remainder of proof is the same as the corresponding part of proofs of Proposition 3.4 and Lemma 3.10 of \cite{FKV} since the arguments in \cite{FKV} are feasible in the case when $E$ has the countable concatenation property (in fact, only in the case Lemma 3.1 can be applied to the sets $\{x\in E|f(x)\leq r\}$ for any given $r\in L^0({\mathcal F})$ and $eip(f)$). \hfill \hfill $\square$
\end{proof}

Let $E$ be an $L^{0}(\mathcal{F},K)$--module with the countable concatenation property and $G\subset E$ a nonempty subset. $H_{cc}(G)$ always denotes the countable concatenation hull of $G$ in $E$, namely $H_{cc}(G)=\{\Sigma_{n=1}^{\infty}\tilde{I}_{A_n}g_n~|~\{A_n,n\in N\}$ is a countable partition of $\Omega$ to $\mathcal{F}$ and $\{g_n,n\in N\}$ is sequence in $G\}$. For any $x\in H_{cc}(G)$, $\Sigma_{n=1}^{\infty}\tilde{I}_{A_n}g_n$ is called a canonical representation of $x$ if $\{A_n,n\in N\}$ is a countable partition of $\Omega$ to $\mathcal{F}$ and $\{g_n,n\in N\}$ is a sequence in $G$ such that $x=\Sigma_{n=1}^{\infty}\tilde{I}_{A_n}g_n$. Lemma 3.16 below is almost obvious but frequently used in the proofs of the subsequent key theorems in Section 4 as well as in Theorem 3.13, thus we summarize and prove it as follows:

\begin{lemma} Let $E$ be an $L^{0}(\mathcal{F})$--module with the countable concatenation property. Then we have the following statements:

\noindent $(1)$. Let $f:E\rightarrow \bar{L}^{0}(\mathcal{F})$ have the local property and $x=\Sigma_{n=1}^{\infty}\tilde{I}_{A_n}x_n$ for some countable partition $\{A_n,n\in N\}$ of $\Omega$ to $\mathcal{F}$ and some sequence $\{x_n,n\in N\}$ in $E$, then $f(x)=\Sigma_{n=1}^{\infty}\tilde{I}_{A_n}f(x_n)$.

\noindent $(2)$. Let $f:E\rightarrow \bar{L}^{0}(\mathcal{F})$ have the local property and $G\subset E$ be a nonempty subset, then $\bigvee\{f(x)~|~x\in G\}=\bigvee\{f(x)~|~x\in H_{cc}(G)\}$.

\noindent $(3)$. Let $f$ and $g$ be any two functions from $E$ to $\bar{L}^{0}(\mathcal{F})$ such that they both have the local property and $G\subset E$ a nonempty subset. If $f(x)=g(x)$ for all $x\in G$, then $f(x)=g(x)$ for all $x\in H_{cc}(G)$.

\noindent $(4)$. Let $\{f_{\alpha},\alpha\in\Gamma\}$ be a family of functions from $E$ to $\bar{L}^{0}(\mathcal{F})$ such that each $f_{\alpha}$ has the locally property, then $f:E\rightarrow \bar{L}^{0}(\mathcal{F})$ defined by $f(x)=\bigvee\{f_{\alpha}(x)~|~\alpha\in\Gamma\}$ for all $x\in E$, also has the local property.
\end{lemma}

\begin{proof} (1). $f(x)=(\Sigma_{n=1}^{\infty}\tilde{I}_{A_n})f(x)=\Sigma_{n=1}^{\infty}\tilde{I}_{A_n}f(x)=\Sigma_{n=1}^{\infty}\tilde{I}_{A_n}f(\tilde{I}_{A_n}x)= \Sigma_{n=1}^{\infty}\tilde{I}_{A_n}f(\tilde{I}_{A_n}x_n)=\Sigma_{n=1}^{\infty}\tilde{I}_{A_n}f(x_n)$.

(2). Let $\xi=\bigvee\{f(x)~|~x\in G\}$ and $\eta=\bigvee\{f(x)~|~x\in H_{cc}(G)\}$, then $\xi\leq\eta$ is clear, it remains to prove $\eta\leq\xi$. For any $x\in H_{cc}(G)$, let $\Sigma_{n=1}^{\infty}\tilde{I}_{A_n}g_n$ be a canonical representation of $x$, then $f(x)=\Sigma_{n=1}^{\infty}\tilde{I}_{A_n}f(g_n)\leq\xi$, so $\eta\leq\xi$.

(3). It is clear by (1).

(4). It is also clear by definition. \hfill \hfill $\square$
\end{proof}

We can now prove Theorem 3.13.

\newproof{pot}{Proof of Theorem 3.13}
\begin{pot}

We first consider the special case when ${\mathcal P}$ has the countable concatenation property. Since $f$ is $L^0$--convex, $f$ is local and $eip(f)$ is $L^0$--convex by Lemma 3.15. Further, since $E$ and $L^0({\mathcal F})$ have the countable concatenation property, $eip(f)$ also has the property. To sum up, $eip(f)$ is a ${\mathcal T}_c$--closed $L^0$--convex set with the countable concatenation property in the random locally convex module $(E\times L^0({\mathcal F}), {\tilde {\mathcal P}})$, where ${\tilde {\mathcal P}}=\{\|\cdot\|+|\cdot|:\|\cdot\|\in {\mathcal P}\}$ and $(\|\cdot\|+|\cdot|)(x, r)=\|x\|+|r|$ for all $(x, r)\in E\times L^0({\mathcal F})$ and $\|\cdot\|\in {\mathcal P}$, it is also obvious that ${\tilde {\mathcal P}}$ has the countable concatenation property and $(E\times L^0({\mathcal F}), {\tilde {\mathcal P}})^\ast_c=E^\ast_c\times L^0({\mathcal F})$. By applying Theorem 3.6 of this paper rather than Theorem 2.8 of \cite{FKV} to $E\times L^0({\mathcal F})$ and $eip(f)$, one can complete the proof of the special case along the idea of proof of Theorem 3.8 of \cite{FKV}.

Now, we consider the general case, namely ${\mathcal P}$ may not necessarily have the countable concatenation property. In fact, based on Theorem 2.30, it is easy to complete the proof of the general case. We consider the problem in $(E, {\mathcal P}_{cc})$. Since ${\mathcal P}_{cc}$ has the countable concatenation property and the locally $L^0$--convex topology induced by ${\mathcal P}_{cc}$ is stronger than that induced by ${\mathcal P}$, applying what has been proved above to $f$ and $(E, {\mathcal P}_{cc})$ we can obtain:$$f(x)=\bigvee\{u(x)-f^\ast_c(u)~|~u\in (E, {\mathcal P}_{cc})^\ast_c\}~\hbox{ for all $x\in E$.}$$
Since $f^\ast_c$ has the local property and $u(x)$ is, of course, local with respect to $u$ for a fixed $x\in E$, then $u(x)-f^{\ast}_{c}(u)$ is local with respect to $u$ when $x$ is fixed. So by (2) of Lemma 3.16, we have that $f(x)=\bigvee\{u(x)-f^{\ast}_{c}(u)~|~u\in H_{cc}(E^{\ast}_c)\}$ (by the fact that $(E, {\mathcal P}_{cc})^\ast_c=H_{cc}(E^\ast_c)$, where $E^\ast_c=(E, {\mathcal P})^\ast_c$) $=\bigvee\{u(x)-f^{\ast}_{c}(u)~|~u\in E^{\ast}_c\}$.
\hfill \hfill $\square$
\end{pot}

\subsection{Random duality under the locally $L^0$--convex topology with respect to random duality pair}

Only the classical duality theory with respect to a duality pair can give a thorough treatment of classical conjugate space theory of locally convex spaces. The theory of random conjugate spaces  occupies a central place in the study of $RN$ modules and $RLC$ modules, it is very natural that random duality theory was studied at the previous time in \cite{TXG-dual,TXG-Sur,TXG-XXC}, where many basic results and useful techniques were already obtained. Before 2009, only the $(\varepsilon, \lambda)$--topology was available, so the work in \cite{TXG-dual,TXG-Sur,TXG-XXC} was carried out under this topology, where the family of $L^0$--seminorms plays a key role. In the subsection, we will establish some basic results on random duality theory with respect to the locally $L^0$--convex topology in order to provide an enough framework for the theory of $RLC$ modules and its financial applications, which is motivated from the study of $L^0$--barreled modules. For a sake of convenience, let us introduce the following:

\begin{definition} (See \cite{FKV}.) Let $(E, {\mathcal T})$ be a locally $L^0$--convex $L^0({\mathcal F}, K)$--module. An $L^0$--balanced, $L^0$--absorbent, $L^0$--convex and closed subset of $E$ is called an $L^0$--barrel. $(E, {\mathcal T})$ is called an $L^0$--barreled module if every $L^0$--barrel is a neighborhood of $\theta$.
\end{definition}

Further, D. Filipovi\'{c}, M. Kupper and N. Vogelpoth established the continuity and subdifferentiability theorems for a proper lower semicontinuous $L^0$--convex function defined on an $L^0$--barreled module in \cite{FKV}. Then we naturally ask: what is a characterization for a locally $L^0$--convex $L^0({\mathcal F}, K)$--module to be $L^0$--barreled? In particular, is a ${\mathcal T}_c$--complete $RN$ module $L^0$--barreled under the locally $L^0$--convex topology induced by its $L^0$--norm? Specially, is $L^p_{\mathcal F}(\mathcal E)$ $L^0$--barreled? We at once realized that these problems are rather complicated. Our study leads to the following:

\begin{definition} Let $(E, {\mathcal T})$ be a locally $L^0$--convex $L^0({\mathcal F}, K)$--module. $(E,$ ${\mathcal T})$ is called an $L^0$--pre-barreled module if every $L^0$--barrel with the countable concatenation property is a neighborhood of $\theta$.
\end{definition}

The notion of an $L^0$--pre-barreled module is weaker than that of an $L^0$--barreled module. Although we have not yet provided a perfect characterization for a locally $L^0$--convex $L^0({\mathcal F}, K)$--module to be $L^0$--barreled or $L^0$--pre-barreled, in Subsection 3.3.3 we will prove that every $RLC$ module $(E, {\mathcal P})$ such that $E$ has the countable concatenation property is $L^0$--pre-barreled under the locally $L^0$--convex topology induced by ${\mathcal P}$ iff ${\mathcal T}_c$ is the topology of random uniform convergence on the family of $\sigma_c(E^\ast_c, E)$--bounded sets of $E^\ast_c$. In particular, every ${\mathcal T}_c$--complete $RN$ module $(E, \|\cdot\|)$ is $L^0$--pre-barreled under its locally $L^0$--convex topology if $E$ has the countable concatenation property. Further, in Section 3.4 we will establish the continuity and subdifferentiability theorems for a proper lower semicontinuous $L^0$--convex function defined on an $L^0$--pre-barreled module $(E, {\mathcal T})$ such that $E$ has the countable concatenation property. Thus the notion of an $L^0$--pre-barreled module is more suitable for the study of conditional risk measures.

\subsubsection{Random compatible locally $L^0$--convex topology}

Let $X$ and $Y$ be two $L^0({\mathcal F}, K)$--modules. A mapping $\langle\cdot, \cdot\rangle\colon X\times Y\to L^0({\mathcal F}, K)$ is said to be $L^0({\mathcal F}, K)$--bilinear function ( briefly, $L^0$--bilinear function if there is not any possible confusion ) if both $\langle x, \cdot\rangle\colon Y\to L^0({\mathcal F}, K)$ and $\langle\cdot, y\rangle\colon X\to L^0({\mathcal F}, K)$ are $L^0$--linear functions for all $x\in X$ and $y\in Y$.

\begin{definition} (See \cite{TXG-dual,TXG-Sur,TXG-XXC}.) Two $L^0({\mathcal F}, K)$--modules $X$ and $Y$ are called a random duality pair over $K$ with base $(\Omega, {\mathcal F}, P)$ with respect to the $L^0$--bilinear function $\langle\cdot,\cdot\rangle\colon X\times Y\to L^0({\mathcal F}, K)$ if the following are satisfied:\\
(1). $\langle x, y\rangle=0$ for all $y\in Y$ iff $x=\theta$;\\
(2). $\langle x, y\rangle=0$ for all $x\in X$ iff $y=\theta$.
\end{definition}

Usually, if $X,~Y$ and $\langle\cdot,\cdot\rangle$ satisfy Definition 3.19, then we simply say that $\langle X, Y\rangle$ is a random duality pair over $K$ with base $(\Omega, {\mathcal F}, P)$. Let $X^\#$ denote the $L^0({\mathcal F}, K)$--module of $L^0$--linear functions from an $L^0({\mathcal F}, K)$--module $X$ to $L^0({\mathcal F}, K)$. It is clear that $X^\#$ has the countable concatenation property. If $\langle X, Y\rangle$ is a random duality pair, we always identify each $x\in X$ with $\langle x, \cdot\rangle\in Y^\#$, namely regard $X$ as a submodule of $Y^\#$, thus for any subset $G\subset X$, we always use $H_{cc}(G)$ for the countable concatenation hull of $G$ in $Y^\#$, which would not cause any possible confusion.

\begin{definition} Let $\langle X, Y\rangle$ be a random duality pair over $K$ with base $(\Omega, {\mathcal F}, P)$. A family ${\mathcal P}$ of $L^0$--seminorms on $X$ is called a random compatible family with $Y$ with respect to the locally $L^0$--convex topology ${\mathcal T}_c$ induced by ${\mathcal P}$ if $(X, {\mathcal P})$ becomes an $RLC$ module over $K$ with base $(\Omega, {\mathcal F}, P)$ such that $(E, {\mathcal P})^\ast_c=Y$, in which case we also say that ${\mathcal T}_c$ is a random compatible locally $L^0$--convex topology with $Y$.
\end{definition}

\begin{remark} In \cite{TXG-XXC}, a family ${\mathcal P}$ of $L^0$--seminorms is called a random compatible family with $Y$ with respect to the $(\varepsilon, \lambda)$--topology ${\mathcal T}_{\varepsilon, \lambda}$ induced by ${\mathcal P}$ if $(X, {\mathcal P})$ becomes an $RLC$ module such that $(E, {\mathcal P})^\ast_{\varepsilon, \lambda}=Y$.
\end{remark}

Let $\langle X, Y\rangle$ be a random duality pair and $\sigma(X, Y)=\{|\langle\cdot, y\rangle|\colon y\in Y\}$, then $\sigma(X, Y)$ is a family of $L^0$--seminorms on $X$ such that $(X, \sigma(X, Y))$ becomes an $RLC$ module. In the sequel, $\sigma_c(X, Y)$ and $\sigma_{\varepsilon, \lambda}(X, Y)$ always denote the locally $L^0$--convex topology and the $(\varepsilon, \lambda)$--topology induced by $\sigma(X, Y)$, respectively.

\begin{theorem} Let $\langle X, Y\rangle$ be a random duality pair over $K$ with base $(\Omega, {\mathcal F}, P)$. Then $\sigma_c(X, Y)$ is a random compatible topology with $Y$.
\end{theorem}

Proof of Theorem 3.22 needs Lemma 3.23 below.

\begin{lemma} $($See \cite{TXG-dual,TXG-XXC}$.)$ Let $E$ be an $L^0({\mathcal F}, K)$--module. If $f_1,f_2,\cdots,f_n$ and $g$ are $n+1$ $L^0$--linear functions from $E$ to $L^0({\mathcal F}, K)$, then there are $\xi_1,\xi_2,\cdots,\xi_n\in L^0({\mathcal F}, K)$ such that $g=\sum^n_{i=1}\xi_if_i$ iff $\bigcap^n_{i=1}N(f_i)\subset N(g)$, where $N(f_i)=\{x\in E~|~f_i(x)=0\}$ ( $1\leq i\leq n$ ) and $N(g)=\{x\in E~|~g(x)=0\}$.
\end{lemma}

We can now prove Theorem 3.22.

\newproof{pot1}{Proof of Theorem 3.22}

\begin{pot1}

Since it is obvious that $Y\subset (X, \sigma(X, Y))^\ast_c$, it remains to prove that $(X, \sigma(X, Y))^\ast_c \subset Y$. Let $f\in (X, \sigma(X, Y))^\ast_c$, then by Proposition 2.19 there are $\xi\in L^0_+({\mathcal F})$ and $y_1,y_2,\cdots, y_n\in Y$ such that $|f(x)|\leq \xi (\bigvee\{|\langle x, y_i\rangle|\colon 1\leq i\leq n\})$ for all $x\in X$. By Lemma 3.23, there are $\xi_1,\xi_2,\cdots,\xi_n\in L^0({\mathcal F}, K)$ such that $f=\sum^n_{i=1}\xi_i\langle x, y_i\rangle$ for all $x\in X$. Let $y=\sum^n_{i=1}\xi_i y_i$, then $f=y\in Y$.  \hfill $\square$

\end{pot1}

\begin{remark} In \cite{TXG-XXC}, since we employed the $(\varepsilon, \lambda)$--topology, we proved that $(X,\sigma(X, Y))^\ast_{\varepsilon, \lambda}=H_{cc}(Y)$, which motivates us to find out Theorem 2.30. In \cite{TXG-XXC}, $Y$ is regular with respect to $X$ if, for each sequence $\{y_n, n\in N\}$ and each countable partition $\{A_n, n\in N\}$ of $\Omega $ to ${\mathcal F}$, there is $y\in Y$ such that $\langle x, y\rangle=\sum^{\infty}_{n=1}{\tilde I}_{A_n}\langle x, y_n\rangle$ for all $x\in X$, which implies that ${\tilde I}_{A_n}\langle x, y\rangle={\tilde I}_{A_n}\langle x, y_n\rangle$ for all $x\in X$ and $n\in N$, namely ${\tilde I}_{A_n}y={\tilde I}_{A_n}y_n$ for each $n\in N$, that is to say, $Y$ has the countable concatenation property. Thus, for a random duality pair $\langle X, Y\rangle$, `` $Y$ is regular '' and `` $Y$ has the countable concatenation property '' are the same thing.
\end{remark}

\begin{theorem} Let $\langle X, Y\rangle$ be a random duality pair. Then there is a strongest one in all the random compatible locally $L^0$--convex topologies with $Y$.
\end{theorem}

\begin{proof} By Corollary 2.33, one can see that the proof is completely similar to the one of the corresponding classical case, so is omitted. \hfill \hfill $\square$
\end{proof}

Let $\langle X, Y\rangle$ be a random duality pair. For a subset $A$ of $X$, $A^0:=\{y\in Y|~|\langle a, y\rangle|\leq 1~\hbox{for all $a\in A$~}\}$ is called the polar of $A$ in $Y$. Similarly, one can define the polar of a subset $B$ of $Y$ in $X$.

\begin{theorem} $(${\bf Random bipolar theorem}$)$ Let $\langle X, Y\rangle$ be a random duality pair over $K$ with base $(\Omega, {\mathcal F}, P)$ such that $X$ has the countable concatenation property. Then, for any subset $A$ of $X$, we have that $A^{00}=[H_{cc}(\Gamma (A))]^-_{\mathcal T}$ for each random compatible topology ${\mathcal T}$ with $Y$, where $\Gamma(A)$ denotes the $L^0$--balanced and $L^0$--convex hull of $A$ and $[H_{cc}(\Gamma (A))]^-_{\mathcal T}$ the ${\mathcal T}$--closure of $H_{cc}(\Gamma (A))$.
\end{theorem}

\begin{proof} Since $(X, {\mathcal T})^\ast_c=Y$, ${\mathcal T}\supset \sigma_c(X, Y)$. On the other hand, it is obvious that $A^{00}$ is an $L^0$--balanced, $L^0$--convex and $\sigma_c(X, Y)$--closed set with the countable concatenation property, so $A^{00}\supset [H_{cc}(\Gamma (A))]^-_{\sigma_c(X, Y)}\supset [H_{cc}(\Gamma (A))]^-_{\mathcal T}$. If there is $x\in A^{00}\setminus [H_{cc}(\Gamma (A))]^-_{\mathcal T}$, then by Corollary 3.9 and Remark 3.11 there is $y\in (X, {\mathcal T})^\ast_c=Y$ such that $|\langle x, y\rangle|\nleqslant 1$ and $\bigvee\{|\langle a, y\rangle|\colon a\in A\}\leq 1$, which is impossible. \hfill \hfill $\square$
\end{proof}

\begin{remark} The classical bipolar theorem is an elegant result and hence frequently employed in the study of classical duality theory. However, the random bipolar theorem under the locally $L^0$--convex topology, namely Theorem 3.26 has the complicated form and also requires $X$ to have the countable concatenation property, so we do our best to avoid the use of it except in Subsection 3.3.3 where we are forced to use it to characterize a class of $L^0$--pre-barreled modules. In \cite{TXG-XXC} we proved a random bipolar theorem under the $(\varepsilon, \lambda)$--topology with the same shape as the classical bipolar theorem, but the countable concatenation property of $Y$ is required. To sum up, we are always forced to look for new methods in order to obtain some most refined results on random duality theory.
\end{remark}

It is time for us to speak of random compatible invariants. Corollary 3.9 shows that any closed $L^0$--convex sets ( in particular, any $L^0$--barrels ) with the countable concatenation property are random compatible invariants with respect to every random duality pair. Theorem 2.35 shows that the same is true for bounded sets in the sense of the locally $L^0$--convex topology.

\subsubsection{Random admissible topology}

\begin{definition} Let $\langle X, Y\rangle$ be a random duality pair over $K$ with base $(\Omega, {\mathcal F}, P)$ and ${\mathcal A}$ a family of $\sigma_c(Y, X)$--bounded sets of $Y$. For any $ A\in {\mathcal A}$, the $L^0$--seminorm $\|\cdot\|_A\colon X\to L^0_+({\mathcal F})$ is defined by $\|x\|_A=\bigvee\{|\langle x, a\rangle|\colon a\in A\}$ for all $x\in X$. Then the locally $L^0$--convex topology induced by the family $\{\|\cdot\|_A\colon A\in {\mathcal A}\}$ of $L^0$--seminorms, denoted by ${\mathcal T}_{\mathcal A}$, is called the topology of random uniform convergence on ${\mathcal A}$.
\end{definition}

\begin{definition} Let $\langle X, Y\rangle$ and ${\mathcal A}$ be the same as in Definition 3.28. ${\mathcal T}_{\mathcal A}$ is said to be random admissible if ${\mathcal T}_{\mathcal A}\supset \sigma_c(X, Y)$, in which case ${\mathcal A}$ is said to be random admissible. If $\mathcal{T}_{\mathscr{A}}$ is random compatible, namely $(X,\mathcal{T}_{\mathscr{A}})^{\ast}_c=Y$, then $\mathscr{A}$ is also said to be random compatible.
\end{definition}

As usual, let us first study $\mathcal{T}_{\mathscr{A}}$.

\begin{proposition} Let $\langle X,Y\rangle$ and $\mathscr{A}$ be the same as in Definition 3.29. Then the following are equivalent:

\noindent (1). $\mathcal{T}_{\mathscr{A}}$ is Hausdorff.

\noindent (2). $\bigcup\mathscr{A}:=\bigcup\{A:A\in\mathscr{A}\}$ is total, namely $\langle x,y\rangle=0$ for all $y\in \bigcup\mathscr{A}$ implies $x=\theta$.

\noindent (3). $Span\mathscr{A}$ (the submodule generated by $\bigcup\mathscr{A}$) is $\sigma_{\varepsilon,\lambda}(Y,X)$--dense in $Y$.

\noindent (4). $H_{cc}(Span\mathscr{A})$ is $\sigma_{c}(H_{cc}(Y),X)$--dense in $H_{cc}(Y)$.
\end{proposition}

\begin{proof} (1)$\Leftrightarrow$(2), (3)$\Rightarrow$(2) and (4)$\Rightarrow$(2) are all obvious.

(2)$\Rightarrow$(3). By Corollary 3.10 and Remark 3.11, one can complete the proof by the same method as used in the classical case.

(3)$\Rightarrow$(4). By applying (2) of Theorem 2.28 to $Span\mathscr{A}$ and $(H_{cc}(Y),\sigma_{c}(H_{cc}(Y),X))$, we have the following relations:$$[H_{cc}(Span\mathscr{A})]^{-}_{\sigma_{c}(H_{cc}(X),Y)}$$$$=[H_{cc}(Span\mathscr{A})]^{-}_{\sigma_{\varepsilon,\lambda}(H_{cc}(X),Y)}$$
$$=[Span\mathscr{A}]^{-}_{\sigma_{\varepsilon,\lambda}(H_{cc}(X),Y)}~~~~~~~$$$$\supset H_{cc}(Y)~~~~~~~~~~~~~~~~~~~~~~~~$$
(by applying (3) to $H_{cc}(Y)$). \hfill $\square$
\end{proof}

Although random bipolar theorem does not necessarily hold for all random duality pairs, (2) of Lemma 3.31 below can complement this point.

\begin{lemma}\noindent Let $\langle X,Y\rangle$ be a random duality pair. Then we have:

\noindent (1). $A\subset Y$ is $\sigma_{c}(Y,X)$--bounded iff $A^{0}$ is a $\sigma_{c}(X,Y)$--$L^0$--barrel.

\noindent (2). For any $\sigma_{c}(Y,X)$--bounded set $A\subset Y$, $\|\cdot\|_B=\|\cdot\|_{B^{00}}$ (and hence $B^{00}$ is also $\sigma_{c}(Y,X)$--bounded), where $\|\cdot\|_B$ and $\|\cdot\|_{B^{00}}$ are defined as in Definition 3.27.
\end{lemma}

\begin{proof} (1) is clear.

(2). Since $B^{00}\supset B$, it is obvious that $\|\cdot\|_{B^{00}}\geq\|\cdot\|_B$. Conversely, if $\|x\|_B\leq 1$, then $x\in B^{0}$, and hence $\|x\|_{B^{00}}\leq 1$, which implies $\|\cdot\|_{B^{00}}\leq\|\cdot\|_B$.

\hfill\hfill $\square$
\end{proof}

\begin{definition} Let $\langle X,Y\rangle$ be a random duality pair over $K$ with base $(\Omega,\mathcal{F},P)$ and $\mathscr{B}$ a family of $\sigma_{c}(Y,X)$--bounded sets of $Y$. $\mathscr{B}$ is saturated if the following are satisfied:

\noindent (a). If $A\subset B$ for some $B\in\mathscr{B}$ , then $A\in \mathscr{B}$;

\noindent (b). $A,B\in\mathscr{B}\Rightarrow A\bigcup B\in\mathscr{B}$;

\noindent (c). $B\in\mathscr{B}\Rightarrow B^{00}\in\mathscr{B}$;

\noindent (d). $\lambda B\in\mathscr{B}$ for all $\lambda\in L^0(\mathcal{F},K)$ and $B\in\mathscr{B}$.
\end{definition}

In the classical definition of a saturated family (which amouts to the case when $\mathcal{F}=\{\Omega,\emptyset\}$), the above (c) in Definition 3.32 is defined as ``$B\in\mathscr{B}\Rightarrow [\Gamma(B)]^{-}_{\sigma(Y,X)}\in\mathscr{B}$". But, generally, we only have the relation that $B^{00}\supset[\Gamma(B)]^{-}_{\sigma_c(Y,X)}$. Although the random bipolar theorem shows that $B^{00}=[H_{cc}(\Gamma(B))]^{-}_{\sigma_{c}(Y,X)}$ if $Y$ has the countable concatenation property, we would like to introduce the notion of a saturated family for an arbitrary random duality pair, so we choose Definition 3.32 to meet all our requirements.

Let $\langle X,Y\rangle$ be a random duality pair, in this paper we always denote by $\mathscr{B}(Y,X)$ the family of $\sigma_{c}(Y,X)$--bounded sets of $Y$ and $\beta(X,Y)=\mathcal{T}_{\mathscr{B}(Y,X)}$. By (2) of Lemma 3.31, $\mathscr{B}(Y,X)$ is saturated. It is also obvious that $\beta(X,Y)$ is the strongest random admissible topology.

For a family $\mathscr{A}$ of $\sigma_{c}(Y,X)$--bounded sets of $Y$, $\mathscr{A}^{s}$ denotes the saturated hull of $\mathscr{A}$, namely the smallest saturated family containing $\mathscr{A}$. It is easy to see that $\mathscr{A}^{s}=\{B\subset Y~|~$ there are $\lambda_1,\lambda_2,\cdots,\lambda_n\in L^{0}(\mathcal{F},K)$ and $A_1,A_2,\cdots,A_n\in\mathscr{A}$ such that $B\subset(\bigcup_{i=1}^{n}\lambda_i A_i)^{00}\}$, again by (2) of Lemma 3.31 one can easily see that $\mathcal{T}_{\mathscr{A}}=\mathcal{T}_{\mathscr{A}^s}$.

\begin{proposition} Let $\langle X,Y\rangle$ be a random duality pair over $K$ with base $(\Omega,\mathcal{F},P)$ and $\mathscr{A}$ and $\mathscr{B}$ two family of $\sigma_c(Y,X)$--bounded sets of $Y$ such that $\mathscr{B}$ is saturated. Then, $\mathcal{T}_{\mathscr{A}}\subset\mathcal{T}_{\mathscr{B}}$ iff $\mathcal{A}\subset\mathcal{B}$.
\end{proposition}

\begin{proof} If $\mathcal{T}_{\mathscr{A}}\subset\mathcal{T}_{\mathscr{B}}$, then for each $A\in\mathscr{A}$ there are $\xi\in L^{0}_{+}(\mathcal{F})$ and a finite subfamily $\{B_i~|~1\leq i\leq n\}$ of $\mathscr{B}$ such that $\|x\|_A\leq\xi(\bigvee\{\|x\|_{B_i}:1\leq i\leq n\})=\|x\|_B$ for all $x\in X$, where $B=\xi(\bigcup_{i=1}^{n}B_i)\in \mathscr{B}$. Thus $A\subset A^{00}\subset B^{00}\in\mathscr{B}$, which has showed that $A\in\mathscr{B}$. The converse is obvious.
\hfill $\square$
\end{proof}

\begin{corollary} Let $\langle X,Y\rangle$ be a random duality pair and $\mathscr{A}$ is a saturated family of $\sigma_{c}(Y,X)$--bounded sets of $Y$.  Then $\mathcal{T}_{\mathscr{A}}$ is random admissible iff $\bigcup\mathscr{A}=Y$.
\end{corollary}

\begin{proof} Let $Y_f$ denote the family of finite subsets of $Y$, then $\sigma_{c}(X,Y)=\mathcal{T}_{Y_f}$. So, $\mathcal{T}_{\mathscr{A}}$ is random admissible iff $\mathcal{T}_{Y_f}\subset\mathcal{T}_{\mathscr{A}}$ iff $Y_f\subset \mathscr{A}$ iff $\bigcup \mathscr{A}=Y$. \hfill $\square$
\end{proof}

\begin{theorem} Let $\langle X,Y\rangle$ be a random duality pair over $K$ with base $(\Omega,\mathcal{F},P)$. Then a locally $L^{0}$--convex topology $\mathcal{T}$ on $X$ is a topology of random uniform convergence iff $\mathcal{T}$ has a local base $\mathcal{B}$ at $\theta$ such that each $U\in\mathcal{B}$ is a $\sigma_c(X,Y)$--$L^{0}$--barrel.
\end{theorem}

\begin{proof}
If $\mathcal{T}=\mathcal{T}_{\mathscr{A}}$ for some family $\mathscr{A}$ of $\sigma_c(X,Y)-$bounded sets of $Y$, we can suppose that $\mathscr{A}$ is saturated, then $\mathcal{B}=\{A^{0}:A\in\mathscr{A}\}$ is a local base at $\theta$ of $\mathcal{T}$ such that each $A^{0}$ is a $\sigma_{c}(X,Y)$--$L^{0}$--barrel.

Conversely, let $\mathcal{T}$ have a local base $\mathcal{B}$ at $\theta$ such that each $U\in \mathcal{B}$ is a $\sigma_c(X,Y)$--$L^0$--barrel. Let $\mathscr{A}=\{U^{0}:U\in\mathscr{B}\}$, then each $U^{0}$ is $\sigma_{c}(Y,X)$--bounded since $(U^{0})^{0}\supset U$ is $L^0-$absorbent. Further, we show that $\mathcal{T}=\mathcal{T}_{\mathscr{A}}$ as follows. Since $\mathcal{T}_{\mathscr{A}}$ is induced by $\{\|\cdot\|_{A}:A\in\mathscr{A}\}$ and , for each $U^{0}\in\mathscr{A}$, $\{x\in X~|~\|x\|_{U^{0}}\leq 1\}=U^{00}\supset U$, which shows that $\|\cdot\|_{U^0}$ is $\mathcal{T}$--continuous, namely $\mathcal{T}_{\mathscr{A}}\subset \mathcal{T}$. On the other hand, for each $U\in\mathcal{B}$, $U\subset U^{00}=(U^{0})^{0}$, namely each element $(U^{0})^{0}$ of a local base at $\theta$ of $\mathcal{T}_{\mathscr{A}}$ is a $\mathcal{T}$--neighborhood of $\theta$, that is to say, $\mathcal{T}_{\mathscr{A}}\subset\mathcal{T}$. \hfill $\square$
\end{proof}

In the classical case, by the classical bipolar theorem it can be easily established that $\{A^{0}:A\in\mathscr{B}(Y,X)\}$ as the local base at $\theta$ of $\beta(X,Y)$ is exactly the family of $\sigma(X,Y)$--barrels. However, in the random setting, we do not know if $\{A^{0}:A\in\mathscr{B}(Y,X)\}$ as the local base at $\theta$ of $\beta(X,Y)$ is still the family of $\sigma_{c}(X,Y)$--$L^0$--barrels, we only know that for each $\sigma_{c}(X,Y)$--$L^0$--barrel $U$ there is $A(=U^0)\in\mathscr{B}(Y,X)$ such that $U\subset A^{0}$. So, we remind the reader of the following useful result:

\begin{theorem} Let $\langle X,Y\rangle$ be a random duality pair such that $X$ has the countable concatenation property. Then the family of  $\sigma_c(X,Y)$--$L^0$--barrels with the countable concatenation property forms a local base at $\theta$ of $\beta(X,Y)$.
\end{theorem}

\begin{proof} By the countable concatenation property of $X$, it is easy to see that $A^0$ has the countable concatenation property for each $A\in\mathscr{B}(Y,X)$, and hence also a $\sigma_{c}(X,Y)$--$L^0$--barrel with this property. On the other hand, for each $\sigma_{c}(X,Y)$--$L^0$--barrel $U$ with the countable concatenation property, then by Theorem 3.26 we have that $U=U^{00}=(U^{0})^{0}$. Since $U^{0}\in\mathscr{B}(Y,X)$, $U\in\{A^{0}:A\in\mathscr{B}(Y,X)\}$. To sum up, the family of $\sigma_{c}(X,Y)$--$L^{0}$--barrels with the countable concatenation property is exactly the local base $\{A^{0}:A\in\mathscr{B}(Y,X)\}$. \hfill $\square$
\end{proof}

Theorem 3.37 below shows that the study of random admissible topology is of universal interest in the theory of $RCL$ modules.

\begin{theorem} Let $(X,\mathcal{P})$ be an $RLC$ module over $K$ with base $(\Omega,\mathcal{F},P)$ and $\mathcal{E}$ the family of all the subsets $E$ of $X_{c}^{\ast}$ such that $E$ is equicontinuous from $(X,\mathcal{T}_c)$ to $L^{0}(\mathcal{F},K)$ endowed with the locally $L^{0}$--convex topology induced by $|\cdot|$. Then $\mathcal{T}_{c}=\mathcal{T}_{\mathcal{E}}$, where we consider the natural pairing $\langle X,X_{c}^{\ast}\rangle$, then $\mathcal{T}_{\mathcal{E}}$ is, clearly, a random admissible topology.
\end{theorem}

\begin{proof} It is clear that $E\in\mathcal{E}$ iff there are $\xi\in L^{0}_{+}(\mathcal{F})$ and a finite subset $\mathcal{Q}$ of $\mathcal{P}$ such that $\|x\|_{E}:=\bigvee\{|f(x)|:f\in E\}\leq \xi(\bigvee\{\|x\|:\|\cdot\|\in\mathcal{Q}\})$ for all $x\in X$, so $\mathcal{T}_{\mathcal{E}}\subset\mathcal{T}_{c}$.

Conversely, for each $\|\cdot\|\in\mathcal{P}$, let $E=\{f\in X^{\ast}_{c}~|~|f(x)|\leq \|x\|$ for all $x\in X\}$, then from the random Hahn-Banach theorem of \cite{TXG-JFA} one can easily see that $\|\cdot\|=\|\cdot\|_E$, so $\mathcal{T}_{c}\subset\mathcal{T}_{\mathcal{E}}$. \hfill $\square$
\end{proof}

\begin{corollary} Let $\langle X,Y\rangle$ be a random duality pair over $K$ with base $(\Omega,\mathcal{F},P)$. Then every random compatible topology  $\mathcal{T}$ on $X$ is random admissible.
\end{corollary}

\begin{proof} By Definition 3.20, there is a family $\mathcal{P}$ of $L^0$--seminorms on $X$ such that $(X,\mathcal{P})$ becomes an $RLC$ module over $K$ with base $(\Omega,\mathcal{F},P)$ and $(X,\mathcal{P})^{\ast}_{c}=Y$, $\mathcal{T}$ is just induced by $\mathcal{P}$, at which time $\langle X,Y\rangle$ is exactly $\langle X,X^{\ast}_{c}\rangle$ and $\mathcal{T}=\mathcal{T}_{\mathcal{E}}$ by Theorem 3.37. \hfill $\square$
\end{proof}

The proof of Theorem 3.39 below (namely the resonance theorem) is omitted since it is the same as that of the classical case.

\begin{theorem} Let $(E,\mathcal{P})$ be an $RLC$ module over $K$ with base $(\Omega,\mathcal{F},P)$ and $H\subset E^{\ast}_{c}$. Then we have the following:

\noindent $(1)$. If $(E,\mathcal{T}_c)$ is $L^0$--barreled, then $H$ is equicontunuous from $(E,\mathcal{T}_c)$ to $(L^{0}(\mathcal{F},K),\mathcal{T}_c)$ iff $H$ is $\sigma_c(E^{\ast}_{c},E)$--bounded.

\noindent $(2)$. If $(E,\mathcal{T}_c)$ is $L^0$--pre-barreled and $E$ has the countable concatenation property, then $H$ is equicontinuous from $(E,\mathcal{T}_c)$ to $(L^{0}(\mathcal{F},K),\mathcal{T}_c)$ iff $H$ is $\sigma_c(E^{\ast}_{c},E)$--bounded.
\end{theorem}

In the classical case, for a locally convex space $(E,\mathcal{T})$, a subset $H\subset E^{\ast}$ is equicontinuous, then it must be $\sigma(E^{\ast},E)$--relatively compact. However the classical Banach-Alaoglu theorem universally fails to hold in the case of $RN$ modules under the $(\varepsilon,\lambda)$--topology (cf. \cite{TXG-Alao}), the same, of course, occurs for the locally $L^{0}$--convex topology, so we can not generalize the construction of the classical Mackey topology to the random setting.

\subsubsection{A characterization for a random locally convex module to be $L^0$--pre-barreled}

The main result of this subsection is Theorem 3.40 below.

\begin{theorem}  Let $(E,\mathcal{P})$ be an $RLC$ module over $K$ with base $(\Omega,\mathcal{F},P)$ such that $E$ has the countable concatenation property. Then $(E,\mathcal{T}_c)$ is $L^0$--pre-barreled iff $\mathcal{T}_c=\beta(E,E^{\ast}_c)$.
\end{theorem}

\begin{proof} $\mathcal{T}_c$ has a local base at $\theta$ consisting of $\{N_{\theta}(\mathcal{Q},\varepsilon)~|~\mathcal{Q}\in\mathcal{P}_f,\varepsilon\in L^{0}_{++}(\mathcal{F})\}$, where $N_{\theta}(\mathcal{Q},\varepsilon)=\{x\in E~|~\|x\|_{\mathcal{Q}}\leq\varepsilon\}$. It is obvious that every $N_{\theta}(\mathcal{Q},\varepsilon)$ is an $\mathcal{T}_c$--$L^0$--barreled with the countable concatenation property. Thus, if $\mathcal{T}_c$ is $L^0$--pre-barreled, then the family of $\mathcal{T}_c$--$L^0$--barrels with the countable concatenation property forms a local base at $\theta$ of $\mathcal{T}_c$. By Corollary 3.9, $L^0$--barrels with the countable concatenation property are random compatible invariants, namely the family of $\mathcal{T}_c$--$L^0$--barrels with the countable concatenation property coincides with the family of $\sigma_c(E,E_c^{\ast})$--$L^0$--barrels with the countable concatenation property, so $\mathcal{T}_c=\beta(E,E_c^{\ast})$ by Theorem 3.36 if $\mathcal{T}_c$ is $L^0$--pre-barreled.

Conversely, if $\mathcal{T}_c=\beta(E,E_c^{\ast})$, then ,since every $\mathcal{T}_c$--$L^0$--barrel with the countable concatenation property is $\sigma_c(E,E_c^{\ast})$--$L^0$--barrel by Corollary 3.9, and hence a $\beta(E,E^{\ast}_c)$--neighborhood of $\theta$ by Theorem 3.36, namely a $\mathcal{T}_c$--neighborhood of $\theta$, that is to say, $(E,\mathcal{T}_c)$ is $L^0$--pre-barreled. ~~~~~~~\hfill $\square$
\end{proof}

\begin{corollary} Let $(E,\|\cdot\|)$ be a $\mathcal{T}_c$--complete $RN$ module over $K$ with $(\Omega,\mathcal{F},P)$ such that $E$ has the countable concatenation property. Then $(E,\mathcal{T}_c)$ is $L^0$--pre-barreled.
\end{corollary}

\begin{proof}  We only need to verify that $\mathcal{T}_c=\beta(E,E_c^{\ast})$.

First, the locally $L^0$--convex topology $\mathcal{T}_c$ induced by $\|\cdot\|$ is a random compatible topology with respect to the natural random duality pair $\langle E,E_c^{\ast}\rangle$, so $\mathcal{T}_c\subset\beta(E,E_c^{\ast})$ by Corollary 3.38.

Conversely, $\beta(E,E_c^{\ast})$ is induced by $\{\|\cdot\|_A:A\in \mathscr{B}(E_c^{\ast},E)\}$, please recall that $\|\cdot\|_A:E\rightarrow L^{0}_{+}(\mathcal{F})$ is given by $\|x\|_A=\bigvee\{|f(x)|:f\in A\}$ for all $x\in E$ and $A\in \mathscr{B}(E_c^{\ast},E)$. Thus we only need to prove that each $\|\cdot\|_A$ is continuous from $(E,\mathcal{T}_c)$ to $(L^0(\mathcal{F},K),\mathcal{T}_c)$. $A\in \mathscr{B}(E_c^{\ast},E)$ means that $\{f(x):f\in A\}$ is $\mathcal{T}_c$--bounded in $(L^0(\mathcal{F},K),\mathcal{T}_c)$ for each $x\in E$, then, by Theorem 2.34 $A$ is $\mathcal{T}_c$--bounded in $E_c^{\ast}$, namely a.s. bounded, and hence there is $\xi_A\in L^0_+(\mathcal{F})$ such that $\|f\|\leq\xi_A$ for all $f\in A$. This shows that $\|x\|_A=\bigvee\{|f(x)|:f\in A\}\leq\bigvee\{\|f\|\cdot\|x\|:f\in A\}\leq(\bigvee\{\|f\|:f\in A\})\|x\|\leq\xi_A\|x\|$ for all $x\in E$, namely $\beta(E,E_c^{\ast})\subset \mathcal{T}_c$. \hfill $\square$
\end{proof}

\begin{corollary} For each $p\in[1,+\infty]$, $(L^{p}_{\mathcal{F}}(\mathcal{E}),\mathcal{T}_c)$ is $L^0$--pre-barreled.
\end{corollary}

\begin{proof} Since $L^{p}_{\mathcal{F}}(\mathcal{E})$ is $\mathcal{T}_c$--complete and has the countable concatenation property, then it immediately follows from Corollary 3.41.  \hfill $\square$
\end{proof}

\subsection*{3.4. Continuity and subdifferentiability theorems in $L^0$--pre-barreled modules}

\quad We will state the results in this subsection under the framework of locally $L^0$--convex modules since the proofs of these results do not necessarily depend on the family of $L^0$--seminorms. Continuity and subdifferentiability theorems in $L^0$--barreled modules were already proved in \cite{FKV}. As shown in \cite{FKV}, the proofs in the random setting are very similar to those in the corresponding classical cases. Thus we omit some details of the proofs in order to save space for some discussions on the relation between the topological structure and stratification structure of a locally $L^0$--convex module.

In the subsection, a locally $L^0$--convex module (a topological $L^0$--module) means a locally $L^0$--convex $L^0({\mathcal{F},R})$--module (resp., a topological $L^0({\mathcal{F},R})$--module). To state our main results, let us first recall the following:

\begin{definition}(See \cite{FKV}.) Let $(E,\mathcal{T})$ be a locally $L^0$--convex module and $f:E\rightarrow\bar{L}^0(\mathcal{F})$ a proper $L^0$--convex function. $f$ is subdifferentiable at $x\in dom(f)$ if there is $u\in E^{\ast}_{c}$ such that $u(y-x)\leq f(y)-f(x)$ for all $y\in E$, at which time $u$ is called a subgradient of $f$ at $x$. The set of subgradients of $f$ at $x$ is denoted by $\partial f(x)$.
\end{definition}

We can now state our main results as follows:

\begin{theorem} Let $(E,\mathcal{T})$ be a real $L^0$--pre-barreled module such that $E$ has the countable concatenation property. Then a proper lower semicontinuous $L^0$--convex function $f:E\rightarrow \bar{L}^{0}(\mathcal{F})$ is continuous on $Int(dom(f)):=$ the interior of $dom(f)$, namely $f$ is continuous from $(Int(dom(f)),\mathcal{T})$ to $(L^0(\mathcal{F}),\mathcal{T}_c)$.
\end{theorem}

\begin{theorem} Let $(E,\mathcal{T})$ be a real $L^0$--pre-barreled module such that $E$ has the countable concatenation property. Then, for a proper lower semicontinuous $L^0$--convex function $f:E\rightarrow\bar{L}^{0}(\mathcal{F})$, $\partial f(x)\neq\emptyset$ for all $x\in Int(dom(f))$.
\end{theorem}

To prove Theorem 3.44, we needs the following known lemmas:

\begin{lemma} ($See$ \cite{FKV}.) Let $E$ be a topological $L^0$--module. If in some neighborhood of an element $x_{0}\in E$ a proper $L^0$--convex function $f:E\rightarrow\bar{L}^{0}(\mathcal{F})$ is bounded above by some $\xi_0\in L^{0}(\mathcal{F})$, then $f$ is continuous at $x_{0}$.
\end{lemma}

\begin{lemma} ($See$ \cite{FKV}.) Let $E$ be a topological $L^0$--module and $f:E\rightarrow\bar{L}^{0}(\mathcal{F})$ a proper $L^0$--convex function. Then the following statements are equivalent:

\noindent $(1)$. There is a nonempty open set $O\subset E$ on which $f$ is bounded above by some $\xi_0\in L^{0}(\mathcal{F})$.

\noindent $(2)$. $f$ is continuous on $Int(dom(f))$ and $Int(dom(f))\neq\emptyset$.
\end{lemma}

\begin{lemma} ($See$ \cite{FKV}.) Let $E$ be a topological $L^0$--module and $x\in E$. Then every proper $L^0$--convex function $f:Span_{L^0}(x)\rightarrow \bar{L}^{0}(\mathcal{F})$ is continuous on $Int(dom(f))$, where $Span_{L^0}(x)$ is the $L^0$--module spanned by $x$ and endowed with the relative topology.
\end{lemma}

We can now prove Theorem 3.44.

\newproof{pot2}{Proof of Theorem 3.44}

\begin{pot2}

Assume that there is $x_{0}\in Int(dom(f))$. By translation, we may assume $x_{0}=0$ and further take $Y_0\in L^{0}(\mathcal{F})$ such that $f(0)<Y_0$ on $\Omega$. Since $f$ is lower semicontinuous, the set $C:=\{x\in E~|~f(x)\leq Y_0\}$ is closed. Further, for all $x\in E$, the net $\{\frac{x}{Y}:Y\in L^{0}_{++}(\mathcal{F})\}$ converge to $\theta$. By Lemma 3.48, the restriction of $f$ to $Span_{L^0}(x)$ is continuous at $\theta$, hence $f(\frac{x}{Y})<Y_0$ on $\Omega$ for large $Y$, which means that $C$ is $L^0$--absorbent. Hence $C\bigcap(-C)$ is an $L^{0}$--barrel. Since $E$ has the countable concatenation property and $f$ has the local property, it is easy to observe that $C\bigcap(-C)$ is an $L^0$--barrel with the countable concatenation property and in turn a neighborhood of $\theta\in E$, so $f$ is continuous on $Int(dom(f))$ by Lemma 3.47.   \hfill $\square$
\end{pot2}

To prove Theorem 3.45, we need the following three lemmas.

Lemma 3.49 below is a slight generalization of Lemma 3.17 of \cite{TXG-JFA}, whereas their proofs are the same, so the proof of Lemma 3.49 is omitted.

\begin{lemma} Let $(E,\mathcal{T})$ be a locally $L^0$--convex $L^0(\mathcal{F},K)$--module and $A\in \cal{F}$ with $P(A)>0$. If $G$ and $M$ are an open set and a closed set of $E$, respectively, such that $\tilde{I}_A G+\tilde{I}_{A^{c}} G\subset G$ and $\tilde{I}_A M+\tilde{I}_{A^{c}} M\subset M$, then $\tilde{I}_A G$ is relatively open in $\tilde{I}_A E$ and $\tilde{I}_A M$ is relatively closed in $\tilde{I}_A E$.
\end{lemma}

From Lemma 3.49 one can see that $(\tilde{I}_{A}E,\mathcal{T}|_{\tilde{I}_{A}E})$ is still a locally $L^0$--convex $L^0(\mathcal{F}_A,K)$--module, where $\mathcal{F}_A=A\bigcap\mathcal{F}:=\{A\bigcap B~|~B\in\mathcal{F}\}$ is the $\sigma$--algebra of $(A,\mathcal{F}_A,P(\cdot|A))$.

\begin{lemma} Let $(E,\mathcal{T})$ be a locally $L^0$--convex module, $A\in\mathcal{F}$ with $P(A)>0$ and $f:E\rightarrow \bar{L}^0(\mathcal{F})$ a proper $L^0$--convex function. If, we define $f_A:\tilde{I}_{A}E\rightarrow \tilde{I}_{A}\bar{L}^0(\mathcal{F})$ by $f_A(\tilde{I}_{A}x)=\tilde{I}_{A}f(\tilde{I}_{A}x)$ for all $x\in E$, then we have:

\noindent $(1)$. For all $x\in dom(f)$, $\tilde{I}_A(x,f(x))\in \partial(epi(f_A))$, where $\partial(epi(f_A))$ denotes the boundary of $epi(f_A)$ in $(\tilde{I}_{A}E,\mathcal{T}|_{\tilde{I}_{A}E})\times(\tilde{I}_{A}L^0(\mathcal{F}),\mathcal{T}_c|$ $_{\tilde{I}_A L^0{(\mathcal{F})}})$. Here, let us recall that $\mathcal{T}_c$ is the locally $L^0$--convex topology on $L^0(\mathcal{F})$ induced by $|\cdot|$.

\noindent $(2)$. For all $x\in dom(f)$, $\tilde{I}_{A}(x,f(x))\not\in \tilde{I}_{A}(Int(epi(f)))$.
\end{lemma}

\begin{proof} (1). It is easy to see that $(x,f(x))\in\partial(epi(f))$ for all $x\in dom(f)$. By the local property of $f$, $\tilde{I}_{A}(x,f(x))=(\tilde{I}_{A}x, \tilde{I}_{A}f(x))=(\tilde{I}_{A}x,f_A(\tilde{I}_{A}x))$ for all $x\in E$. So, if we consider the corresponding problem in $(\tilde{I}_{A}E,\mathcal{T}|_{\tilde{I}_{A}E})$, then we have that $\tilde{I}_{A}(x,f(x))=(\tilde{I}_{A}x,f_A(\tilde{I}_{A}x))\in\partial(epi(f_A))$ for all $x\in dom(f)$.

(2). By the above (1), it is , of course, that $\tilde{I}_{A}(x,f(x))\not\in Int_A(epi(f_A))$, where $Int_A(epi(f_A))$ denotes the interior of $epi(f_A)$ in $(\tilde{I}_{A}E,\mathcal{T}|_{\tilde{I}_{A}E})\times(\tilde{I}_{A}L^0(\mathcal{F}),$ $\mathcal{T}_c|_{\tilde{I}_A  L^0{(\mathcal{F})}})$. It is obvious that $epi(f_A)=\tilde{I}_A(epi(f))$. By Lemma 3.49, $\tilde{I}_{A}(Int(epi(f)))$ is an open set in $(\tilde{I}_{A}E,\mathcal{T}|_{\tilde{I}_{A}E})\times(\tilde{I}_{A}L^0(\mathcal{F}),$ $\mathcal{T}_c|_{\tilde{I}_A L^0{(\mathcal{F})}})$, so $Int_A(epi(f_A))=Int_A(\tilde{I}_A epi(f))\supset \tilde{I}_A(Int(epi(f)))$, which implies that $\tilde{I}_A(x,$ $f(x))\not\in \tilde{I}_A(Int$ $(epi(f)))$ for all $x\in dom(f)$. \hfill $\square$
\end{proof}

Proof of Lemma 3.51 below is the same as that of Lemma 3.14 of \cite{FKV}, so is omitted.

\begin{lemma} Let $(E,\mathcal{T})$ be a locally $L^{0}$--convex module and $f:E\rightarrow\bar{L}^{0}(\mathcal{F})$ a proper lower semicontinuous $L^{0}$--convex function. Then $Int(epi(f))\neq\emptyset$ implies $Int(dom(f))\neq\emptyset$. Furthermore, if, in addition, $(E,\mathcal{T})$ is $L^0$--pre-barreled such that $E$ has the countable concatenation property, then $Int$ $(dom(f))$ $\neq\emptyset$ iff $Int(epi(f))\neq\emptyset$.
\end{lemma}

We can now prove Theorem 3.45.

\newproof{pot3}{Proof of Theorem 3.45}

\begin{pot3}

By Lemmas 3.50 and 3.51 together with Theorem 3.12, one can complete the proof by the same reasoning as in the proof of Theorem 3.7 of \cite{FKV}. \hfill $\square$
\end{pot3}

If the hypothesis ``that $(E,\mathcal{T})$ is $L^0$--pre-barreled module such that $E$ has the countable concatenation property" in Theorems 3.44 and 3.45 is replaced by the one ``that $(E,\mathcal{T})$ is $L^0$--barreled", then Theorems 3.44 and 3.45 change to Proposition 3.5 and Theorem 3.7 of \cite{FKV}, respectively, so we naturally present the following open problem:

\noindent {\bf Open problem:} If $E$ has the countable concatenation property, then is $(E,\mathcal{T})$ $L^0$--barreled if $(E,\mathcal{T})$ is $L^0$--pre-barreled?

\section{Random convex analysis over random locally convex modules under the $(\varepsilon,\lambda)$--topology}

\subsection{Lower semicontinuous $L^0$--convex functions under the $(\varepsilon,\lambda)$--topology}

\begin{definition} Let $(E,\mathcal{P})$ be an $RLC$ module over $R$ with base $(\Omega,\mathcal{F},P)$ and $f:E\rightarrow \bar{L}^{0}(\mathcal{F})$ a proper $L^{0}$--convex function. $f$ is $\mathcal{T}_{\varepsilon,\lambda}$--lower semicontinuous if $epi(f)$ is closed in $(E,\mathcal{T}_{\varepsilon,\lambda})\times(L^0(\mathcal{F}),\mathcal{T}_{\varepsilon,\lambda})$.
\end{definition}

As usual, let $(E,\mathcal{P})$ be an $RLC$ module over $R$ with base $(\Omega,\mathcal{F},P)$ and $f:E\rightarrow L^0(\mathcal{F})$, $f$ is $\mathcal{T}_{\varepsilon,\lambda}$--continuous if $f$ is continuous from $(E,\mathcal{T}_{\varepsilon,\lambda})$ to $(L^0(\mathcal{F}),\mathcal{T}_{\varepsilon,\lambda})$. If $\mathcal{T}_{\varepsilon,\lambda}$ is replaced by $\mathcal{T}_{c}$, then we have the notion of $\mathcal{T}_{c}$--continuity, which is just a special case of the notion of continuity for a function from a locally $L^0$--convex module $(E,\mathcal{T})$ to $(L^0(\mathcal{F}),\mathcal{T}_{c})$.

As to why we adopt Definition 4.1 for the $\mathcal{T}_{\varepsilon,\lambda}$--lower semicontinuity of an $L^0$--convex function, we interpret this as follows. If we define the $\mathcal{T}_{\varepsilon,\lambda}$--lower semicontinuity of a proper function $f:(E,\mathcal{P})\rightarrow\bar{L}^0({\mathcal{F}})$ via ``$\{x\in E~|~f(x)\leq r\}$ is $\mathcal{T}_{\varepsilon,\lambda}$--closed for all $r\in L^0(\mathcal{F})$", then this notion is too weak to meet some natural needs of other topics as in \cite{TXG-YJY}. If we define $f$ to be lower semicontinuous via ``$\underline{lim}_\alpha f(x_\alpha):=\bigvee_{\beta\in\Gamma}(\bigwedge_{\alpha\geq\beta}f(x_{\alpha}))\geq f(x)$ for all nets $\{x_{\alpha},\alpha\in\Lambda\}$ in $E$ such that it converges in the $(\varepsilon,\lambda)$--topology to some $x\in E$", the notion is, however, meaningless in the random setting, since we can construct a real $RLC$ module $(E,\mathcal{P})$ and a $\mathcal{T}_{\varepsilon,\lambda}$--continuous $L^0$--convex function $f$ from $E$ to $L^0(\mathcal{F})$, whereas $f$ is not a lower semicontinuous function under this notion. In fact, Definition 4.1 has been proved natural and fruitful, see \cite{TXG-YJY} or this subsection.

To connect the $L^0$--convexity and ordinary convexity, we introduce the following terminology:

\begin{definition} Let $(E,\mathcal{P})$ be an $RLC$ module over $R$ with base $(\Omega,\mathcal{E},P)$, $\mathcal{F}$ a $\sigma$--subalgebra of $\mathcal{E}$ and $f:E\rightarrow \bar{L}^0(\mathcal{E})$. $f$ is $L^0(\mathcal{F})$--convex if $f(\xi x+(1-\xi)y)\leq \xi f(x)+(1-\xi)f(y)$ for all $x,y\in E$ and $\xi\in L^0_+(\mathcal{F})$ with $0\leq\xi\leq 1$. $f$ is $\mathcal{F}$--local if $\tilde{I}_A f(x)=\tilde{I}_A f(\tilde{I}_A x)$ for all $x\in E$ and $A\in \mathcal{F}$.
\end{definition}

Let $(\Omega,\mathcal{E},P)$ be a probability space and $\mathcal{F}$ a $\sigma$--subalgebra of $\mathcal{E}$. For any $p\in[1,+\infty]$, $L^p(\mathcal{E}):=L^p(\Omega,\mathcal{E},P)$ denotes the ordinary function space. We can also similarly introduce the $L^0(\mathcal{F})$--convexity and $\mathcal{F}$--locality for a function $f$ from $L^p(\mathcal{E})$ to $\bar{L}^0(\mathcal{E})$.

\begin{proposition} Let $(E,\mathcal{P})$ be an $RLC$ module over $R$ with base $(\Omega,\mathcal{E},P)$, $f$ a function from $E$ to $\bar{L}^0(\mathcal{E})$ and $\mathcal{F}$ a $\sigma$--subalgebra of $\mathcal{E}$. Then we have the following statements:

\noindent (1). If $f$ is $L^0(\mathcal{F})$--convex, then $f$ is convex (namely, $f$ is a convex function in the ordinary sense) and $\mathcal{F}$--local.

\noindent (2). If $f$ is a $\mathcal{T}_{\varepsilon,\lambda}$--continuous function from $E$ to $L^0(\mathcal{E})$, then $f$ is $L^0(\mathcal{F})$--convex iff $f$ is convex and $\mathcal{F}$--local.
\end{proposition}

\begin{proof} (1). It is clear that $f$ is convex, and the proof of the $\mathcal{F}$--local property of $f$ is similar to the proof of necessity of Lemma 3.14.

(2). We only need to prove that $f$ is $L^0(\mathcal{F})$--convex if it is convex and $\mathcal{F}$--local. In fact, for any $\xi\in L_{+}^{0}(\mathcal{F})$ with $0\leq\xi\leq 1$, there is a nondecreasing sequence $\{\xi_n,n\in N\}$ of simple functions in $L^{0}_{+}(\mathcal{F})$ such that $0\leq\xi_n\leq\xi$ for each $n\in N$ and $\{\xi_n,n\in N\}$ converges to $\xi$ with respect to the essentially maximal norm $\|\cdot\|_{\infty}$. Since $f$ is convex and $\mathcal{F}$--local, it is easy to see that $f(\xi_n x+(1-\xi_n)y)\leq \xi_n f(x)+(1-\xi_n)f(y)$ for all $x,y\in E$ and $n\in N$, letting $n\rightarrow\infty$ will yield that $f(\xi x+(1-\xi)y)\leq \xi f(x)+(1-\xi)f(y)$. \hfill $\square$
\end{proof}

By the same reasoning, one can prove that for a continuous function $f$ from $(L^p(\mathcal{E}),\|\cdot\|_p)$ to $(L^0(\mathcal{E}),\mathcal{T}_{\varepsilon,\lambda})$ or from $(L^p(\mathcal{E}),\|\cdot\|_p)$ to $(L^r(\mathcal{E}),\|\cdot\|_r)$, $f$ is $L^0(\mathcal{F})$--convex iff $f$ is convex and $\mathcal{F}$--local, where $1\leq r,p\leq+\infty$.

The above discussions clarify the relation between $L^0$--convexity and ordinary convexity. Theorem 4.4 below further clarifies the relation between the $\mathcal{T}_{\varepsilon,\lambda}$--lower semicontinuity and $\mathcal{T}_{c}$--lower semicontinuity.

\begin{theorem} Let $(E,\mathcal{P})$ be an $RLC$ module over $R$ with base $(\Omega,\mathcal{F},P)$ and $f$ a local proper function from $E$ to $\bar{L}^0(\mathcal{F})$. Then we have the following statements:

\noindent $(1)$. If $E$ has the countable concatenation property, then $f$ is $\mathcal{T}_{\varepsilon,\lambda}$--lower semicontinuous iff $epi(f)$ is closed in $(E,\mathcal{T}_c)\times(L^{0}(\mathcal{F}),\mathcal{T}_c)$.

\noindent  $(2)$. If both $E$ and $\mathcal{P}$ have the countable concatenation property, then $f$ is $\mathcal{T}_{\varepsilon,\lambda}$--lower semicontinuous iff $f$ is $\mathcal{T}_{c}$--lower semicontinuous (namely $\{x\in E~|~f(x)\leq r\}$ is $\mathcal{T}_c$--closed for all $r\in L^0(\mathcal{F})$).
\end{theorem}

\begin{proof} (1). Since $E$ has the countable concatenation property and $f$ is local, $epi(f)$ has the countable concatenation property. Then, by Proposition 2.27, $epi(f)$ is $\mathcal{T}_{\varepsilon,\lambda}$--closed iff it is $\mathcal{T}_{c}$--closed.

(2). If bath $E$ and $\mathcal{P}$ have the countable concatenation property, then , by Proposition 3.15, $f$ is $\mathcal{T}_c$--lower semicontinuous iff $epi(f)$ is $\mathcal{T}_c$--closed, namely closed in $(E,\mathcal{T}_c)\times (L^0(\mathcal{F}),\mathcal{T}_c)$, which together with (1) ends the proof of (2). \hfill\hfill $\square$
\end{proof}

\subsection{Fenchel-Moreau dual representation theorem under the $(\varepsilon,\lambda)$--topology}

Theorem 4.5 below is the main result of this subsection.

\begin{theorem} Let $(E,\mathcal{P})$ be an $RLC$ module over $R$ with base $(\Omega,\mathcal{F},P)$ and $f:E\rightarrow \bar{L}^0(\mathcal{F})$ a proper $\mathcal{T}_{\varepsilon,\lambda}$--lower semicontinuous $L^0$--convex function. The $f^{\ast\ast}_{\varepsilon,\lambda}=f$. Here, $f^{\ast}_{\varepsilon,\lambda}:E^{\ast}_{\varepsilon,\lambda}\rightarrow\bar{L}^{0}(\mathcal{F})$ is defined by $f^{\ast}_{\varepsilon,\lambda}(g)=\bigvee\{g(x)-f(x)~|~x\in E\}$ for all $g\in E^{\ast}_{\varepsilon,\lambda}$, and $f^{\ast\ast}_{\varepsilon,\lambda}:E\rightarrow\bar{L}^0(\mathcal{F})$ by $f^{\ast\ast}_{\varepsilon,\lambda}(x)=\bigvee\{g(x)-f^{\ast}_{\varepsilon,\lambda}(g)~|~g\in E^{\ast}_{\varepsilon,\lambda}\}$ for all $x\in E$.
\end{theorem}

\begin{proof}\noindent It is complete similar to the first part of proof of Theorem 3.13 only by using Theorem 3.6 in the place of Theorem 3.5, so is omitted. \hfill $\square$
\end{proof}

\begin{remark} From Theorem 4.5 we can derive Theorem 3.13. In fact, since $\mathcal{P}_{cc}$ and $\mathcal{P}$ induce the same $(\varepsilon,\lambda)$--topology $\mathcal{T}_{\varepsilon,\lambda}$ on $E$, Theorem 4.5 holds for $(E,\mathcal{P}_{cc})$. We first consider the proof of Theorem 3.13 for $(E,\mathcal{P}_{cc})$: let $\mathcal{T}^{\prime}_{c}$ be the locally $L^{0}$--convex topology induced by $\mathcal{P}_{cc}$ and $\mathcal{T}_c$ still denote the locally $L^0$--convex topology induced by $\mathcal{P}$, then $f$ must be $\mathcal{T}^{\prime}_{c}$--lower semicontinuous since $f$ is $\mathcal{T}_c$--lower semicontinuous, and hence also $\mathcal{T}_{\varepsilon,\lambda}$--lower semicontinuous by (2) of Theorem 4.4 since both $E$ and $\mathcal{P}_{cc}$ have the countable concatenation property. Now, by Theorem 4.5 $f^{\ast\ast}_{\varepsilon,\lambda}=f$, further by $E^{\ast}_{\varepsilon,\lambda}=H_{cc}(E^{\ast}_c)$ and (2) of Lemma 3.16, $f^{\ast\ast}_{c}=f^{\ast\ast}_{\varepsilon,\lambda}$, so $f^{\ast\ast}_{c}=f$. We should also point out that Theorem 4.5 is more convenient in use since it has the same shape as the classical Fenchel-Moreau dual Theorem!
\end{remark}

In classical convex analysis, people very often need to consider the Fenchel-Moreau dual representation theorem for a not necessarily proper extended real-valued function, where the notion of a closed function is important. Let $(E,\mathcal{T})$ be a locally convex space. $f:E\rightarrow[-\infty,+\infty]$ is closed if either $f\equiv+\infty$, or $f\equiv-\infty$, or $f$ is a proper lower semicontinuous, cf. \cite{ET}. Thus we should also define and study closed functions in the random setting. In fact, D. Filipovi\'{c}, M. Kupper and N. Vogepoth already studied the problem for a special class of $RN$ module $L^{p}_{\mathcal{F}}(\mathcal{E})$ for financial applications. Here, we make use of Theorem 4.5 to give a unified treatment for the problem.

Let $(E,\mathcal{P})$ be an $RLC$ module over $R$ with base $(\Omega,\mathcal{F},P)$ and $f:E\rightarrow \bar{L}^{0}(\mathcal{F})$ an local function. Let us first give the following notation:

$\mathscr{A}=\{A\in \mathcal{F}~|~$there is $x\in E$ such that $\tilde{I}_A f(x)=\tilde{I}_A(-\infty)\}$;

$\mathscr{B}=\{A\in \mathcal{F}~|~\tilde{I}_A f=\tilde{I}_A(+\infty),$ namely $\tilde{I}_A f(x)=\tilde{I}_A(+\infty)$ for all $x\in E\}$;

$MI(f)=esssup(\mathscr{A})$;

$PI(f)=esssup(\mathscr{B})$;

$BP(f)=\Omega\setminus(MI(f)\bigcup PI(f))$;

$\mathscr{D}=\{A\subset BP(f)~|~A\in \mathcal{F}$ is such that there are $D\in\mathcal{F}$ with $D\subset A$ and $x\in E$ satisfying $f(x)<+\infty$ on $D\}$.

Here, $esssup(\mathcal{H})$ denotes the essential supremum of a subfamily $\mathcal{H}$ of $\mathcal{F}$, cf. \cite{FKV,TXG-JFA}. We can think that $MI(f)$ and $PI(f)$ are disjoint.

It is obvious that $\tilde{I}_{PI(f)}f=\tilde{I}_{PI(f)}(+\infty)$ and $f(x)>-\infty$ on $BP(f)$ for all $x\in E$. Since $\mathscr{A}$ and $\mathscr{D}$ are upward directed, one can use the essential supremum theorem to prove Proposition 4.7 below.

\begin{proposition} We have the following statements:

\noindent $(1)$. There are a countable partition $\{A_n,n\in N\}$ of $MI(f)$ to $\mathcal{F}$ and a sequence $\{y_n,n\in N\}$ in $E$ such that $\tilde{I}_{A_n}f(y_n)=\tilde{I}_{A_n}(-\infty)$ for each $n\in N$.

\noindent $(2)$. There are a countable partition $\{D_n,n\in N\}$ of $BP(f)$ to $\mathcal{F}$ and a sequence $\{x_n,n\in N\}$ in $E$ such that $f(x_n)<+\infty$ on $D_n$ for each $n\in N$ (namely, each $\tilde{I}_{D_n}f$ is proper). Further, if, in addition, $P(BP(f))>0$, then each $D_n$ can be chosen such that $P(D_n)>0$.
\end{proposition}

Let us observe that if $E$ has the countable concatenation property then the local property of $f$ can be used to prove: there are $y\in E$ such that $\tilde{I}_{MI(f)}f(y)=\tilde{I}_{MI(f)}(-\infty)$, and $x\in E$ such that $f(x)<+\infty$ on $BP(f)$, namely $\tilde{I}_{BP(f)}f$ is proper.

For each $D\in\mathcal{F}$, let $E^{D}=\tilde{I}_{D}E:=\{\tilde{I}_{D}x~|~x\in E\}$ and $\|\cdot\|^{D}=$ the restriction of $\|\cdot\|$ to $E^{D}$ for each $\|\cdot\|\in\mathcal{P}$. Then $(E^{D},\mathcal{P}^{D})$ can , of course, be regarded as an $RLC$ module over $R$ with base $(D,D\bigcap\mathcal{F},P(\cdot|D))$ if $P(D)>0$, where $\mathcal{P}^{D}=\{\|\cdot\|^{D}~|~\|\cdot\|\in\mathcal{P}\}$. Further, $f_D:E^{D}\rightarrow \tilde{I}_{D}\bar{L}^{0}(\mathcal{F})$ is defined by $f_{D}(\tilde{I}_{D}x)=\tilde{I}_{D}f(\tilde{I}_{D} x)$ for all $x\in E$.

We can now introduce the notion of a closed function. We can assume, without loss of generality, that $P(BP(f))>0$ for the function $f$ in discussion.

\begin{definition} Let $(E,\mathcal{P})$ be an $RLC$ module over $R$ with base $(\Omega,\mathcal{F},P)$, $f:E\rightarrow \bar{L}^{0}(\mathcal{F})$ a local function. Then $f$ is $\mathcal{T}_{\varepsilon,\lambda}$ (resp., $\mathcal{T}_{c}$)-closed if $\tilde{I}_{MI(f)}f=\tilde{I}_{MI(f)}(-\infty)$ and if $f_A$ is
$L^{0}(A\cap\mathcal{F})-$convex and $\mathcal{T}_{\varepsilon,\lambda}$ (resp., $\mathcal{T}_{c}$)-lower semicontinuous for each $A\in\mathcal{F}$ with $A\subset BP(f)$ and $P(A)>0$ such that $f_A$ is proper.
\end{definition}

\begin{remark} First, $A$ in Definition 4.8 universally exists, for example, let $\{D_n,n\in N\}$ be the same as in (2) of Proposition 4.7, then each $f_{D_{n}}$ is proper. Furthermore, if $f$ is a closed function then $f=\tilde{I}_{PI(f)}(+\infty)+\tilde{I}_{MI(f)}(-\infty)+\sum_{n=1}^{\infty}\tilde{I}_{D_n}f$ with each $\tilde{I}_{D_n}f$ (namely $f_{D_n}$) is proper $L^{0}-$convex lower semicontinuous, so our definition of a closed function is not only very similar to the classical definition of a closed function but also more complicated than the latter. By the way, it is easy to see that a closed function must be $L^{0}-$convex. Secondly, the notion of a $\mathcal{T}_c-$closed function in the sense of Definition 4.8 is more general than that introduced in \cite{FKV-appro}: \cite{FKV-appro} only considered the special case when $E=L^{p}_{\cal F}(\cal E)$, in which case $\tilde{I}_{BP(f)}f$ is proper, whereas $\tilde{I}_{BP(f)}f$ is not necessarily proper in our general case and the study of our general case needs a decomposition of $BP(f)$ as in (2) of Proposition 4.7. Besides, \cite{FKV-appro} employed the strongest notion of a $\mathcal{T}_c-$lower semicontinuous function, whereas we employ the weakest one.
\end{remark}

\begin{proposition} Let $(E,\mathcal{P})$ be the same as in Definition 4.8, $\{f_\alpha,\alpha\in\Gamma\}$ a family of $\mathcal{T}_{\varepsilon,\lambda}$ (resp., $\mathcal{T}_{c}$)-closed functions from $(E,\mathcal{P})$ to $\bar{L}^{0}(\mathcal{F})$ and define $f=\bigvee\{f_\alpha:\alpha\in \Gamma\}$ by $f(x)=\bigvee\{f_\alpha(x):\alpha\in \Gamma\}$ for all $x\in E$. Then $f$ is still $\mathcal{T}_{\varepsilon,\lambda}$ (resp., $\mathcal{T}_{c}$)-closed.
\end{proposition}

\begin{proof} It is easy to see that $MI(f)=essinf\{MI(f_{\alpha}),\alpha\in\Gamma\}$, $PI(f)=esssup\{PI(f_{\alpha}),\alpha\in\Gamma\}$ and $\tilde{I}_{MI(f)}f=\tilde{I}_{MI(f)}(-\infty)$. It remains to show that $f_A$ is $L^{0}(A\cap\cal F)-$convex and $\mathcal{T}_{\varepsilon,\lambda}$ (resp. $\mathcal{T}_c$)$-$lower semicontinuous for each $A\in\cal F$ with $A\subset BP(f)$ and $P(A)>0$ such that $f_A$ is proper. We only gives the proof for the $(\varepsilon,\lambda)-$topology since the case for the locally $L^{0}-$convex topology is similar.

Since each $f_{\alpha}$ is $\mathcal{T}_{\varepsilon,\lambda}-$closed, each $f_{\alpha}$ is $L^{0}-$convex, then $f$ is $L^{0}-$convex, so $f_A$ is $L^{0}(A\cap\cal F)-$convex. Further, since $epi(f_A)=\cap_{\alpha\in\Gamma}epi((f_{\alpha})_A)$, we only need to check that each $epi((f_{\alpha})_A)$ is $\mathcal{T}_{\varepsilon,\lambda}-$closed in $\tilde{I}_A(E\times L^{0}(\cal F))$. In fact, for any fixed $\alpha\in\Gamma$, $A$ must be a subset of $(PI(f_{\alpha}))^c$ since $A\subset BP(f)$, so $A=(A\cap BP(f_{\alpha}))\cup(A\cap MI(f_{\alpha}))$. According to the fact that $\tilde{I}_{MI(f_{\alpha})}f_{\alpha}=\tilde{I}_{MI(f_{\alpha})}(-\infty), epi((f_{\alpha})_A)=epi((f_{\alpha})_{A\cap BP(f_{\alpha})})+\tilde{I}_{A\cap MI(f_{\alpha})}(E\times L^{0}(\cal F))$. Since $f_A$ is proper, it is obvious that $(f_{\alpha})_{A\cap BP(f_{\alpha})}$ is also proper, which shows that $epi((f_{\alpha})_{A\cap BP(f_{\alpha})})$ is $\mathcal{T}_{\varepsilon,\lambda}-$closed in $\tilde{I}_{A\cap BP(f_{\alpha})}(E\times L^{0}(\cal F))$ since $f_{\alpha}$ is a $\mathcal{T}_{\varepsilon,\lambda}-$closed function. Again by noting the fact that $A\cap BP(f_{\alpha})$ and $A\cap MI(f_{\alpha})$ are disjoint we have that $epi((f_{\alpha})_A)$ is $\mathcal{T}_{\varepsilon,\lambda}-$closed.\hfill $\square$
\end{proof}

\begin{definition} Let $(E,\mathcal{P})$ and $f$ be the same as in Definition 4.8. The greatest $\mathcal{T}_{\varepsilon,\lambda}$ (resp., $\mathcal{T}_{c}$)-closed function majorized by $f$, denoted by $Cl_{\varepsilon,\lambda}(f)$ (resp., $Cl_c(f)$), is the $\mathcal{T}_{\varepsilon,\lambda}$ (resp., $\mathcal{T}_c$)-closure of $f$.
\end{definition}

\begin{lemma} Let $(E,\mathcal{P})$ and $f$ be the same as in Definition 4.8. If $f$ is $\mathcal{T}_{\varepsilon,\lambda}$--closed, then $f^{\ast\ast}_{\varepsilon,\lambda}=f$.
\end{lemma}

\begin{proof} Since $f$ is $\mathcal{T}_{\varepsilon,\lambda}$--closed, it is obvious that $\tilde{I}_{MI(f)}f^{\ast\ast}_{\varepsilon,\lambda}=\tilde{I}_{MI(f)}f=\tilde{I}_{MI(f)}(-\infty)$ and $\tilde{I}_{PI(f)}f^{\ast\ast}_{\varepsilon,\lambda}=\tilde{I}_{PI(f)}f=\tilde{I}_{PI(f)}(+\infty)$. Let $\{D_n,n\in N\}$ be the same as in (2) of Proposition 4.7 with $P(D_n)>0$ for all $n\in N$,then each $f_{D_n}$ is a proper $L^{0}(D_n\bigcap\mathcal{F})$--convex $\mathcal{T}_{\varepsilon,\lambda}$--lower semicontinuous on $E^{D_n}$. It is also obvious that $\tilde{I}_{D_n}f^{\ast\ast}_{\varepsilon,\lambda}=f^{\ast\ast}_{D_n}=f_{D_n}=\tilde{I}_{D_n} f$ for each $n\in N$ by Theorem 4.5, so $f^{\ast\ast}_{\varepsilon,\lambda}=f$. \hfill $\square$
\end{proof}

\begin{theorem} Let $(E,\mathcal{P})$ and $f$ be the same as in Definition 4.8. Then $f^{\ast\ast}_{\varepsilon,\lambda}=Cl_{\varepsilon,\lambda}(f)$.
\end{theorem}

\begin{proof} It is obvious that $f^{\ast\ast}_{\varepsilon,\lambda}\leq f$ and $f^{\ast\ast}_{\varepsilon,\lambda}$ is $\mathcal{T}_{\varepsilon,\lambda}$--closed, so $f^{\ast\ast}_{\varepsilon,\lambda}\leq Cl_{\varepsilon,\lambda}(f)$. On the other hand, $Cl_{\varepsilon,\lambda}(f)\leq f$, then $Cl_{\varepsilon,\lambda}(f)=(Cl_{\varepsilon,\lambda}(f))^{\ast\ast}_{\varepsilon,\lambda}\leq f^{\ast\ast}_{\varepsilon,\lambda}$ by Lemma 4.12. \hfill $\square$
\end{proof}

\begin{corollary} Let $(E,\mathcal{P})$ be an $RLC$ module over $R$ with base $(\Omega,\mathcal{F},P)$ such that $E$ has the countable concatenation property and $f:E\rightarrow\bar{L}^{0}(\mathcal{F})$ a local function, then $f^{\ast\ast}_{c}=Cl_{c}(f)$.
\end{corollary}

\begin{proof} It is similar to Remark 4.6, so is omitted. \hfill $\square$\end{proof}

Let $(E,\mathcal{P})$ be an $RLC$ module over $R$ with base $(\Omega,\mathcal{F},P)$ such that $E$ has the countable concatenation property and $f:E\rightarrow\bar{L}^{0}(\mathcal{F})$ a local function, then by (2) of Lemma 3.16 one can easily observe that $f^{\ast\ast}_{\varepsilon,\lambda}= f_c^{\ast\ast}$, furthermore if ,in addition, $\mathcal{P}$ has the countable concatenation property, then $E^{\ast}_{\varepsilon,\lambda}=E^{\ast}_c$, so $f^{\ast}_{\varepsilon,\lambda}=f_c^{\ast}$, in which case we always denote $f^{\ast}_{\varepsilon,\lambda}$ or $f_c^{\ast}$ by $f^{\ast}$.

Finally, let us conclude the subsection with the two representation theorems for $L^{p}_{\mathcal{F}}(\mathcal{E})$--condition risk measures, where $L^{p}_{\mathcal{F}}(\mathcal{E})$ are the $RN$ modules as constructed in Example 2.8 ($1\leq p\leq +\infty$). Definition 4.15 below was introduced for the case when $1\leq p<+\infty$, here we also introduce it for the case when $p=+\infty$ since it will be used in this paper.

\begin{definition} Let $1\leq p\leq+\infty$. A proper function $f:L^{p}_{\mathcal{F}}(\mathcal{E})\rightarrow \bar{L}^{0}(\mathcal{F})$ is said to be:

\noindent (1). monotone if $f(x)\leq f(y)$ for all $x,y\in L^{p}_{\mathcal{F}}(\mathcal{E})$ such that $x\geq y$;

\noindent (2). cash invariant if $f(x+y)=f(x)-y$ for all $x\in L^{p}_{\mathcal{F}}(\mathcal{E})$ and $y\in L^{0}(\mathcal{F})$.

\noindent Further, a proper, monotone and cash invariant function $f$ from $L^{p}_{\mathcal{F}}(\mathcal{E})$ to $\bar{L}^{0}(\mathcal{F})$ is called an $L^{p}_{\mathcal{F}}(\mathcal{E})$--conditional risk measure.
\end{definition}

Since $L^{p}_{\mathcal{F}}(\mathcal{E})$ has the countable concatenation property, and when $1\leq p<+\infty$ $(L^{p}_{\mathcal{F}}(\mathcal{E}))^{\ast}\cong L^{q}_{\mathcal{F}}(\mathcal{E})$, where $p$ and $q$ are a pair of H\"{o}lder conjugate numbers. Further, since $(L^{p}_{\mathcal{F}}(\mathcal{E}),|||\cdot|||_p)$ is an $RN$ module, a proper $L^0(\mathcal{F})$--convex function $f$ from $L^{p}_{\mathcal{F}}(\mathcal{E})$ to $\bar{L}^{0}(\mathcal{F})$ is $\mathcal{T}_{\varepsilon,\lambda}$--lower semicontinuous iff it is $\mathcal{T}_{c}$--lower semicontinuous ($1\leq p<+\infty$). In \cite{FKV-appro}, D. Filipovi\'{c}, M. Kupper and N. Vogelpoth essentially used Theorem 3.13 as well as the typical techniques from conditional risk measures to obtain the following representation theorem:

\begin{theorem}($See$ \cite{FKV-appro}.) Let $1\leq p<+\infty$ and $1<q\leq +\infty$ be a pair of H\"{o}lder conjugate numbers and $f:L^{p}_{\mathcal{F}}(\mathcal{E})\rightarrow \bar{L}^{0}(\mathcal{F})$ a $\mathcal{T}_c$ (or equivalently, $\mathcal{T}_{\varepsilon,\lambda}$)-lower semicontinuous $L^{0}(\mathcal{F})$--convex $L^{p}_{\mathcal{F}}(\mathcal{E})$--conditional risk measure. Then $f(x)=\bigvee\{E(xy|\mathcal{F})-f^{\ast}(y)~|~y\in L^{q}_{\mathcal{F}}(\mathcal{E})$, $y\leq 0$ and $E(y|\mathcal{F})=-1\}$ for all $x\in L^{p}_{\mathcal{F}}(\mathcal{E})$.
\end{theorem}

We consider the natural random duality pair $\langle L^{1}_{\mathcal{F}}(\mathcal{E}),L^{\infty}_{\mathcal{F}}(\mathcal{E})\rangle$, since $L^{1}_{\mathcal{F}}(\mathcal{E})$ and $L^{\infty}_{\mathcal{F}}(\mathcal{E})$ both have the countable concatenation property, a proper $L^{0}(\mathcal{F})$--convex function $f$ from $L^{\infty}_{\mathcal{F}}(\mathcal{E})$ to $\bar{L}^{0}(\mathcal{F})$ is $\sigma_{\varepsilon,\lambda}(L^{\infty}_{\mathcal{F}}(\mathcal{E}),L^{1}_{\mathcal{F}}(\mathcal{E}))$--lower semicontinuous iff it is $\sigma_{c}(L^{\infty}_{\mathcal{F}}(\mathcal{E}),L^{1}_{\mathcal{F}}(\mathcal{E}))$--lower semicontinuous. Further,  $(L^{\infty}_{\mathcal{F}}(\mathcal{E}),
$ $\sigma(L^{\infty}_{\mathcal{F}}(\mathcal{E}),L^{1}_{\mathcal{F}}(\mathcal{E})))^{\ast}_{\varepsilon,\lambda}=
(L^{\infty}_{\mathcal{F}}(\mathcal{E}),\sigma(L^{\infty}_{\mathcal{F}}(\mathcal{E}),L^{1}_{\mathcal{F}}(\mathcal{E})))^{\ast}_{c}= L^{1}_{\mathcal{F}}(\mathcal{E})$. Thus by Theorem 3.13 (or Theorem 4.5) and the similar techniques in \cite{FKV-appro}, we can obtain Theorem 4.17 below.

\begin{theorem} Let $f:L^{\infty}_{\mathcal{F}}(\mathcal{E})\rightarrow \bar{L}^{0}(\mathcal{F})$ be a $\sigma_{\varepsilon,\lambda}(L^{\infty}_{\mathcal{F}}(\mathcal{E}),L^{1}_{\mathcal{F}}(\mathcal{E}))$ (or equivalently, $\sigma_{c}(L^{\infty}_{\mathcal{F}}(\mathcal{E}),L^{1}_{\mathcal{F}}(\mathcal{E}))$)-lower semicontinuous $L^0(\mathcal{F})$--convex $L^{\infty}_{\mathcal{F}}(\mathcal{E})$-- conditional risk measure. Then $f(x)=\bigvee\{E(xy|\mathcal{F})-f^{\ast}(y)~|~ y\in L^{1}_{\mathcal{F}}(\mathcal{E}), y\leq0$ and $E(y|\mathcal{F})=-1\}$ for all $x\in L^{\infty}_{\mathcal{F}}(\mathcal{E})$.
\end{theorem}

\subsection{Extension theorem for $L^{\infty}$--conditional risk measures}

In the sequel of this paper $(\Omega,\mathcal{E},P)$ always denotes a probability space, $\mathcal{F}$ a $\sigma$--subalgebra of $\mathcal{E}$, $L^{\infty}(\mathcal{E}):=L^{\infty}(\Omega,\mathcal{E},P)$ and $L^{\infty}(\mathcal{F}):=L^{\infty}(\Omega,\mathcal{F},P)$.

\begin{definition} (See \cite{Nadal,DS}.) A function $f:L^{\infty}(\mathcal{E})\rightarrow L^{\infty}(\mathcal{F})$ is said to be:

\noindent (1). monotone if $f(x)\leq f(y)$ for all $x,y\in L^{\infty}(\mathcal{E})$ such that $x\geq y$;

\noindent (2). cash invariant if $f(x+y)=f(x)-y$ for all $x\in L^{\infty}(\mathcal{E})$ and $y\in L^{\infty}(\mathcal{F})$.

\noindent Furthermore, a monotone and cash invariant function from $L^{\infty}(\mathcal{E})$ to $L^{\infty}(\mathcal{F})$ is called an $L^{\infty}$--conditional risk measure.
\end{definition}

Let $\mathcal{P}$ be the set of all the probability measures $Q$ on $(\Omega,\mathcal{E})$ such that $Q$ is absolutely continuous with respect to $P$ on $\mathcal{E}$ and $\mathcal{P}_{\mathcal{F}}=\{Q\in \mathcal{P}~|~Q=P$ on $\mathcal{F}\}$.

For an $L^{\infty}$--conditional risk measure $f:L^{\infty}(\mathcal{E})\rightarrow L^{\infty}(\mathcal{F})$, $\alpha:\mathcal{P}_{\mathcal{F}}\rightarrow \bar{L}^{0}(\mathcal{F})$ defined by $\alpha(Q)=\bigvee\{E_{Q}(-x|\mathcal{F})-f(x):x\in L^{\infty}(\mathcal{E})\}$ for all $Q\in \mathcal{P}_{\mathcal{F}}$, is called the random penalty function of $f$, where $E_{Q}(\cdot|\mathcal{F})$ denotes the conditional expectation given $\mathcal{F}$ under $Q$. The following representation theorem is known.

\begin{theorem} ($See$ \cite{Nadal,DS}.) Let $f:L^{\infty}(\mathcal{E})\rightarrow L^{\infty}(\mathcal{F})$ be an $L^{0}(\mathcal{F})$--convex $L^{\infty}$--conditional risk measure. Then the following statements are equivalent:

\noindent (1). $f$ is continuous from above, namely $f(x_n)\nearrow f(x)$ whenever $x_n\searrow x$;

\noindent (2). $f$ has the ``Fatou property": for any bounded sequence $\{x_n,n\in N\}$ which converges P-a.s. to some $x$, then $f(x)\leq \underline{lim}_{n}f(x_n)$;

\noindent (3). $f(x)=\bigvee\{E_{Q}(-x|\mathcal{F})-\alpha(Q)~|~Q\in\mathcal{P}_{\mathcal{F}}\}$ for all $x\in L^{\infty}(\mathcal{E})$.
\end{theorem}

For an $L^{0}(\mathcal(F))$--convex $L^{\infty}$--conditional risk measure $f:L^{\infty}(\mathcal{E})\rightarrow L^{\infty}(\mathcal{F})$, $f^{\ast}:L^{1}_{\mathcal{F}}(\mathcal{E})\rightarrow \bar{L}^{0}(\mathcal{F})$ defined by $f^{\ast}(y)=\bigvee\{E(xy|\mathcal{F})-f(x)~|~x\in L^{\infty}(\mathcal{E})\}$ for all $y\in L^{1}_{\mathcal{F}}(\mathcal{E})$, is called the random conjugate function of $f$, where $E(\cdot|\mathcal{F})$ denotes the conditional expectation given $\mathcal{F}$ under $P$.

By identifying $\mathcal{P}_{\mathcal{F}}$ with $\{y~|~y\in L^{1}_{+}(\mathcal{E}),E(y|\mathcal{F})=1\}$, then (3) of Theorem 4.19 amounts to the following (4).

(4). $f(x)=\bigvee\{E(xy|\mathcal{F})-f^{\ast}(y)~|~y\in L^{1}(\mathcal{E}),y\leq 0$ and $E(y|\mathcal{F})=-1\}$.

\begin{theorem} (\cite{TXG-reflesive,TXG-SBL}.) Let $(E,\|\cdot\|)$ be an $RN$ module over $K$ with base $(\Omega,\mathcal{F},P)$ and $1\leq p\leq +\infty$. Let $L^{p}(E)=\{x\in E~|~\|x\|_{p}<+\infty\}$, where $\|\cdot\|_{p}:E\rightarrow[0,+\infty]$ is defined by

\[ \|x\|_{p}=\left\{
   \begin{array}{ll}
   (\int_{\Omega} \|x\|^{p}dP)^{\frac{1}{p}},    &\mbox{when $1\leq p<+\infty$};\\
    inf\{M\in[0,+\infty]~|~\|x\|\leq M\},  &\mbox{when $p=+\infty$,}
   \end{array}
   \right.
\]\\
for all $x\in E$.

\noindent Then $(L^{p}(E),\|\cdot\|_{p})$ is a normed space and $L^{p}(E)$ is $\mathcal{T}_{\varepsilon,\lambda}$--dense in $E$.
\end{theorem}

\begin{theorem} Let $f:L^{\infty}(\mathcal{E})\rightarrow L^{\infty}(\mathcal{F})$ be an $L^{\infty}$--conditional risk measure. Then there is a unique $L^{\infty}_{\mathcal{F}}(\mathcal{E})$--conditional risk measure $\bar{f}:L^{\infty}_{\mathcal{F}}(\mathcal{E})\rightarrow L^0(\cal{F})$ such that $|\bar{f}(x)-\bar{f}(y)|\leq |||x-y|||_{\infty}$ for all $x,y\in L^{\infty}_{\mathcal{F}}(\mathcal{E})$ and $\bar{f}|_{L^{\infty}(\mathcal{E})}=f$.
\end{theorem}

\begin{proof} Let us first recall the $L^{0}$--norm $|||\cdot|||_{\infty}:L^{\infty}_{\mathcal{F}}(\mathcal{E})\rightarrow L^0_+(\cal{F})$, which is defined by $|||x|||_{\infty}=\bigwedge\{\xi\in \bar{L}^{0}_{+}(\mathcal{F})~|~|x|\leq\xi\}$ for all $x\in L^{\infty}_{\mathcal{F}}(\mathcal{E})$, then it is obvious that $|||x|||_{\infty}\in L^{\infty}_{+}(\mathcal{F})$ for all $x\in L^{\infty}(\mathcal{E})$.

Since $x=y+x-y\leq y+|x-y|\leq y+|||x-y|||_{\infty}$ for all $x,y\in L^{\infty}(\mathcal{E})$, $f(x)\geq f(y+|||x-y|||_{\infty})=f(y)-|||x-y|||_{\infty}$, namely $f(y)-f(x)\leq|||x-y|||_{\infty}$, which shows that $|f(y)-f(x)|\leq|||x-y|||_{\infty}$ for all $x,y\in L^{\infty}(\mathcal{E})$.

Since $L^{\infty}(\mathcal{E})=L^{\infty}(L^{\infty}_{\mathcal{F}}(\mathcal{E}))$, $L^{\infty}(\mathcal{E})$ is $\mathcal{T}_{\varepsilon,\lambda}$--dense in $L^{\infty}_{\mathcal{F}}(\mathcal{E})$ by Theorem 4.20. Further, it is clear that $f$ is uniformly $\mathcal{T}_{\varepsilon,\lambda}$--continuous from $(L^{\infty}(\mathcal{E}),|||\cdot|||_{\infty})$ to $(L^{\infty}(\mathcal{F}),|\cdot|)$ (namely $L^{\infty}(\mathcal{E})$ is regarded as a subspace of $(L^{\infty}_{\mathcal{F}}(\mathcal{E}),|||\cdot|||_{\infty})$ and $L^{\infty}({\mathcal{F}})$ as a subspace of $(L^{0}(\mathcal{F}),|\cdot|)$). Thus $f$ has a unique extension $\bar{f}:L^{\infty}_{\mathcal{F}}(\mathcal{E})\rightarrow L^{0}(\mathcal{F})$ such that $|\bar{f}(x)-\bar{f}(y)|\leq|||x-y|||_{\infty}$ for all $x,y\in L^{\infty}_{\mathcal{F}}(\mathcal{E})$. Since $f$ has the $\mathcal{F}$--local property, $\bar{f}$ must have this property. Further, by Lemma 4.22 below every $x\in L^{\infty}_{\mathcal{F}}(\mathcal{E})$ can be expressed as $x=\Sigma_{n=1}^{\infty}\tilde{I}_{A_n}x_n$ for some countable partition $\{A_n,n\in N\}$ of $\Omega$ to $\mathcal{F}$ and some sequence $\{x_n,n\in N\}$ in $L^{\infty}(\mathcal{E})$ and every $y\in L^{0}(\mathcal{F})$ as $y=\Sigma_{n=1}^{\infty}\tilde{I}_{B_n}y_n$ for some countable partition $\{B_n,n\in N\}$ of $\Omega$ to $\cal{F}$ and some sequence $\{y_n,n\in N\}$ in $L^{\infty}(\cal{F})$, where the first series converges in $\mathcal{T}_{\varepsilon,\lambda}$ in $(L^{\infty}(\mathcal{E}),|||\cdot|||_{\infty})$ and the second does in $(L^{0}(\mathcal{F}),|\cdot|)$. Thus one can easily see that $\bar{f}$ is monotone and cash invariant in the sense of Definition 4.15.  \hfill $\square$
\end{proof}

In Theorem 4.21, when $f$ is $L^{0}(\mathcal{F})$--convex, it is clear that $\bar{f}$ is also $L^{0}(\mathcal{F})$--convex.

Lemma 4.22 below is a stronger proposition than the fact that $L^p(\mathcal{E})$ is $\mathcal{T}_{\varepsilon,\lambda}$--dense in $(L^{p}_{\mathcal{F}}(\mathcal{E}),|||\cdot|||_p)$, which will be used in the proofs of Theorems 4.25 and 4.31 below.

\begin{lemma} Let $1\leq p\leq+\infty$, then $L^{p}_{\mathcal{F}}(\mathcal{E})=H_{cc}(L^{p}(\mathcal{E}))$.\end{lemma}

\begin{proof} Since $L^{p}_{\mathcal{F}}(\mathcal{E})$ has the countable concatenation property, $H_{cc}(L^{p}(\mathcal{E}))\subset L^{p}_{\mathcal{F}}(\mathcal{E})$ is obvious. Conversely, since $L^{p}_{\mathcal{F}}(\mathcal{E})=L^{0}(\mathcal{F})\cdot L^{p}(\mathcal{E}):=\{\xi g~|~\xi\in L^{0}(\mathcal{F})$ and $g\in L^{p}(\mathcal{E})\}$, for any $x=\xi g\in L^{p}_{\mathcal{F}}(\mathcal{E})$ for some $\xi\in L^{0}(\mathcal{F})$ and $g\in L^{p}(\mathcal{E})$, let $\xi^{0}$ be any chosen representative of $\xi$ and $A_n=\{\omega\in\Omega~|~n-1\leq|\xi^{0}(\omega)|<n\}$ for each $n\in N$, then it is clear that $\xi=\Sigma_{n=1}^{\infty}\tilde{I}_{A_n}\xi$, and hence $x=\xi g=(\Sigma_{n=1}^{\infty}\tilde{I}_{A_n}\xi)g=\Sigma_{n=1}^{\infty}\tilde{I}_{A_n}(\tilde{I}_{A_n}\xi g)\in H_{cc}(L^{p}(\mathcal{E}))$ by an easy observation that each $\tilde{I}_{A_n}\xi g\in L^{p}(\mathcal{E})$.  \hfill $\square$
\end{proof}

\begin{remark} Let $x\in L^{p}_{\mathcal{F}}(\mathcal{E})$, $\{A_n,n\in N\}$ be a countable partition of $\Omega$ to $\mathcal{F}$ and $\{x_n,n\in N\}$ a sequence in $L^p(\mathcal{E})$ such that $x=\Sigma_{n=1}^{\infty}\tilde{I}_{A_n}x_n$. Since it is obvious that $\Sigma_{n=1}^{\infty}\tilde{I}_{A_n}|||x_n|||_p$ converges in probability measure $P$, $\Sigma_{n=1}^{\infty}\tilde{I}_{A_n}x_n$ both converges in $\mathcal{T}_{\varepsilon,\lambda}$ to $x$ and unconditionally converges in $\mathcal{T}_{\varepsilon,\lambda}$ to $x$ in $(L^{p}_{\mathcal{F}}(\mathcal{E}),|||\cdot|||_p)$.
\end{remark}

Lemma 4.24 below is also crucial for the proofs of Theorems 4.25 and 4.31 below.

\begin{lemma} Let $1\leq q \leq+\infty$ and $\gamma$ be any positive number. Then we have the following statements:

\noindent (1). Let $G_1=\{y\in L^{q}_{\mathcal{F}}(\mathcal{E})~|~y\leq 0$ and $E(y|\mathcal{F})=-1\}$ and $G_2=\{y\in L^{q}(\mathcal{E})~|~y\leq 0$ and $E(y|\mathcal{F})=-1\}$, then $G_1=H_{cc}(G_2)$.

\noindent (2). Let $G_1$ be the same as in (1) above, $G_3=\{y\in L^{q}(\mathcal{E})~|~y\leq 0, E(y|\mathcal{F})=-1$ and $E(|y|^q|\mathcal{F})\in L^{\infty}(\mathcal{F})\}$ and $G_4=\{y\in L^{q}(\mathcal{E})~|~y\leq 0, E(y|\mathcal{F})=-1$ and $E(|y|^q|\mathcal{F})\in L^{\gamma}(\mathcal{F})\}$, then $G_1=H_{cc}(G_3)=H_{cc}(G_4)$.
\end{lemma}

\begin{proof} (1). It is obvious that $G_1$ has the countable concatenation property and $G_2\subset G_1$, so $H_{cc}(G_2)\subset G_1$. Conversely, let $y$ be a given element in $G_1$, $\xi^0$ any chosen representative of $|||y|||_q$, $A_n=\{\omega\in\Omega~|~n-1\leq\xi^0(\omega)<n\}$ and $y_n=\tilde{I}_{A_n}y+\tilde{I}_{A_n^c}(-1)$ for all $n\in N$, then it is clear that each $y_n\in G_2$ and $y=(\Sigma_{n=1}^{\infty}\tilde{I}_{A_n})y=\Sigma_{n=1}^{\infty}\tilde{I}_{A_n}y=\Sigma_{n=1}^{\infty}\tilde{I}_{A_n}y_n$, so $y\in H_{cc}(G_2)$. Thus $G_1=H_{cc}(G_2)$.

(2). It is obvious that $G_3\subset G_4\subset G_1$, so $H_{cc}(G_3)\subset H_{cc}(G_4)\subset G_1$, it remains to prove that $G_1\subset H_{cc}(G_3)$ as follows: let $y\in G_1$, $\xi^0$, $A_n$ and $y_n$ be the same as in the proof of (1) for each $n\in N$, then it is very easy to observe that each $y_n$ also belongs to $G_3$, so $y=\Sigma_{n=1}^{\infty}\tilde{I}_{A_n}y_n\in H_{cc}(G_3)$, which shows that $G_1\subset H_{cc}(G_3)$. \hfill $\square$
\end{proof}

\begin{theorem} Let $f:L^{\infty}(\mathcal{E})\rightarrow L^{\infty}(\mathcal{F})$ be an $L^{0}(\mathcal{F})$--convex $L^{\infty}$--conditional risk measure and $\bar{f}:L^{\infty}_{\mathcal{F}}(\mathcal{E})\rightarrow L^{0}(\mathcal{F})$ the unique extension as determined in Theorem 4.21. Then the following statements are equivalent:

\noindent $(1)$. $f$ is continuous from above;

\noindent $(2)$. $f$ has the Fatou property;

\noindent $(3)$. $f(x)=\bigvee\{E(xy|\mathcal{F})-f^{\ast}(y)~|~y\in L^{1}(\mathcal{E}), y\geq 0$ and $E(y|\mathcal{F})=-1\}$ for all $x\in L^{\infty}(\mathcal{E})$;

\noindent $(4)$. $f(x)=\bigvee\{E(xy|\mathcal{F})-f^{\ast}(y)~|~y\in L^{1}_{\mathcal{F}}(\mathcal{E}),y\leq 0$ and $E(y|\mathcal{F})=-1\}$ for all $x\in L^{\infty}(\mathcal{E})$;

\noindent $(5)$. $\bar{f}(x)=\bigvee\{E(xy|\mathcal{F})-\bar{f}^{\ast}(y)~|~y\in L^{1}_{\mathcal{F}}(\mathcal{E}),y\leq 0$ and $E(y|\mathcal{F})=-1\}$ for all $x\in L^{\infty}_{\mathcal{F}}(\mathcal{E})$;

\noindent $(6)$. $\bar{f}$ is $\sigma_{\varepsilon,\lambda}(L^{\infty}_{\mathcal{F}}(\mathcal{E}),L^{1}_{\mathcal{F}}(\mathcal{E}))$ (or equivalently, $\sigma_{c}(L^{\infty}_{\mathcal{F}}(\mathcal{E}),$ $L^{1}_{\mathcal{F}}(\mathcal{E}))$)-lower semicontinuous.

\end{theorem}

\begin{proof} (1)$\Leftrightarrow$(2)$\Leftrightarrow$(3) is just Theorem 4.19.

(3)$\Leftrightarrow$(4). The equivalence relation is easily seen by applying (1) of Lemma 4.24 for $q=1$ and (2) of Lemma 3.16 since $E(xy|\mathcal{F})-f^{\ast}(y)$ is a local function of $y$ for each fixed $x\in L^{\infty}(\mathcal{E})$.

(5)$\Rightarrow$(4). Before the proof, let us first notice that $f^{\ast}(y)=\bar{f}^{\ast}(y)$ for all $y\in L^{1}_{\mathcal{F}}(\mathcal{E})$: since $f^{\ast}(y)=\bigvee\{E(xy|\mathcal{F})-f(x)~|~x\in L^{\infty}(\mathcal{E})\}$, $\bar
{f}^{\ast}(y)=\bigvee\{E(xy|\mathcal{F})-\bar{f}(x)~|~x\in L^{\infty}_{\mathcal{F}}(\mathcal{E})\}$ and $L^{\infty}_{\mathcal{F}}(\mathcal{E})=H_{cc}(L^{\infty}(\mathcal{E}))$ by Lemma 4.22, then $\bar{f}^{\ast}(y)=\bigvee\{E(xy|\mathcal{F})-f(x)~|~x\in L^{\infty}(\mathcal{E})\}=f^{\ast}(y)$ by noticing that both $E(xy|\mathcal{F})$ and $\bar{f}(x)$ are local functions of $x\in L^{\infty}_{\mathcal{F}}(\mathcal{E})$ for each fixed $y\in L^{1}_{\mathcal{F}}(\mathcal{E})$ and applying (2) of Lemma 3.16 to the local function $g:L^{\infty}_{\mathcal{F}}(\mathcal{E})\rightarrow \bar{L}^{0}(\mathcal{F})$ defined by $g(x)=E(xy|\mathcal{F})-\bar{f}(x)$.

(4)$\Rightarrow$(5). by (4) of Lemma 3.16 it is clear that $\bar{g}:L^{\infty}_{\mathcal{F}}(\mathcal{E})\rightarrow \bar{L}^{0}(\mathcal{F})$ defined by $\bar{g}(x)=\bigvee\{E(xy|\mathcal{F})-\bar{f}^{\ast}(y)~|~y\in L^{1}_{\mathcal{F}}(\mathcal{E}),y\leq 0$ and $E(y|\mathcal{F})=-1\}$ for any $x\in L^{\infty}_{\mathcal{F}}(\mathcal{E})$ is local since $E(xy|\mathcal{F})-\bar{f}^{\ast}(y)$ is a local function of $x$ for each fixed $y$. By applying (3) of Lemma 3.16 to $\bar{f}$ and $\bar{g}$ one can see that $\bar{f}=\bar{g}$ since $L^{\infty}_{\mathcal{F}}(\mathcal{E})= H_{cc}(L^{\infty}(\mathcal{E}))$ and it is just (4) that $\bar{f}(x)=\bar{g}(x)$ for all $x\in L^{\infty}(\mathcal{E})$.

(5)$\Rightarrow$(6) is clear.

(6)$\Rightarrow$(5) is by Theorem 4.17. \hfill $\square$
\end{proof}

\subsection*{4.4 Extension theorem for $L^p$--conditional risk measures}

Motivated by the work in \cite{Zowe,SGP}, D. Filipovi\'{c}, M. Kupper and N. Vogelpoth introduced the following $L^p$--conditional risk measures in \cite{FKV-appro} as follows:

\begin{definition} (See \cite{FKV-appro}.) Let $1\leq r\leq p<+\infty$ and $f$ be a function from $L^p(\mathcal{E})$ to $L^{r}(\mathcal{F})$. $f$ is an $L^p$--conditional risk measure if the following two conditions are satisfied:

\noindent (1). $f(x)\leq f(y)$ for all $x,y\in L^p(\mathcal{E})$ with $x\geq y$;

\noindent (2). $f(x+y)=f(x)-y$ for all $x\in L^p(\mathcal{E})$ and $y\in L^p(\mathcal{F})$.
\end{definition}

The following representation result is known:

\begin{proposition} ($See$ \cite{FKV-appro}.) Let $f$ be an $L^{0}(\mathcal{F})$--convex $L^p$--conditional risk measure from $L^p(\mathcal{E})$ to $L^{r}(\mathcal{F})$ and $1\leq r\leq p<+\infty$. If $f$ is continuous from $(L^p(\mathcal{E}),\|\cdot\|_p)$ to $(L^r(\mathcal{E}),\|\cdot\|_r)$, then $f(x)=\bigvee\{E(xy|\mathcal{}F)-f^{\ast}(y)~|~y\in L^{q}(\mathcal{E}),y\leq 0,E(|y|^{q}|\mathcal{F})\in L^{\frac{r(p-1)}{p-r}}(\mathcal{F})$ and $E(y|\mathcal{F})=-1\}$ for all $x\in L^p(\mathcal{E})$, where $\frac{1}{p}+\frac{1}{q}=1$ and $\frac{r(p-1)}{p-r}=+\infty$ when $p=r$ and $f^{\ast}:L^{q}_{\mathcal{F}}(\mathcal{E})\rightarrow \bar{L}^0(\mathcal{F})$ is defined by $f^{\ast}(y)=\bigvee\{E(xy|\mathcal{F})-f(x)~|~x\in L^p(\mathcal{E})\}$ for all $y\in L^{q}_{\mathcal{F}}(\mathcal{E})$.
\end{proposition}

The aim of this subsection is to give an extension theorem for any $L^0(\mathcal{F})$--convex $L^p$--conditional risk measure, in particular, in this process we also improve Proposition 4.27 in that we can give a sufficient and necessary condition that any $L^0(\mathcal{F})$--convex $L^p$--conditional risk measure can be represented as in Proposition 4.27, in fact, a new and shorter proof of Proposition 4.27 will be also given.

An important special case of $L^0(\mathcal{F})$--convex $L^p$--conditional risk measures is the following conditional risk measure derived from the solution of backward stochastic differential equations (BSDE, for short).

Let $(B_t)_{t\geq 0}$ be a standard $d$--dimensional Brown motion defined on a probability space $(\Omega,\mathcal{F},P)$ and $(\mathcal{F}_t)_{t\geq 0}$ the augmented filtration generated by $(B_t)_{t\geq 0}$.

Let a function $g:R^{+}\times\Omega\times R^d\rightarrow R$, $(t,\omega,z)\rightarrow g(t,\omega,z)$ (briefly,$g(t,z)$) satisfy the following conditions:

\noindent (A). $g$ is Lipschitz in $z$, i.e., there exists a constant $\mu>0$ such that we have $dt\times dP$--a.s., for any $z_0,z_1\in R^d$, $|g(t,z_0)-g(t,z_1)|\leq \mu\|z_0-z_1\|$.

\noindent (B). For all $z\in R^d$, $g(\cdot,z)$ is a predictable process such that for any $T>0$, $E[\int^{T}_{0}g(t,\omega,z)^2 dt]<+\infty$ for any $z\in R^d$.

\noindent (C). $dt\times dP$--a.s., $g(t,0)=0$.

\noindent (D). $g$ is convex in $z$: $\forall \alpha\in[0,1]$, $\forall z_0,z_1\in R^d$, $dt\times dP$--a.s., $g(t,\alpha z_0+(1-\alpha)z_1)\leq \alpha g(t,z_0)+(1-\alpha)g(t,z_1)$.

Then, for any $T>0$, the following BSDE:

\[ \left\{
   \begin{array}{ll}
   -dY_t=g(t,Z_t)dt-Z_tdB_t,    \\
    Y_T=\xi,
   \end{array}
   \right.
\]\\

\noindent where $\xi\in L^{2}(\Omega,\mathcal{F}_T,P)$, has a unique solution $(Y_t,Z_t)_{t\in[0,T]}$ consisting of predictable stochastic processes such that $E[\int^{T}_{0}Y_t^{2}dt]<+\infty$ and $E[\int^{T}_{0}\|Z_t\|^{2}dt]<+\infty$, cf. \cite{DPG,TXG-dual}. Peng defined the conditional $g$--expectation of $\xi$ at time $t$ as $$\mathcal{E}_g(\xi|\mathcal{F}_t):=Y_t.$$

Now, for any fixed $t\in[0,T]$, define $\rho^{g}_{t}(\cdot):L^2(\mathcal{F}_T)\rightarrow L^2(\mathcal{F}_t)$ by $\rho^{g}_{t}(x)=\mathcal{E}_g(-x|\mathcal{F}_t)$ for all $x\in L^2(\mathcal{F}_T)$, then $\rho^{g}_{t}$ is an $L^{0}(\mathcal{F}_t)$--convex $L^2$--conditional risk measure. By Theorem 3.2 of \cite{SGP}, $\rho^{g}_{t}$ is continuous from $(L^2(\mathcal{F}_T),\|\cdot\|_2)$ to $(L^2(\mathcal{F}_t),\|\cdot\|_2)$. Again by Theorem 3.2 of \cite{SGP}, $|\rho^{g}_{t}(x)-\rho^{g}_{t}(y)|\leq c(E[|x-y|^2|\mathcal{F}_t])^{\frac{1}{2}}$ for all $x,y\in L^2(\mathcal{F}_T)$, where $c=e^{8(1+\mu^2)(T-t)}$, namely when $L^2(\mathcal{F}_T)$ is regarded as a subspace of the $RN$ module $(L^2_{\mathcal{F}_t}(\mathcal{F}_T),|||\cdot|||_2)$ and $L^2(\mathcal{F}_t)$ is regarded as a subspace of the $RN$ module $(L^0(\mathcal{F}_t),|\cdot|)$, $\rho^{g}_t$ is Lipschitz with respect to the $L^0$--norms. So, $\rho^{g}_t$ can be uniquely extended to an $L^0(\mathcal{F}_t)$--convex $L^2_{\mathcal{F}_t}(\mathcal{F}_T)$--conditional risk measure $\bar{\rho}^{g}_t$ such that $|\bar{\rho}^{g}_{t}(x)-\bar{\rho}^{g}_{t}(y)|\leq c(E[|x-y|^2|\mathcal{F}_t])^{\frac{1}{2}}$ for all $x,y\in L^2_{\mathcal{F}_t}(\mathcal{F}_T)$ since $L^2(\mathcal{F}_T)$ is $\mathcal{T}_{\varepsilon,\lambda}$--dense in $(L^2_{\mathcal{F}_t}(\mathcal{F}_T),|||\cdot|||_2)$ (note: $L^2(\mathcal{F}_T)=L^2(L^2_{\mathcal{F}_t}(\mathcal{F}_T))$ and use Theorem 4.20), the proof is the same as that of Theorem 4.21.

Since a general $L^0(\mathcal{F})$--convex $L^p$--conditional risk measure from $L^p(\mathcal{E})$ to $L^r(\mathcal{F})$ may not necessarily $\mathcal{T}_{\varepsilon,\lambda}$--uniformly continuous when $L^p(\mathcal{E})$ is regarded a subspace of $(L^p_{\mathcal{F}}(\mathcal{E}),|||\cdot|||_p)$ and $L^{r}(\mathcal{F})$ as a subspace of $(L^0(\mathcal{F}),|\cdot|)$, we are forced to use a constructive way to obtain a unique extension, namely Theorem 4.28 below.

\begin{theorem} Let $f:L^p(\mathcal{E})\rightarrow L^r(\mathcal{F})$ be an $L^0(\mathcal{F})$--convex $L^p$--conditional risk measure. Then there is a unique $L^0(\mathcal{F})$--convex $L^p_{\mathcal{F}}(\mathcal{E})$--conditional risk measure $\bar{f}:L^p_{\mathcal{F}}(\mathcal{E})\rightarrow L^0(\mathcal{F})$ such that $\bar{f}|_{L^{p}(\mathcal{E})}=f$.
\end{theorem}

\begin{proof} Define $\bar{f}:L^p_{\mathcal{F}}(\mathcal{E})\rightarrow L^0(\mathcal{F})$ by $\bar{f}(x)=\Sigma_{n=1}^{\infty}\tilde{I}_{A_n}f(x_n)$ for any canonical representation $\Sigma_{n=1}^{\infty}\tilde{I}_{A_n}x_n$ of $x$,

First, $\bar{f}$ is well defined. In fact, for any two canonical representations $\Sigma_{n=1}^{\infty}\tilde{I}_{A_n}x_n$ and $\Sigma_{n=1}^{\infty}\tilde{I}_{B_n}y_n$ of $x$, $\Sigma_{n=1}^{\infty}\tilde{I}_{A_n}f(x_n)=\Sigma_{i,j=1}^{\infty}\tilde{I}_{A_i\bigcap B_j}f(x_i)=\Sigma_{i,j=1}^{\infty}$ $\tilde{I}_{A_i\bigcap B_j}f(y_j)=\Sigma_{n=1}^{\infty}\tilde{I}_{B_n}f(y_n)$.

Second, $\bar{f}$ is $L^0(\mathcal{F})$--convex: let $\xi\in L^0_+(\mathcal{F})$ such that $0\leq\xi\leq1$ and $x,y\in L^p_{\mathcal{F}}(\mathcal{E})$ have the canonical representations $\Sigma_{n=1}^{\infty}\tilde{I}_{A_n}x_n$ and $\Sigma_{n=1}^{\infty}\tilde{I}_{B_n}y_n$, respectively. Then $x=\Sigma_{i,j=1}^{\infty}\tilde{I}_{A_i\bigcap B_j}x_i$ and $y=\Sigma_{i,j=1}^{\infty}\tilde{I}_{A_i\bigcap B_j}y_j$, and thus $\bar{f}(\xi x+(1-\xi)y)=\bar{f}(\Sigma_{i,j=1}^{\infty}\tilde{I}_{A_i\bigcap B_j}(\xi x_i+(1-\xi)y_j))=\Sigma_{i,j=1}^{\infty}\tilde{I}_{A_i\bigcap B_j}f(\xi x_i+(1-\xi)y_j)\leq \xi\Sigma_{i,j=1}^{\infty}\tilde{I}_{A_i\bigcap B_j}f(x_i)+(1-\xi)\Sigma_{i,j=1}^{\infty}\tilde{I}_{A_i\bigcap B_j}f(y_j)=\xi \bar{f}(x)+(1-\xi)\bar{f}(y)$.

Similarly, $\bar{f}$ is also monotone. Further, by Lemma 4.22 and the local property of $\bar{f}$, $\bar{f}$ is also cash invariant in the sense of Definition 4.15.

Finally, any $L^0(\mathcal{F})$--convex $L^p_{\mathcal{F}}(\mathcal{E})$--conditional risk measure $g$ with $g|_{L^p(\mathcal{E})}$ $=f$ must coincide with $\bar{f}$ since $g$ has the local property by $L^{p}_{\mathcal{F}}(\mathcal{E})=H_{cc}(L^p(\mathcal{E}))$ and applying (3) of Lemma 3.16 to $\bar{f}$ and $g$, which proves the uniqueness. \hfill $\square$
\end{proof}

Theorem 4.31 below shows that the continuity of $f$ in Proposition 4.27 can be weakened to that $\bar{f}$ is $\mathcal{T}_{\varepsilon,\lambda}$ (or equivalently, $\mathcal{T}_{c}$)-lower semicontinuous, whereas the implication of the continuity of $f$ is reflected to some extent by Theorem 4.29 below.

\begin{theorem} Let $f$ and $\bar{f}$ be the same as in Theorem 4.28. If $f$ is continuous from $(L^p(\mathcal{E}),\|\cdot\|_p)$ to $(L^r(\mathcal{F}),\|\cdot\|_r)$, then $\bar{f}$ is $\mathcal{T}_{\varepsilon,\lambda}$--continuous from $(L^p_{\mathcal{F}}(\mathcal{E}),|||\cdot|||_p)$ to $(L^0(\mathcal{F}),|\cdot|)$.
\end{theorem}

\begin{proof} When $L^p(\mathcal{E})$ is regard as a subspace $(L^p_{\mathcal{F}}(\mathcal{E}),|||\cdot|||_p)$ and $(L^r(\mathcal{F}),\|\cdot\|_r)$ is regarded as a subspace of $(L^0(\mathcal{F}),|\cdot|)$, we first prove that $f$ is $\mathcal{T}_{\varepsilon,\lambda}$--continuous from $(L^p(\mathcal{E}),|||\cdot|||_p)$ to $(L^r(\mathcal{F}),|\cdot|)$. For this, we only need to prove that , for each fixed $x_{0}\in L^p(\mathcal{E})$ and each sequence $\{x_n,n\in N\}$ in $L^p(\mathcal{E})$ such that $\{E[|x_n-x_{0}|^p|\mathcal{F}]^{\frac{1}{p}}:n\in N\}$ converges in probability measure $P$ to 0, there exists a subsequence $\{x_{n_k},k\in N\}$ of $\{x_n,n\in N\}$ such that $\{f(x_{n_k}),k\in N\}$ converges in probability measure $P$ to $f(x_{0})$. Since $f$ is $\mathcal{F}$--local, we only need to prove that, for any positive number $\delta$, there exist an $\mathcal{F}$--measurable subset $H_{\delta}$ of $\Omega$ and a subsequence $\{x_{n_k},k\in N\}$ of $\{x_n,n\in N\}$ such that $P(\Omega\setminus H_{\delta})>1-\delta$ and $\{f(x_{n_k}),k\in N\}$ converges in probability measure $P$ to $f(x_{0})$ on $\Omega\setminus H_{\delta}$. In fact, by the Egoroff theorem there are such $H_{\delta}$ and $\{x_{n_k},k\in N\}$ such that $\{E[|x_{n_k}-x_{0}|^{p}|\mathcal{F}]^{\frac{1}{p}},k\in N\}$ converges uniformly to 0 on $\Omega\setminus H_{\delta}$, so that $\{\tilde{I}_{\Omega\setminus H_{\delta}}x_{n_k},k\in N\}$ converges to $\tilde{I}_{\Omega\setminus H_{\delta}}x_{0}$ in the usual $L^p$--norm $\|\cdot\|_p$ by the Lebesgue domination convergence theorem, hence $\{\tilde{I}_{\Omega\setminus H_{\delta}}f(x_{n_k}),k\in N\}$ converges in the $L^r$--norm to $\tilde{I}_{\Omega\setminus H_{\delta}}f(x_{0})$, which implies that $\{f(x_{n_k}),k\in N\}$ converges in probability measure $P$ to $f(x_{0})$ on $\Omega\setminus H_{\delta}$.

We can now prove that $\bar{f}$ is $\mathcal{T}_{\varepsilon,\lambda}$--continuous. Let $\{x_k,k\in N\}$ be a sequence in $L^p_{\mathcal{F}}(\mathcal{E})$ convergent in $\mathcal{T}_{\varepsilon,\lambda}$ to $x\in L^p_{\mathcal{F}}(\mathcal{E})$, where $\mathcal{T}_{\varepsilon,\lambda}$ is the $(\varepsilon,\lambda)$--topology on $L^p_{\mathcal{F}}(\mathcal{E})$ induced by $|||\cdot|||_p$, then for any canonical representation $\Sigma_{n=1}^{\infty}\tilde{I}_{A_n}x_n$ of $x$, we only need to prove that $\{\bar{f}(x_k),k\in N\}$ converges in probability $P$ to $\bar{f}(x)$ on each fixed $A_n$. Now, let $n_0\in N$ be fixed. For each given canonical representation $\Sigma_{n=1}^{\infty}\tilde{I}_{A_n^k}x_n^k$ of $x_k$, we choose $m_k\in N$ such that $P(\Sigma_{n\geq m_k} A_n^k)<\frac{1}{k}$, then $\{\Sigma_{n=1}^{m_k}\tilde{I}_{A_n^k}x_n^k,k\in N\}$ still converges in $\mathcal{T}_{\varepsilon,\lambda}$ to $x$, hence $\{\tilde{I}_{A_{n_0}}\Sigma_{n=1}^{m_k}\tilde{I}_{A_n^k}x_n^k,k\in N\}$, of course, converges in $\mathcal{T}_{\varepsilon,\lambda}$ to $\tilde{I}_{A_{n_0}}x (=\tilde{I}_{A_{n_0}}x_{n_0})$. By what we have proved, $\{f(\tilde{I}_{A_{n_0}}\Sigma_{n=1}^{m_k}\tilde{I}_{A_n^k}x_n^k),k\in N\}$ converges in the probability measure $P$ to $f(\tilde{I}_{A_{n_0}}x_{n_0})$. By the $\mathcal{F}$--local property of $f$, $\tilde{I}_{A_{n_0}}f(\Sigma_{n=1}^{m_k}\tilde{I}_{A_n^k}x_n^k)=\tilde{I}_{A_{n_0}}f(\tilde{I}_{A_{n_0}}\Sigma_{n=1}^{m_k}\tilde{I}_{A_n^k}x_n^k)$ and $\tilde{I}_{A_{n_0}}f(x_{n_0})=\tilde{I}_{A_{n_0}}f(\tilde{I}_{A_{n_0}}x_{n_0})$, then $\{f(\tilde{I}_{A_{n_0}}\Sigma_{n=1}^{m_k}\tilde{I}_{A_n^k}x_n^k),k\in N\}$ converges in the probability measure $P$ to $f(x_{n_0})$ on $A_{n_0}$.

Finally, since $\bar{f}(x_k)=\tilde{I}_{\bigcup_{n=1}^{m_k}A_{n}^{k}}f(\Sigma_{n=1}^{m_k}\tilde{I}_{A_n^k}x_n^k))+\Sigma_{n\geq m_k}^{\infty}\tilde{I}_{A_n^k}f(x_n^k)$ and $\tilde{I}_{A_{n_0}}$ $f(x)=\tilde{I}_{A_{n_0}}\bar{f}(x)=\tilde{I}_{A_{n_0}}\bar{f}(\tilde{I}_{A_{n_0}}x)=\tilde{I}_{A_{n_0}}f(x_{n_0})$, we have that $\{\bar{f}(x_k),k\in N\}$ converges in the probability measure $P$ to $\bar{f}(x)$ on $A_{n_0}$ by noticing that $P(\Sigma_{n\geq m_k}^{\infty}A_n^k)\rightarrow 0$ when $k\rightarrow \infty$. \hfill $\square$
\end{proof}

\begin{lemma} Let $f:L^p({\mathcal{E}})\rightarrow L^r({\mathcal{F}})$ be an $L^0(\mathcal{F})$--convex $L^p$--conditional risk measure and $\bar{f}:L^{p}_{\mathcal{F}}(\mathcal{E})\rightarrow L^0(\mathcal{F})$ its unique extension. Then we have that $f^{\ast}(y)=\bar{f}^{\ast}(y)$ for all $y\in L^{q}_{\mathcal{F}}(\mathcal{E})$.
\end{lemma}

\begin{proof} let us recall: $f^{\ast}(y)=\bigvee\{E(xy|\mathcal{F})-f(x)~|~x\in L^p(\mathcal{E})\}$ and $\bar{f}^{\ast}(y)=\bigvee\{E(xy|\mathcal{F})-\bar{f}(x)~|~x\in L^p_{\mathcal{F}}(\mathcal{E})\}$. By Lemma 4.22, $L^p_{\mathcal{F}}(\mathcal{E})=H_{cc}(L^p(\mathcal{E}))$, which, together with the local property of $g:L^{p}_{\mathcal{F}}(\mathcal{E})\rightarrow \bar{L}^{0}(\mathcal{F})$ defined by $g(x)=E(xy|\mathcal{F})-\bar{f}(x)$ for all $x\in L^{p}_{\mathcal{F}}(\mathcal{E})$ implies the $f^{\ast}(y)=\bar{f}^{\ast}(y)$ for all $y\in L^q_{\mathcal{F}}(\mathcal{E})$ by applying (2) of Lemma 3.16 to the local function $g$ and $G:=L^{p}(\mathcal{E})$. \hfill $\square$
\end{proof}

\begin{theorem} Let $1\leq r\leq p<+\infty$, $q$ be the H\"{o}lder conjugate number of $p$, $f:L^p({\mathcal{E}})\rightarrow L^r({\mathcal{F}})$ an $L^0(\mathcal{F})$--convex $L^p$--conditional risk measure and $\bar{f}:L^{p}_{\mathcal{F}}(\mathcal{E})\rightarrow L^0(\mathcal{F})$ the unique extension of $f$ determined by Theorem 4.28. Then the following statements are equivalent:

\noindent $(1)$. $f(x)=\bigvee\{E(xy|\mathcal{F})-f^{\ast}(y)~|~y\in L^{q}(\mathcal{E}),y\leq 0,E(|y|^{q}|\mathcal{F})\in L^{\infty}(\mathcal{F})$ and $E(y|\mathcal{F})=-1\}$ for all $x\in L^{p}(\mathcal{E})$.

\noindent $(2)$. For any given positive number $\gamma$, $f(x)=\bigvee\{E(xy|\mathcal{F})-f^{\ast}(y)~|~y\in L^{q}(\mathcal{E}),y\leq 0,E(|y|^{q}|\mathcal{F})\in L^{\gamma}(\mathcal{F})$ and $E(y|\mathcal{F})=-1\}$ for all $x\in L^{p}(\mathcal{E})$.

\noindent $(3)$. $f(x)=\bigvee\{E(xy|\mathcal{F})-f^{\ast}(y)~|~y\in L^{q}(\mathcal{E}),y\leq 0,E(|y|^{q}|\mathcal{F})\in L^{\frac{r(p-1)}{p-r}}(\mathcal{F})$ and $E(y|\mathcal{F})=-1\}$ for all $x\in L^{p}(\mathcal{E})$.

\noindent $(4)$. $f(x)=\bigvee\{E(xy|\mathcal{F})-f^{\ast}(y)~|~y\in L^{q}_{\mathcal{F}}(\mathcal{E}),y\leq 0$ and $E(y|\mathcal{F})=-1\}$ for all $x\in L^{p}(\mathcal{E})$.

\noindent $(5)$. $\bar{f}(x)=\bigvee\{E(xy|\mathcal{F})-\bar{f}^{\ast}(y)~|~y\in L^{q}_{\mathcal{F}}(\mathcal{E}),y\leq 0$ and $E(y|\mathcal{F})=-1\}$ for all $x\in L^{p}_{\mathcal{F}}(\mathcal{E})$.

\noindent $(6)$. $\bar{f}$ is $\mathcal{T}_{\varepsilon,\lambda}$ (or equivalently, $\mathcal{T}_{c}$)-lower semicontinuous.

\noindent $(7)$. $\bar{f}$ is continuous from $(L^{p}_{\mathcal{F}}(\mathcal{E}),\mathcal{T}_c)$ to $ (L^{0}(\mathcal{F}),\mathcal{T}_c)$.
\end{theorem}

\begin{proof} For any fixed $x\in L^{p}(\mathcal{E})$, let $g(y)=E(xy|\mathcal{F})-f^{\ast}(y)$ for any $y\in L^{q}_{\mathcal{F}}(\mathcal{E})$, then $g$ has the local property. By applying (2) of Lemma 3.16 and (2) of Lemma 4.24 one can easily see that (1)$\Leftrightarrow$(2)$\Leftrightarrow$(3)$\Leftrightarrow$(4).

By Lemma 4.30, $f^{\ast}=\bar{f}^{\ast}$, so (5)$\Rightarrow$(4) is clear. If (4) is true, let $\bar{g}:L^{p}_{\mathcal{F}}(\mathcal{E})\rightarrow \bar{L}^{0}(\mathcal{F})$ be defined by $\bar{g}(x)=\bigvee\{E(xy|\mathcal{F})-\bar{f}^{\ast}(y)~|~y\in L^{q}_{\mathcal{F}}(\mathcal{E}),y\leq 0$ and $E(y|\mathcal{F})=-1\}$ for any $x\in L^{p}_{\mathcal{F}}(\mathcal{E})$, then by (3) of Lemma 3.16 we have that $\bar{f}=\bar{g}$ since (4) just shows that $\bar{f}(x)=\bar{g}(x)$ for any $x\in L^p(\mathcal{E})$, which shows (4)$\Rightarrow$(5).

(5)$\Rightarrow$(6) is clear.

(6)$\Rightarrow$(5) is by Theorem 4.16.

(6)$\Rightarrow$(7) is by Theorem 3.44.

(7)$\Rightarrow$(6) is clear.  \hfill $\square$
\end{proof}

Theorem 4.31 not only gives a very short proof of D. Filipovi\'{c}, M. Kupper and N. Vogelpoth's Proposition 4.27, whose original proof in \cite{FKV-appro} is somewhat complicated, but also improves Proposition 4.27 in that we give a sufficient and necessary condition for (3) of Theorem 4.31 to hold, namely that $\bar{f}$ is $\mathcal{T}_{\varepsilon,\lambda}$ (or $\mathcal{T}_{c}$)-lower semicontinuous. Besides, (1) of Theorem 4.31 maybe is more interesting.

The whole Section 3 together with Theorems 4.5 and 4.13 and Corollary 4.14 has formed a complete random convex analysis and thus this paper has provided a solid analytic foundation for the module approach to conditional risk measures. Furthermore, Theorems 4.25 and 4.31 not only give the complete relations among three kinds of conditional convex risk measures, which shows that $L^p$--conditional risk measures can be incorporated into $L^{p}_{\mathcal{F}}(\mathcal{E})$--conditional risk measures ($1\leq p\leq+\infty$), but also their proofs provide many useful analytic skills.  Maybe the module approach together with these skills will develop their power in the future study of dynamic risk measures with the model spaces consisting of stochastic processes, this topic just began in \cite{CDK}. Just as classical convex analysis has many other rich applications as well as the application to convex risk measures, cf. \cite{ET}, we may also hope that random convex analysis can be applied in other aspects.



\section*{Acknowledgements}
The first author of this paper thanks Professor Zhenting Hou for invaluable talks and encouragement and Professor Quanhua Xu for kindly providing us the excellent literature \cite{HLR} in February, 2012 to make us know, for the first time, the existence of \cite{HLR} since \cite{HLR} has never been mentioned before in the literature of related fields (e.g. \cite{SS}).


\label{}











\end{document}